\documentclass[a4paper,11pt]{article}

\usepackage{csquotes}

\usepackage{amsthm,stmaryrd}  
\usepackage{amsfonts, mathrsfs,amssymb}
\usepackage{amsmath,bbm}
\usepackage[numbers, sort&compress]{natbib}
\usepackage{graphicx}
\usepackage{vmargin, enumerate}
\usepackage{paralist}
\usepackage{times}
\usepackage{color,soul}
\usepackage{hyperref}
\usepackage{cleveref}

\usepackage[utf8]{inputenc}
\usepackage[labelfont=bf,font=it]{caption} % reference to fig doesn't work

\usepackage[normalem]{ulem}

\usepackage{pgf}
\usepackage{tikz}
\usetikzlibrary{arrows,shapes}

\usepackage{xcolor} 
\usepackage[many]{tcolorbox}
\tcbuselibrary{breakable}

\newtheorem{thm}{Theorem}[section]

\newtheorem{lem}[thm]{Lemma}
\newtheorem{prop}[thm]{Proposition}
\newtheorem{cor}[thm]{Corollary}
\theoremstyle{definition}
\newtheorem{rem}[thm]{Remark}
\newcommand{\Leb}{\operatorname{Leb}}

\newcommand{\cL}{{\mathcal L}}

\newcommand{\cP}{{\mathcal P}}

\newcommand{\cS}{{\mathcal S}}
\newcommand{\cV}{{\mathcal V}}

\newcommand{\sC}{\mathscr C}

\newcommand{\sL}{\mathscr L}

\newcommand{\cF}{\mathcal F}
\newcommand{\cC}{\mathcal C}

\newcommand{\sT}{\mathscr T}

\newcommand{\sM}{\mathscr M}

\newcommand{\BE}{\begin{eqnarray*}}
\newcommand{\EE}{\end{eqnarray*}}
\newcommand{\BEN}{\begin{eqnarray}}
\newcommand{\EEN}{\end{eqnarray}}

\newcommand{\bs}{{\mathbf s}}

\newcommand{\bu}{{\mathbf u}}

\newcommand{\llb}{\llbracket}
\newcommand{\rrb}{\rrbracket}

\newcommand{\Z}{\ensuremath{\mathbb Z}}

\newcommand{\N}{\ensuremath{\mathbb N}}
\newcommand{\C}{\ensuremath{\mathbb{C}}}
\newcommand{\R}{\ensuremath{\mathbb{R}}}

\newcommand{\p}[1]{{\mathbf P}\left(#1\right)}
\newcommand{\pc}[1]{{\mathbf P}(#1)}

\newcommand{\Ec}[1]{\ensuremath{\mathbf{E} [#1]}}

\newcommand{\eqdist}{\ensuremath{\stackrel{d}{=}}}

\newcommand{\E}[1]{\ensuremath{\mathbf{E} \left[#1 \right]}}

\newcommand{\I}[1]{\ensuremath{\mathbf{1}_{ { #1 } }}}

\newcommand{\Skel}{\ensuremath{\operatorname{Skel}}}

\hypersetup{
    bookmarks=true,         % show bookmarks bar?
    unicode=false,          % non-Latin characters in Acrobat’s bookmarks
    pdftoolbar=true,        % show Acrobat’s toolbar?
    pdfmenubar=true,        % show Acrobat’s menu?
    pdffitwindow=true,      % page fit to window when opened
    pdftitle={My title},    % title   
    pdfauthor={Author},     % author
    pdfsubject={Subject},   % subject of the document
    pdfnewwindow=true,      % links in new window
    pdfkeywords={keywords}, % list of keywords
    colorlinks=true,       % false: boxed links; true: colored links
    linkcolor=blue,          % color of internal links
    citecolor=blue,        % color of links to bibliography
    filecolor=blue,      % color of file links
    urlcolor=blue           % color of external links
}

\newcommand{\DimH}{\operatorname{\dim_{\text{\textsc{h}}}}}

\setmarginsrb{2.5cm}{2cm}{2.5cm}{2cm}{.4cm}{4mm}{0cm}{7.5mm}

\def \sur#1#2{\mathrel{\mathop{\kern 0pt#1}\limits^{#2}}}

\usepackage{expdlist}

\newcounter{c}
\def \bir{\begin{itemize}\compact \setcounter{c}{0}}

\def \eir{\end{itemize}\vspace{-2em}~}

\newcounter{d}
\def \bia{\begin{itemize}\compact \setcounter{d}{0}}
\def \eia{\end{itemize}\vspace{-2em}~}

\usepackage{lineno}
\usepackage{todonotes}
\usepackage{lineno}
\usepackage{pbsi}
\usepackage{tocloft}

\definecolor{red-orange}{RGB}{255,40,0}

\usepackage{xcolor} 
\usepackage[many]{tcolorbox}
\tcbuselibrary{breakable}

\usepackage{amssymb}

\newcommand{\sP}{{\mathscr P}}

\newcommand{\sF}{{\mathscr F}}
\newcommand{\sE}{{\mathscr E}}
\newcommand{\bP}{{\mathbf P}}
\newcommand{\bn}{{\mathbf n}}
\newcommand{\bE}{{\mathbf E}}
\newcommand{\bU}{{\mathbf U}}
\newcommand{\bV}{{\mathbf V}}

\newcommand{\bW}{{\mathbf W}}
\newcommand{\cG}{{\mathcal G}}

\newcommand{\cE}{{\mathcal E}}

\newcommand{\cU}{{\mathcal U}}

\newcommand{\dP}{\operatorname{d}_{\text{\textsc{p}}}}
\newcommand{\dH}{\operatorname{d}_{\text{\textsc{h}}}}
\newcommand{\diam}{\operatorname{diam}}
\newcommand{\bgamma}{\boldsymbol{\gamma}}
\newcommand{\bbeta}{\boldsymbol{\beta}}
\newcommand{\skel}{\operatorname{Skel}}
\newcommand{\Br}{\operatorname{Br}}
\newcommand{\CMT}{\operatorname{CMT}}

\newcommand{\black}{\color{black}}
\newcommand{\ssp}{{\!\scalebox{.6}{$R$}}}
\newcommand{\ssm}{{\!\scalebox{.6}{$L$}}}
\newcommand{\tsp}{{\scalebox{.6}{$R$}}}
\newcommand{\tsm}{{\scalebox{.6}{$L$}}}
\newcommand{\exc}{{\textup{e}}}
\newcommand{\subW}{{\scalebox{.6}{$W$}}}
\newcommand{\subX}{{\scalebox{.6}{$X$}}}
\newcommand{\slo}{{{\text{\bsifamily s}}}}
\newcommand{\ju}{{{\text{\bsifamily j}}}}
\newcommand{\zi}{{{\text{\bsifamily r}}}}

\newcommand{\li}{{{\text{\bsifamily l}}}}
\newcommand{\ri}{{{\text{\bsifamily r}}}}
\newcommand{\Merge}{\operatorname{Merge}}

\newcommand{\Span}{\operatorname{Span}}
\newcommand{\Code}{\operatorname{Code}}
\newcommand{\surp}{\operatorname{surp}}

\def \floor#1{\lfloor #1 \rfloor}

\begin{document}
\title{\bf Convex minorant trees associated with Brownian paths and the continuum limit of the minimum spanning tree}

\author{
Nicolas Broutin
\thanks{LPSM, Sorbonne Universit\'e, 4 Place Jussieu, 75005 Paris and Institut Universitaire de France (IUF)}
\and 
Jean-François Marckert
\thanks{Univ.\ Bordeaux, CNRS, Bordeaux INP, LaBRI, UMR 5800, 33400 Talence, France}
} 

\maketitle

\begin{abstract}  
We give an explicit construction of the scaling limit of the minimum spanning tree of the complete graph. The limit object is described using a recursive construction involving the convex minorants of a Brownian motion with parabolic drift (and countably many i.i.d.\ uniform random variables); we call it the Brownian parabolic tree. 

Aside from the new representation, this point of view has multiple consequences. For instance, it permits us to prove that its Hausdorff dimension is almost surely 3. It also intrinsically contains information related to some underlying dynamics: one notable by-product is the construction of a standard metric multiplicative coalescent which couples the scaling limits of random graphs at different points of the critical window in terms of the same simple building blocks. 

The above results actually fit in a more general framework. They result from the introduction of a new family of continuum random trees associated with functions via their convex minorants, that we call convex minorant trees. We initiate the study of these structures in the case of Brownian-like paths. In passing, we prove that the convex minorant tree of a Brownian excursion is a Brownian continuum ranndom tree, and that it provides a coupling between the Aldous--Pitman fragmentation of the Brownian continuum random tree and its representation by Bertoin. 
\end{abstract}

\section{Introduction}

%!TEX root = MST_brownian.tex

\subsection{Main results}

For a connected graph $G=(V,E)$, together with distinct positive weights associated to the edges, the minimum weight spanning tree is the unique connected spanning subgraph of $G$ that minimizes the total sum of the edge weights. 
The classical random model consists in taking the complete graph on $[n]:=\{1,2,\dots, n\}$ and independent and identically distributed (i.i.d.) random weights $w_e$, $e\in \binom{[n]}2$, uniform on $[0,1]$. 
%\JF{je pense qu'il faudrait prendre autre chose comme notation que $W_e$: plus loin, $W$ est un browien, $e$ apparaît comme une excursion, et une exponentielle. Ça fait trop... $w_e$? $u_e$?}
Then let $M_n$ denote the corresponding minimum spanning tree (MST) rooted at $\rho_n=1$. It has been proved by \citet*{AdBrGoMi2013a} that, seen as a metric space, $M_n$ admits a scaling limit in the following sense: Let $d_n$ be the graph distance on $M_n$, 
%meaning that $d_n(u,v)$ is the number of edges on the unique path between $u$ and $v$ in $M_n$; 
let $\mu_n$ be the counting measure on $[n]$. Then, there exists a (non-trivial) compact measured metric space $(\sM,d)$, a point $\rho\in \sM$, and a Borel probability measure $\mu$ on $(\sM,d)$ such that
\begin{equation}\label{eq:scaling_limit_mst}
  (M_n, n^{-1/3} d_n, n^{-1}\mu_n, \rho_n) \xrightarrow[n\to\infty]{} (\sM,d,\mu, \rho)\,  
\end{equation}
in distribution, in the sense of Gromov--Hausdorff--Prokhorov. The main result of this paper is to provide an explicit representation of the measured metric space $(\sM,d,\mu)$ using a Brownian motion, and a countable collection of i.i.d.\ uniform random variables, and to initiate the study of some of its properties and consequences. To do so, we introduce a new general class of tree-like structures constructed from functions in a way that differs from the classical contour function encoding.

% \begin{alert}Switch: 
% \begin{compactenum}
%   \item uniform weights $V_e$, $w_e$ ?
%   \item Excursion/Brownian motion: $\sf e$, Brownian motion $\sf W$ or the less ugly $\textup e$, and  $\textup W$ ? Or we could keep everything with $W$ for Brownian motion, and $W^{\sc e}$ for the excursion ? Or $\mathbf e, \mathbf W$: bold $e$ is ok, but bold $W$ is not.
% \end{compactenum}
% \end{alert}

The study of trees and their encoding has a long history. A prominent example is the now classical encoding of trees from a height or contour function which defines a tree-like metric $d$ from a continuous function using the recursive structure of its level sets. 
%NIC In this setting this explains for instance the connection local times and continuous space branching processes and superprocesses \cite{LeLe1998a,LeGall1991} \JF{Je ne comprends pas trop pourquoi parler de local time ou de superprocesses ici}. 
The representation is intimately related to branching processes and fragmentations related to heights, and thus to the process of local times of the height function \cite{LeLe1998a,LeGall1991,Miermont2003}. 
Notable examples include the Brownian continuum random tree \cite{Aldous1991b} seen as encoded by a Brownian excursion \cite{Aldous1993a,Legall1993}, and Lévy trees \cite{LeLe1998a}. 

Our construction differs radically. The tree will be associated to a continuous function $\omega$ defined on an interval $D\subseteq \R$ using the tree-like structure of the family of greatest convex minorants of the graph of $\omega$ on the intervals $[0,x]$, $x\in D$. Furthermore, while the classical height function encoding provides a metric $d_\omega$ that is continuous on $D^2$, the metrics we construct are discontinuous at every local minimum, and the information contained in the encoding function is greatly shuffled when $\omega$ is irregular. We nonetheless hope to demonstrate that the proposed construction provides a convincing point of view for a number of natural problems involving dynamics, in particular those related to remarkable coalescent and fragmentation processes. 

In the present paper, we focus on the very specific case of Brownian like functions, but the reader will easily be convinced that the procedure should apply more generally to càdlàg functions that have only positive jumps such as spectrally positive Lévy processes which will be studied elsewhere. We will in particular define a \emph{convex minorant tree} $\CMT(\exc,\bU)$ from a Brownian excursion $\exc=(\exc_s)_{s\in[0,1]}$ and an independent family of uniform random variables $\bU=(U_i)_{i\ge 1}$. Formally, $\CMT(\exc,\bU)$ will be 
a compact pointed measure metric space that we initially define together with a metric $d$ on $[0,1]$. 

Defining $\CMT(\exc,\bU)$ is an important building block towards the definition of our main object of interest, where we replace the Brownian excursion $\exc$ by another Brownian-like path. Let $(W_s)_{s\ge 0}$ be a standard (linear) Brownian motion on $\R_+$ and for $\lambda\in \R$, and $s\geq 0$, define the Brownian motion with parabolic drift by 
\BEN\label{eq:BwPd}
X^\lambda_s=W_s - \frac{s^2}2 + \lambda s.\EEN 
We usually write $X:=X^0$ when $\lambda=0$. Our main result is the following: 
\begin{thm}\label{thm:limit_mst_Kn}
%Let $\rho_n=1$ be the root of $M_n$. 
As $n\to\infty$, we have the following convergence in distribution for the Gromov--Hausdorff--Prokohorov topology:
\[(M_n, n^{-1/3} d_n, n^{-1}\mu_n, \rho_n) \xrightarrow[n\to\infty]{} \CMT(X,\bU)\,.\]
\end{thm}
We call $\CMT(X,\bU)$ the Brownian parabolic tree. In particular, the limit appearing in \eqref{eq:scaling_limit_mst} is such~that
\[(\sM,d, \mu, \rho) \eqdist \CMT(X,\bU)\,.\]
Its structure and properties provide a way to make explicit computations. For instance, the Hausdorff dimension of $(\sM,d)$ was still unknown, and we show directly
\begin{thm}\label{thm:compact_dimH}Almost surely, the space $\CMT(X,\bU)$ is compact and has Hausdorff dimension $3$. 
\end{thm} 

$\CMT(\exc, \bU)$ is one of the central objects of this paper, together with its variants. 
The following results can be seen as consequences of Theorem~\ref{thm:limit_mst_Kn}. For a natural number $s\ge 0$, we let $C_n^s$ denote a uniformly random connected graph on $[n]$ with $n-1+s$ edges. Assuming that the edge weights on this component are i.i.d.\ uniform on $[0,1]$, $C_n^s$ possesses an a.s.\ unique minimum spanning tree that we denote by $T_n^s$. It is a consequence of \cite{AdBrGoMi2013a} that, for any $s\ge 0$, the graphs $T_n^s$ considered as metric spaces equipped with the graph distance $d_n^s$ and the counting measure on the nodes $\mu_n^s$ have a limit when suitably rescaled. The following theorem provides an explicit representation of these limits. 
For $s\ge 0$, let $\exc^{(s)}$ be a process on $[0,1]$ whose distribution is characterized by (for all $f:\cC([0,1])\to \R$ bounded continuous) 
\[\Ec{f(\exc^{(s)})} = \frac{\Ec{f(\exc) \cdot (\int_0^1 \exc(u)du)^s}} {\Ec{(\int_0^1 \exc(u) du)^s}}\,,\]
where $\exc$ is a standard normalized Brownian excursion. Let $\rho_n^s=1$ be the root of $T_n^s$.
\begin{thm}\label{thm:limit_mst_surplus}
For any natural number $s\ge 0$, we have the following convergence in distribution for the Gromov--Hausdorff--Prokhorov topology:
\[(T_n^s, n^{-1/2} d_n^s, \mu_n^s, \rho_n^s) \xrightarrow[n\to\infty]{} \CMT(\exc^{(s)}, \bU)\,.\]
In particular, for $s=0$, this implies that $\CMT(\exc, \bU)$ is a Brownian continuum random tree.
\end{thm}
The last claim when $s=0$ follows from simple observations: first $\exc^{(0)}$ is simply a standard Brownian excursion; second $T_n^0=C_n^0$ since the latter is already a tree, which must then be uniform, and it is well-known that such trees converge to the Brownian continuum random tree \cite{Aldous1991b,Aldous1993a,Legall1993}. 

Let us to back to the case of the Brownian motion with parabolic drift $X$. The construction of the convex minorant tree inherently captures some hidden dynamics. The explanation shall come later, and we will for now only present some facts. For $\lambda\in \R$ and $t\ge 0$, let
\[B^\lambda_t:=X^\lambda_t - \underline X^\lambda_t \qquad \textrm{ and }\qquad Z^\lambda=\{s\in \R_+: B^\lambda_s=0\}\,.\]
The process $(Z^\lambda)_{\lambda \in \R}$ is non-increasing for the inclusion, and therefore induces a coalescent of $\R_+$: the intervals of $\R_+\setminus Z^\lambda$ can be a.s.\ indexed in decreasing order of their lengths as $\bgamma^\lambda=(\gamma^\lambda_1, \gamma_2^\lambda, \dots)$. It is known \cite{BrMa2015a,Armendariz2001} that the process of the lengths of the intervals $(|\bgamma^\lambda|)_{\lambda \in \R}$ is the standard multiplicative coalescent constructed by Aldous \cite{Aldous1997}. However, the space $\CMT(X,\bU)$ being constructed as $(\sM,d,\mu,\rho)$ from the completion of a random metric $d$ on $\R_+$, it comes with a canonical injection $\pi:\R_+\to \sM$ that allows to transport $Z^\lambda$ into $\sM$. As a consequence, as $\lambda$ varies, the points of $Z^\lambda$ actually also induce a coalescent/fragmentation of $\CMT(X,\bU)$ in the sense that $\pi(Z^\lambda)$ is a non-increasing set of points in $\sM$. We shall now explore more precisely this process. 

In the construction of $\CMT(X,\bU)$, the entries in $\bU$, which are i.i.d.\ uniform random variables, are assigned to the local minima of $X$. For an interval $I\subseteq \R_+$, let $\bU|_I$ denote the sequence of those entries that are assigned to local minima lying in $I$ (in the same order as in $\bU$). For each $i\ge 1$, let 
\[\tilde e^\lambda_i(s):=B^\lambda(s+\inf \gamma_i^\lambda) \I{0\le s \le |\gamma^\lambda_i|}\,.\]
Let $\mathfrak F^\lambda=(\CMT(\tilde e^\lambda_i, \bU|_{\gamma_i^\lambda}), i\ge 1)$ be the collection of convex minorant trees of the excursions $\tilde e^\lambda_i$, $i\ge 1$. Let now $\cS$ be an independent Poisson point process with intensity a half on $\R_+\times \R_+ \times \R$. There exists a measurable function of $(X, \bU, \cS)$ that yields, for each $\lambda\in \R$, a collection of measured metric spaces $\mathfrak G^\lambda$ obtained from $\mathfrak F^\lambda$ by identifying the points $\pi(x)$ and $\pi(y)$ for each $(x,y,t)\in \cS$ such that $t\le \lambda$ and no point of $Z^t$ lies in the closed interval between $x$ and $y$ (i.e., $x$ and $y$ are in the same interval of the fragmentation at time $t$). Almost surely, there are only finitely many points of $\cS$ satisfying these constraints for each $\lambda\in\R$ and $i\ge 1$. 

We now define some discrete analogs, which are more classical. Let $E^n$ denote $\binom{[n]}{2}$. For each $p\in [0,1]$ write $E^n_p:=\{e\in E^n: w_e\le p\}$ so that the graph $G(n,p)=([n], E^n_p)$ is a classical Erd\H{o}s--Rényi random graph, and the process $(G_n^p)_{p\in [0,1]}$ is non-decreasing (in the sense of inclusion of edge sets). The regime of interest is the one when 
\begin{equation}\label{eq:def_pnlambda}
p=p_n(\lambda):=\frac 1n+\frac{\lambda} {n^{4/3}}.
\end{equation} 
Let $C^{n,\lambda}_i$ be the $i$-th largest connected component of $G(n,p_n(\lambda))$, breaking ties using the minimum label. Let $\mathfrak G^{n,\lambda}=(\mathfrak G^{n,\lambda}_i, i\ge 1)$, where
\[\mathfrak G^{n,\lambda}_i=(C^{n,\lambda}_i,n^{-1/3} d^{n,\lambda}_i, n^{-2/3} \mu^{n,\lambda}_i)\]
denotes the corresponding measured metric space, where $d^{n,\lambda}_i$ is the graph distance, and $\mu_i^{n,\lambda}$ denote the counting measure on (the vertex set of) $C^{n,\lambda}_i$. One may similarly define the minimum spanning forest $\mathfrak F^{n,\lambda}=(\mathfrak F^{n,\lambda}_i, i\ge 1)$, where
\[\mathfrak F^{n,\lambda}_i= (C_i^{n,\lambda}, n^{-1/3} \delta^{n,\lambda}_i, n^{-2/3} \mu^{n,\lambda}_i)\]
and $\delta^{n,\lambda}_i$ is the graph distance on the minimum spanning tree of $C_i^{n,\lambda}$ (constructed from the same collection of weights $(w_e)$).

% \begin{thm}[Continuum Kruskal dynamics]Let $\mathfrak F^\lambda=(\CMT(\tilde e^\lambda_i), i\ge 1)$. Then one has 
% \begin{itemize}
%   \item for each $\lambda\in \R$, $\mathfrak F^\lambda$ is isometric to the collection of metric spaces $\mathfrak M$ minus $Z^\lambda$
%   \item for each $\lambda_1<\lambda_2<\dots <\lambda_k$, $(\mathfrak F^{\lambda_1}, \dots, \mathfrak F^{\lambda_k})$ is the scaling limit of $(\mathfrak F^{n,p_n(\lambda_1)}, \dots, \mathfrak F^{n, p_n(\lambda_k)})$.
% \end{itemize}
% \end{thm}

Then the processes $(\mathfrak F^\lambda)_{\lambda \in \R}$ and $(\mathfrak G^\lambda)_{\lambda\in \R}$ enjoy some continuum Kruskal and Erd\H{o}s--Rényi dynamics reflecting the evolution of $\mathfrak F^{n,\lambda}$ and $\mathfrak G^{n,\lambda}$, respectively, in following sense:
\begin{thm}\label{thm:dynamics_X}
For each $\lambda$, and each $i\ge 1$, $\mathfrak F^\lambda_i$ is isometric to the subet of $\sM$ induced by $\gamma^\lambda_i$. Furthermore we have, for any $k\ge 1$ and $\lambda_1<\lambda_2<\dots<\lambda_k$, jointly
\begin{align*}
  (\mathfrak G^{n,\lambda_1}, \mathfrak G^{n,\lambda_2},\dots, \mathfrak G^{n,\lambda_k}) &\xrightarrow[n\to\infty]{} (\mathfrak G^{\lambda_1}, \mathfrak G^{\lambda_2}, \dots, \mathfrak G^{\lambda_k}) \quad and,\\
  (\mathfrak F^{n,\lambda_1}, \mathfrak F^{n,\lambda_2},\dots, \mathfrak F^{n,\lambda_k}) &\xrightarrow[n\to\infty]{} (\mathfrak F^{\lambda_1}, \mathfrak F^{\lambda_2}, \dots, \mathfrak F^{\lambda_k})\,,
\end{align*}
in distribution, where, in each case, the convergence holds with respect to the product Gromov--Hausdorff--Prokhorov topology on sequences of measured metric spaces. 
\end{thm}

Theorem~\ref{thm:dynamics_X} provides an explicit coupling for the standard metric coalescent dynamics constructed by Rossignol in \cite{Rossignol2017a} (see also \cite{AdBrGoMi2019a}). In particular, this shows that $\CMT(X,\bU)$ is the right object to lift the multiplicative coalescent defined by Aldous to the level of metric spaces (as well as its augmented version \cite{BhBuWa2014a}). We are not interested here in verifying that there is indeed a natural Markov semigroup acting on measured metric spaces that formalizes these dynamics; such a Markov processes is constructed and studied in \cite{AdBrGoMi2019a} (see also \cite{Frilet2021}). 

There is also an analog to Theorem~\ref{thm:dynamics_X} replacing $\CMT(X,\bU)$ by $\CMT(\exc,\bU)$ which is relevant to the additive coalescent. The following notation are intentionally similar to that used previously; it shall always be clear to which case we refer. For each $\lambda\le 0$, let $\exc^\lambda(s)=\exc(s)+\lambda s$, and $\underline \exc^{\lambda}(s)=\inf\{\exc^\lambda(r): 0\le r\le s\}$. Write $Z^\lambda=\{s\in [0,1]: \exc^\lambda(s)=\underline \exc^\lambda(s)\}$. The process $Z^\lambda$ is non-increasing in $\lambda$ and induces a coalescent of $[0,1]$, as $\lambda$ varies in $(-\infty,0]$. Let $(\gamma^\lambda_i)_{i\ge 1}$ denote the sequence of lengths of the intervals of $[0,1]\setminus Z^\lambda$, in decreasing order. It is known since the results of \citet{Bertoin2000a} that the process $t\mapsto (\gamma^{-t}_i)_{i\ge 1}$ for $t\geq 0$ is the fragmentation dual to the standard additive coalescent introduced by \citet{AlPi1998a}. Just as before $\CMT(\exc,\bU)$ is a random measured real tree $\mathfrak T=(\sT,d, \mu)$ defined through the completion of a random metric $d$ on $[0,1]$, and we let $\pi:[0,1]\to \sT$ denote the canonical injection. 
%\JF{$\sT$ est défini implicitement? est-ce que tu veux dire,  $\CMT(e,\bU)$ is a (mesured?) random tree  $(\sT,d, \mu)$}
 This allows one to transport $Z^\lambda$ in $\sT$ and therefore, to see $\sT\setminus \pi(Z^{-t})$, $t\ge 0$, as a fragmentation of $\CMT(\exc,\bU)$. 

Formally, for each $\lambda\le 0$ and $i\ge 1$, let
\[\exc^\lambda_i(s):=(\exc^\lambda(\inf \gamma^\lambda_i+s)-\underline \exc^\lambda(s))\I{0\le s\le |\gamma^\lambda_i|}\,.\]
Let $\mathfrak T^\lambda_i:=\CMT(\exc_i, \bU^\lambda_i)=(\sT^\lambda_i, d_i^\lambda, \mu_i^\lambda)$. The following theorem provides an explicit coupling between the representations of the fragmentation that is dual to the additive coalescent due to Aldous \& Pitman on the one hand \cite{AlPi1998a}, and to Bertoin \cite{Bertoin2000a} on the other. It also provides another point of view on some recent results of \citet{KoTh2023a}. Define $\cP:=\{(\pi(x), - \lambda): x\in Z^{\lambda-}\setminus Z^\lambda, \pi(x)\in \Skel(\sT), x\in [0,1], \lambda\le 0\}$, where $\Skel(\sT)$ is the skeleton of $\sT$ that we define here as the set of points $u\in \sT$ such that $\sT\setminus \{u\}$ has at least two connected components. 
 
\begin{thm}[Aldous--Pitman vs Bertoin]\label{thm:additive_coalescent}
Let $\exc$ be a normalized Brownian excursion and recall that $\mathfrak T=(\sT,d,\mu)=\CMT(\exc,\bU)$.
Almost surely, for all $\lambda\leq 0$ and all $i\ge 1$, $\mathfrak T_i^\lambda$ is isometric to the subtree of $\sT$ induced by $\gamma^\lambda_i$. Furthermore
\begin{compactenum}[i)]
  \item for each $\lambda\le 0$ and each $i\ge 0$, $|\gamma^\lambda_i|=\mu(\sT^\lambda_i)$,
  \item conditionally on $(\sT,d)$, $\cP$ is a Poisson point process of unit intensity on $\Skel(\sT)\times \R_+$.
\end{compactenum}
Hence, the process $(\bgamma^{-t})_{t\ge 0}$ is precisely the Aldous--Pitman fragmentation of the Brownian CRT $\mathfrak T$. 
\end{thm}

% \begin{alert}Just a tiny problem: $Z^\lambda$ is not countable, so $\cP$ is not clearly a Poisson point process (the length measure is sigma-finite...). But the reason why $Z^\lambda$ is not countable is simply because of the structure of $\R_+$. We should extract the "relevant" points from $Z^\lambda$ first, these are the ones that are really separating two intervals. It seems that taking the points $x$ which are in $Z^{\lambda-}$ but not in the closure of $Z^\lambda$ should work
% \end{alert}

% \begin{alert}Find good notations to distinguish the cases $\CMT(X^0,\bU)$ and $\CMT(e,\bU)$. So far, we use the same notation for zero set and excursion lengths.

% Attention to the definition of $\cP$: the set $Z$ is not countable, need to remove the points that are only there because of the struture on $\R_+$. 
% \end{alert}

% We conjecture that the construction maybe adapted naturally (distances in $\CMT(e,\bU)$ and $\CMT(X^0,\bU)$ are specific to the Brownian regime) to prove that 
% \begin{itemize}
%   \item the scaling limits of the minimum spanning trees of the multiplicative random graphs may be obtained as $\CMT^(X^*, \bU)$
%   \item inhomogeneous continuum random trees are obtained as $\CMT^*(e^*,\bU)$ for the exchangeable bridges of Bertoin \cite{Bertoin2001a} and Miermont \cite{Miermont2001}
%   \item $\CMT^*$ provides the way to couple the representation of Aldous \& Pitman \cite{AlPi2000a} 
% \end{itemize}

\subsection{Motivation and history of related results} % (fold)
\label{sec:motivation_and_history}

% {\red 
% \begin{itemize}
%   \item history of the MST complete graph, the questions and the CRT
%   \item The construction of \cite{AdBrGoMi2013a} already relies on the connection with random graphs and their scaling limits \cite{AdBrGo2012a}, and thus implicitly on the multiplicative coalescent. 
%   \item Explain the discrete intuition: Kruskal's algorithm and the fact that edges connect random points in the distinct connectec components that are themselves size-biased; in other words, there is an underlying discrete coalescent, and the metric is inductively constructed by connecting the fragments by adding an edge between random points. The idea is to construct a metric space from these dynamics: for each time, have countably many tree-like metric spaces, when two merge construct a metric space on the disjoint union by connecting two independent uniform random points. Doing this would preserve the increasing property (we never mess with the metric by adding to many connections that we then need to remove), but this requires to understand every step of the process and the fact that the times of connections are dense raises some additional difficulty (the idea of adding too many edges and then remove allows to work "in batch" and to understand only what happens in non vanishing time intervals).  
% \end{itemize}
% }

It was already known from the work of \citet*{AdBrGoMi2013a} that $\mathfrak M_n = (M_n,n^{-1/3} d_n, n^{-1}\mu_n, \rho_n)$ converges in distribution. The proof relies on a Cauchy sequence argument for the distribution of $\mathfrak M_n$, and is thus essentially existential. In particular, it does not provide an explicit construction of the limit. The novelty of Theorem~\ref{thm:limit_mst_Kn} lies in the identification of the limit as the convex minorant tree $\CMT(X,\bU)$. Note that, by results of \citet{AdSe2021a}, this is also the scaling limit of random 3-regular graphs.

\medskip
\noindent\textsc{About the scaling limit of $\mathfrak M_n$.}
In order to understand the underlying issues, let us be more specific about the approach used in \cite{AdBrGoMi2013a}. The general idea is to analyse the minimum spanning tree using Kruskal's algorithm \cite{Kruskal1956}. This algorithm proceeds by adding the edges by increasing order of weights to an initially empty graph, provided doing so does not create a cycle. Since in the random setting, the order is uniformly random, the (conditional) distribution according to which the edges are added at each step is straightforward, and the difficulty consists in avoiding the cycles. So one may try to first add edges regardless of whether they create cycles or not, with the hope to be able to deal with that issue later on. One shall do this up to a threshold for the weights that ensures that there are not too many cycles (or dealing with them would be hard), but that the connected components are already fairly large (or we have basically gathered no information). These two competing constraints lead to the choice of keeping only edges with weight at most $p_n(\lambda)=\tfrac 1 n + \lambda n^{-4/3}$ with $\lambda\in \R$ large.  

This $p_n(\lambda)$ happens to be precisely the critical window of the random graphs.  The scaling limit of $G(n,p_n(\lambda))$, seen as the sequence of compact metric spaces $\mathfrak G^{n,p_n(\lambda)}$ is known from the results of \citet*{AdBrGo2012a} who built on the pioneering work of Aldous who had previously obtained the scaling limit for the vector of the sizes of the connected components \cite{Aldous1997}. The analysis in \cite{AdBrGoMi2013a} relies on the fact that, (1) given a connected component of the random graph, one may obtain a tree distributed like its minimum spanning tree by breaking cycles randomly (removing uniformly edges, unless they disconnect the component), and (2) that a similar procedure works on the scaling limit. 
% From there, a significant part of the work consists in controlling what happens as the parameter $\lambda\to\infty$, and in particular to verify that the minimum spanning tree of the largest connected component is a good approximation for $M_n$, even though is total mass is only $O(n^{2/3})$. 
% This part uses refinements of the arguments in \cite{AdBrRe2009} for the diameter of $M_n$. 
This forward/backward procedure provides some geometric information but it is inherently tricky to track it precisely. This explains why it does not lead to an explicit construction of the limit in terms of simple building blocks, or also why the Hausdorff dimension (Proposition~\ref{pro:space_compact}) remained unknown. Furthermore, this approach is fundamentally incapable of providing any result about the behavior at different times, since the cycle breaking procedure removes all cycles. 

\medskip
\noindent\textsc{Related results on the MST.}
Let us mention that Angel and Senizergues are currently finishing a paper in which they study the scaling limit of the local limit of $M_n$, that was described by \citet{Ad2013a}. As the local weak limit of $M_n$ is an infinite tree, their object $\cal M$ is not compact; still $\cal M$ has Hausdorff dimension 3, and it also seems to be the local limit of $\CMT(X,\bU)$. It is our understanding that they also plan to study a ``mesoscopic'' limit that would be an analog to the self-similar CRT of Aldous \cite{Aldous1991} for~$\mathfrak M_n$. 

\medskip
\noindent\textsc{About the scaling limit of random graphs.} It is known that the critical random graphs have a scaling limit \cite{AdBrGo2012a}, which has been constructed for each $\lambda\in \R$ in \cite{AdBrGo2012a} (see also \cite{AdBrGo2010}): for each $i\ge 1$, a connected component is built as the tree with height process $\tilde e^\lambda_i$, in which cycles are created by identifying pairs of points whose locations are given by a Poisson point process under $\tilde e^\lambda_i$. The tree is genuinely different from $\CMT(\tilde e^\lambda_i, \bU|_{\gamma^\lambda_i})$ that we use here. However, the marginals described in Theorem~\ref{thm:dynamics_X} of course correspond. For instance, the number of pairs of points that are identified must have the same distribution conditionally on the excursion. One quickly verifies that (Lemma~\ref{lem:area}), for $\lambda\in \R$ and $i\ge 1$, the average number of pairs given $\tilde e^\lambda_i$ (which also determines all the $\gamma^r_j$ which are subsets of $\gamma^\lambda_i$), is
\[\int_{\gamma^\lambda_i} \tilde e^\lambda_i(s)ds = \frac 1 2 \int_{-\infty}^\lambda \sum_{j\ge 1} |\gamma^{r}_j|^2 \I{\gamma^r_j\subseteq \gamma^\lambda_i} dr\,.\]
The spanning subtree used in \cite{AdBrGo2012a} is discovered by a depth-first search; quite recently, \citet{MiSe2022a} have studied the construction of these scaling limits from a breadth-first exploration. The procedures used in \cite{Aldous1997,AdBrGo2012a,MiSe2022a} are not consistent as $\lambda\in \R$ varies, and the objects obtained for two different values of $\lambda$ have no reason to be close, and do not relate simply to any dynamics.

% \begin{alert}Check the scaling factor for the intensity of the Poisson process; need to match the one in \cite{AdBrGo2010,AdBrGo2012a}, where it is a factor one under the curve when $X^n$ is the Lukasiewiecz walk (or $1/2$ with the height process instead).
% \end{alert}

\medskip
\noindent\textsc{About the limit Erd\H{o}s--Rényi and Kruskal dynamics.} 
Since there is an obvious process version for the entire structure at the discrete level, the question of the dynamics for the limit objects (continuum forests or graphs) is quite natural. 
First it is known from results of \citet{Armendariz2001} and \citet{BrMa2015a} that the process $X=X^0$ defined in \eqref{eq:BwPd} encodes the standard multiplicative coalescent, and thus permits to obtain a coupling of the limit of the sizes of the connected components (see also \cite{MaRa2017a}). A minor modification also yields a coupling of both the sizes and the number of extra edges via an explicit construction of the augmented multiplicative coalescent constructed by \citet*{BhBuWa2014a} (see also the recent point of view by \citet{CoLi2023a,CoLi2023b}). The metrics require the new point of view of the convex minorant tree. We emphasize that what we mean here by dynamics is a process in $\lambda$ whose marginals are the scaling limits for fixed $\lambda$, and that we do not consider the question of the existence of nice Markov semigroup acting on sequences of compact measured metric spaces; this question is addressed in \cite{AdBrGoMi2019a,Rossignol2017a}. 

Let us now say a few words about the case of the (standard) additive coalescent. It was first introduced by Aldous and Pitman \cite{AlPi1998a} as the time reversal of the fragmentation process where a Brownian continuum random tree is split as time goes using a Poisson point process. Bertoin \cite{Bertoin2000a} then observed that one obtains the same fragmentation process by cutting the unit interval at the times where a Brownian excursion plus an increasing linear drift touches its running infimum. These two constructions have been connected in a number of ways at the discrete level, starting with \citet{ChLo2002} who used a representation based on hashing with linear probing \cite{KoWe1966a,Knuth1973b}; the construction of $\CMT(\exc,\bU)$ can be seen as a scaling limit for the tree appearing there. \citet{BrMa2015a} and \citet{MaWa2018a} provide alternative approaches. Quite recently, the two processes have been coupled directly in the continuous by \citet{KoTh2023a}, just as Theorem~\ref{thm:additive_coalescent}. Let us also emphasize the fact that Theorem~\ref{thm:additive_coalescent} is a by-product of the same construction used for Theorem~\ref{thm:dynamics_X}: this shows that the standard additive and multiplicative coalescent are, even when considered at the enriched metric level, very strongly related since they are two versions of the same construction applied two different functions ($\exc$ and $X$, respectively). 

\subsection{Intuition and techniques} % (fold)
\label{sub:intuition_and_techniques}

% \begin{alert}
% 1) In discrete the coalescent giving the minimum spanning tree is simple: random edges

% 2) one may expect that in the limit, the object should be constructed in a similar way. Following the dynamics: from the multiplicative coalescent, each time two components merge, identify two random points. 

% 3) Difficulty, the range is $\R$, and connections happen densely in any time interval

% 4) Will rely on the "linearization" of the coalescent. For instance, the linearization simplifies some aspects, but seriously complicates others. For instance, the distribution of the processes encoding the discrete coalescent are not very nice any more. But once we know the coalescing events, the connections are easy. The components to the right are connected randomly to the left (there is a collection of discrete uniforms here, say $\bU^n$)

% 5) There is a scaling limit where the coalescents are also given by a simple process. 

% 6) It remains to construct something that has the required dynamics (again, the dynamics are easy, but one should start from "something"). The uniforms are the scaling limit of $\bU^n$.

% 7) Prove that we have a construction that makes sense; this involves a novel class of tree-like metric spaces defined from excursions

% 8) Prove that the law is the correct one, and we do this by coupling.
% \end{alert}

We shall now try to convey the main ideas that underlie the construction of our scaling limits. The intuition comes from the discrete setting, and we shall explain why the relevant objects should have continuum analogs, and how these limits could be formally defined. There is no very simple axiomatic definition of the minimum spanning tree, at least none that seems suitable to a direct analysis, and one is lead to track the evolution of a construction algorithm in order to obtain the minimum spanning tree. While the construction in \cite{AdBrGoMi2013a} relies on Kruskal's algorithm which grows a forest \cite{Kruskal1956}, our approach is based on a combination of algorithms by Kruskal and Prim \cite{Prim1957}, which grows a tree containing a given vertex. 
% Compared with the approach in \cite{AdBrGoMi2013a}, rather than adding to many edges and fix the issues afterwards, we shall only add the edges that are really part of the minimum spanning tree. 

Fix $n\ge 1$. Prim's algorithm proceeds as follows. Let $v_1=1$. We define the order of vertices $v_2,v_3,\dots, v_n$ iteratively. For every $j=1,\dots, n$, we let $V_j=\{v_1,\dots, v_j\}$. For $i=2,\dots, n$, let $e_i$ be the edge between $V_{i-1}$ and $[n]\setminus V_{i-1}$ that has the smallest weight. Write $e_i=\{u_i,v_i\}$ with $u_i\in V_{i-1}$ and $v_i\in [n]\setminus V_{i-1}$. Then the minimum spanning tree $M_n$ is the graph on $[n]$ with edge set $\{e_2,\dots, e_n\}$. The order $v_1,v_2,\dots, v_n$ is called the Prim order. It turns out that, for any $p\in [0,1]$, the connected components of $G(n,p)=([n], E_p^n)$, where $E_p^n=\{e\in E^n: w_e\le p\}$ are intervals in the Prim order (that is, the vertex set of each connected components is $\{v_a,v_{a+1},\cdots,v_b\}$ for some $1\leq a\leq b \leq n$). In particular, as $p$ increases, only adjacent intervals may merge. 

Now, consider the graph consisting of edges with weights (strictly) lower than $w_{e_i}$, $\{e: w_e<w_{e_i}\}$, that is just before the edge $e_i$ is added. Let $L_i$ be the connected component containing $v_{i-1}$ in this graph. These are precisely the connected components that merge when $e_i$ is added. For $p\in [0,1]$, let $\cF_p$ denote the sigma-algebra generated by the events $\{w_e\le p, e\in E\}$. The following is straightforward:
\begin{lem}\label{lem:discrete_merges}
For each $2\le i\le n$, conditionally on $\cF_{w_{e_i}-}$, the vertex $u_i$ is uniformly random in $L_i$. 
\end{lem}

In other words, in this discrete representation in which the vertices are placed in the Prim order $v_1,\cdots,v_n$, conditionally on the sequence of intervals that merge, the edges that are part of the minimum spanning tree precisely connect a uniform random vertex in the left interval to the left-most vertex in the right interval. Still in this discrete representation, determining the distribution of the sequence of pairs of intervals that merge together is not quite as easy any longer. Fortunately, in the limit, it is given explicitly by the rather nice process $\R_+\setminus Z^\lambda$.
One might thus hope that, in the limit, one should be able to construct the scaling limit of the minimum spanning tree as follows: for each $\lambda\in \R$, each interval $\gamma^\lambda_i$ should be associated to a continuum random tree, and as $\lambda$ increases, these trees should merge using an analog of the discrete dynamics: each time two intervals merge, a uniformly random point in the left interval and the left-most point of the right one should be identified; the minimum spanning tree should then be the limit as $\lambda\to \infty$ (which would indeed be a tree since $\gamma^\lambda_1\uparrow (0,\infty)$ as $\lambda\to\infty$).

While these dynamics are reasonable, they do not really provide a clear path towards a construction: while at the discrete level, the addition of edges does create some length, identifying points in the limit does not, and it remains to understand from what the length emerges. A natural idea consists in constructing the length using some kind of local time arising from the process $(Z^\lambda)_{\lambda \in \R}$. With this objective in mind, let us go back to the discrete setting.  
For any $i,j\in [n]$, we may find all the nodes on the path between $i$ and $j$ in the minimum spanning tree as follows. For some $p\in [0,1]$, it is convenient to write $i\sim_{n,p} j$ if $i$ and $j$ lie in the same connected component of the graph with edges of weight at most $p$: let $i<j\in [n]$ and let $p(i,j)=\inf\{p: i\sim_{n,p} j\}$. The path between $i$ and $j$ must go through the unique edge $e_k=\{u_k,v_k\}$ with weight $p(i,j)$; then, at time $p(i,j)-$ we are left with two connected components, each containing a pair of points ($u_k$ and $i$ on the one hand, and $v_k$ and $j$ on the other) that should each be connected by a path. Proceeding recursively, the process eventually terminates and yields precisely the collection of nodes which are on the path between $i$ and $j$, and the distance $d_n(i,j)$ is then simply the cardinality of that set (minus one). 

This approach is amenable to an extension to the continuous setting, that we expose here informally. For $x,y\in \R_+$, let $x\sim_\lambda y$ if there is no point of $Z^\lambda$ in the closed interval between $x$ and $y$. Let $I^\lambda(x):=\{y\in \R_+: x\sim_\lambda y\}$. Take now $x<y$ for convenience. Let $\lambda(x,y)=\inf\{\lambda \in \R: x\sim_\lambda y\}$. It turns out that $Z^{\lambda(x,y)}$ almost surely contains a single point in $[x,y]$, that we denote by $\kappa(x,y)$. Then, just before $x$ and $y$ get connected, we have two distinct intervals $I^{\lambda(x,y)}(x)$, and $I^{\lambda(x,y)}(y)$, which are separated by the point $\kappa(x,y)$. The discrete setting suggests that one should choose a uniformly random point $\eta(x,y)$ in $I^{\lambda(x,y)}$ (this is where the uniforms in $\bU$ are used). Then, the two points $\eta(x,y)$ and $\kappa(x,y)$ should be the continuous analog of the extremities of the maximum weight edge on the path between $x$ and $y$. Proceeding recursively by looking for the path between $x$ and $\eta(x,y)$ in $I^\lambda(x,y)(x)$ on the left, and the path between $\kappa(x,y)$ and $y$ in $I^{\lambda(x,y)}(y)$ on the right should yield a random subset of $\R_+$ containing all the points used to go from $x$ to $y$, that should resemble some kind of random Cantor set, and the distance between $x$ and $y$ should be some Hausdorff measure of that set. Our main objective is now to verify that this intuition can be turned into formal definitions,
but also that the objects constructed are indeed the ones we are looking for.

% % subsection overview_of_the_results (end)

\subsection{Organization of the paper} % (fold)
\label{sec:plan_of_the_paper}

The paper is organized as follows. In Section~\ref{sec:recursive_convex_minorants}, we discuss recursive convex minorants, the associated trees and their properties. In particular, it is there that we define the convex minorant trees $\CMT(\exc,\bU)$ and $\CMT(X,\bU)$. In Section~\ref{sec:a_new_point_of_view_on_the_continuum_random_tree}, we prove that the tree  $\CMT(\exc,\bU)$ is a Brownian CRT, and we exhibit the coupling mentioned above between the representations of the fragmentation dual to the additive coalescent by Adous--Pitman \cite{AlPi1998a} on the one hand, and by Bertoin \cite{Bertoin2000a} on the other. In Section~\ref{sec:compactness}, we prove that $\CMT(X,\bU)$ is almost surely compact. In Section~\ref{sec:mass_hausdorff}, we construct the mass measure and use it to lower bound the Hausdorff dimension. Finally, Section~\ref{sec:coupling} is devoted to proving that the Brownian parabolic tree $\CMT(X,\bU)$ is distributed like scaling limit of the minimum spanning tree. 
%{\nicc Auxialiary technical results are found in the Appendix.} %A more detailed table of contents follows. 

% subsection plan_of_the_paper (end)

{\small
\setlength{\cftbeforesecskip}{3pt}
\tableofcontents
}

\section{Notation}
\label{sec:notation}

Let $\bW$ be the Wiener measure on $\cC(\R_+,\R)$, the set of continuous functions $f:\R_+\to \R$; this is the law of standard Brownian motion $(W_t)_{t\ge 0}$ starting at $0$. For a continuous process $\omega = (\omega_t)_{t\ge 0}$, we let $\underline \omega$ and $\overline \omega$ denote respectively the running infimum and supremum processes: $\underline \omega_t :=\inf\{\omega_s : 0\le s\le t\}$ and $\overline \omega_t = \sup\{\omega_t: 0\le s \le t\}$.

 % and $\tilde \bW$ the law of Brownian motion with parabolic drift $X$ given by $X_t=W_t-t^2/2$. 
 % %Let $\cF_t$ denote the filtration generated by $(W_s, s\le t)$, completed to contain all the events of probability $0$. 

Let $\N=\{1,2,\dots\}$ and $\N_0=\N=\cup\{0\}$. Let $\cU=\bigcup_{n\ge 0} \N^n$ be the set of finite words on $\N$. The empty word, denoted by $\varnothing$, is the only element of $\N^0$.
We see the elements of $\cU$ as words on $\N$. For $u\in \N^n$ and $i\in \N$, we let $ui$ denote the element of $\N^{n+1}$ obtained by appending $i$ after $u$, so if $u=(u_1, u_2, \dots, u_k)$, $ui=(u_1,\dots, u_n, i)$. We see $\cU$ as a tree rooted at $\varnothing$, where the natural genealogical order denoted by $\preceq$ is such that we have $u\preceq v$ if $u$ is a prefix of $v$, potentially $u=v$. 
%For $u\in \N^n$ and $v\in \N^m$ we let $uv\in \N^{n+m}$ denote the concatenation of $u$ and $v$.  The children of a node $u\in \cU$ are the $ui$, $i\ge 1$.
Similarly, we let $\cU_2= \bigcup_{n\ge 0} \{0,1\}^n$.

%\section{Stupid questions and stuff to do} % (fold)
%\label{sec:stupid_questions}
 
%\input{questions}
% section stupid_questions (end)
%\section{The geometric point of view}\label{sec:geometry}
%\input{geometry}
%\section{Background}

\section{Recursive convex minorants and their associated trees} % (fold)
\label{sec:recursive_convex_minorants}
%!TEX root = MST_brownian.tex

\subsection{Convex minorants of continuous functions}
\label{sec:convex_minorants_continuous}

Let $D\subseteq \R$ be an interval containing $0$ that will in general be $[0,1]$ or $\R_+$ in the sequel. Let $\cC(D,\R)$ be the set of continuous functions on $D$ equipped with the uniform distance. 
For $\omega\in \cC(D,\R)$ such that $\omega(0)=0$ and $x\in D$, the (greatest) convex minorant of $\omega$ on $[0,x]$ is the maximum convex function $c_x(\cdot, \omega)$ defined on $[0,x]$ such that $c_x(t,\omega)\le \omega(t)$ for all $t\in [0,x]$. We let $\cV_x(\omega)=\{t\in [0,x]: c_x(t,\omega)=\omega(t)\}$, and call the elements of $\cV_x(\omega)\setminus\{x\}$ the vertices of the convex minorant of $\omega$ on $[0,x]$. Observe that $0\in \cV_x(\omega)$. We shall see shortly why the extremity $x$ ought to be treated differently. 

Up to now, the literature has mainly focused on properties of the convex minorant of a function on a fixed domain (for a fixed $x$). Of prime importance to us, is instead the structure of the different convex minorants $c_x(\cdot, \omega)$ of a fixed function $\omega$ as $x\in D$ varies. We start with the following straightforward (deterministic) geometric observation: 

\begin{lem}\label{lem:convex_minorants_intersection}
Let $\omega\in \cC(D,\R)$ be such that $\omega(0)=0$ and let $x,y\in D$ with $0\le y<x$. For any $t\in \cV_x(\omega) \cap [0,y]$, we have $\cV_y(\omega) \cap [0,t] = \cV_x(\omega) \cap [0,t]$.
\end{lem}
In other words, traversing them from the left to the right, the convex minorants $c_x(\cdot, \omega)$ and $c_{y}(\cdot, \omega)$ on $[0,x]$ and $[0,y]$ coincide on a non-empty closed interval, and then split for good. This induces a natural branching structure for $\{\cV_x(\omega), x\in D\}$ that is depicted in Figure~\ref{fig:CM_branching} that is central to the paper. This also justifies that for $t\in \cV_x(\omega)$, the slope of the convex minorant $c_x(\cdot, \omega)$ to the left of $t$, defined by  
\[\slo(t,\omega)=\sup\bigg\{\frac{c_x(t) - c_x(t-s)}{s}: t-s\in \cV_x(\omega) \bigg\}\,,\]
is well-defined intrinsically since for any $x'$ such that $t\in \cV_{x'}(\omega)$ would yield the same value. 
If $t\in \cV_x(\omega)$ for some $x$, let 
\[\zi(t,\omega)=\inf\{s>t: \omega(s) \le \omega(t) +\slo(t,\omega) (s-t)\}\,.\]
We call $\zi(t,\omega)$ the \emph{intercept associated to $t$}; this is defined independently of the choice of $x$ for which $t\in \cV_x(\omega)$. (The notation $\zi(t,\omega)$ comes from ``right''.) The following is clear by construction: 
\begin{lem}\label{lem:intercept}Suppose that $t\in \cV_x(\omega)$ for some $x\in D$. 
\begin{compactitem}[i)]
    \item If $y\in [t,\zi(t,\omega)]$, then $t\in \cV_y(\omega)$.
    \item If $y<t$ or $y>\zi(t,\omega)$ then $t\not\in \cV_y(\omega)$.
\end{compactitem}
\end{lem}
Observe that if $\sL(\omega)$ denotes the set of local minima of $\omega$, then $\slo(t,\omega)$ and $\zi(t,\omega)$ are well-defined for every $t\in \sL(\omega)\setminus \{0\}$. 

\begin{figure}[tb]
  \centering
  \includegraphics[width=.7\textwidth]{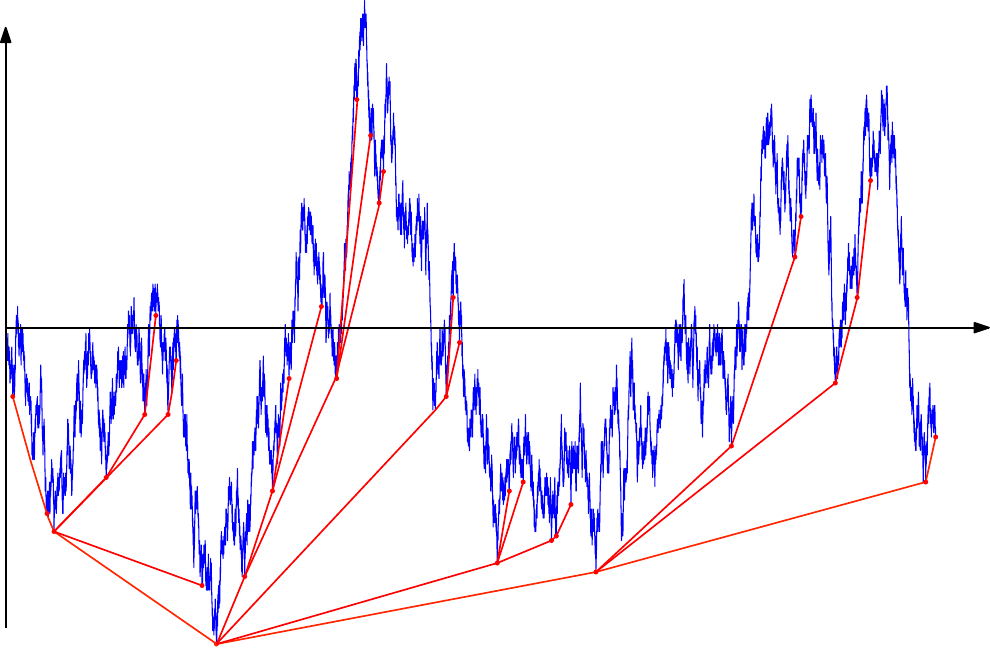}
  \begin{minipage}{.8\linewidth}
  \caption{The tree-like structure of convex minorants: in red, portions of the convex minorants of a fixed given function on finitely many intervals $[0,x]$, for $x\in \R_+$. }
  \label{fig:CM_branching}
  \end{minipage}
\end{figure}

\subsection{Convex minorants of Brownian paths} % (fold)
\label{sec:convex_minorant_of_brownian_paths}

We are interested in convex minorants of various Brownian-like paths such as Brownian motion or the Brownian excursion, the latter being more essential because of classical path decompositions. Such convex minorants have been studied for instance by \citet{Groeneboom1983a} and \citet{Pitman1982a}; in the following, we will mostly rely on the work of \citet{PiRo2011a} that provides means to do explicit calculations; more information about related studies and references can be found there. We therefore now focus on these cases.

Since we should focus on the structure of $\cV_x(\omega)$ as $x$ varies, the following lemma is crucial. 
%{\nic Let $\bW$ be the Wiener measure on $\cC(\R_+,\R)$.}

\begin{lem}[No exceptional point]\label{lem:no_exception_convex}There exists a Borel set $\Omega^\star$ of $\cC(\R_+,\R)$ with $\bW(\Omega^\star)=1$, such that if $\omega\in \Omega^\star$, then for every $x\in \R_+$, it holds that: 
\begin{compactenum}[i)]
  \item $\cV_x(\omega)$ is countable; 
  \item %if $\sup \cV_x < x$ then $\cV_x$ is finite; equivalently, 
  $\cV_x(\omega)$ has no accumulation point in $(0,x)$; 
  \item the elements of $\cV_x(\omega)\setminus \{0,x\}$ are all local minima;
  \item the slopes $\slo(t,\omega)$ at the points $t\in \cV_x(\omega)\setminus \{0,x\}$ are all distinct.
\end{compactenum}
\end{lem}
% \begin{alert}  
% \JF{Question: je ne comprends pas pourquoi il y a un prime, à $c$. Il faut le supprimer, je pense}

% Réponse: C'est bien une pente, il faut une dérivée
% \end{alert}

\begin{rem}Note that a version of Lemma~\ref{lem:no_exception_convex} holds for a Brownian excursion $\exc$ on $[0,1]$ instead of a Brownian motion: in this case, the claims in \emph{i)--iv)} also hold, even with $[0,x)$ instead of $(0,x)$ in \emph{ii)}. 
\end{rem}

Let $(W_t)_{t\ge 0}$ be a Brownian motion on $\R_+$. The results of \cite{PiRo2011a,Groeneboom1983a} imply that for $W$, the set of exceptional points $x$, for which one of the properties in \emph{i)}--\emph{iii)} might fail has Lebesgue measure zero. We verify that with probability one, there is no exceptional point by showing that if there were an exceptional point, then a.s.\ the set of such points would be of positive Lebesgue measure. In the following, we drop the dependence in $W$. 

\begin{proof}[Proof of Lemma~\ref{lem:no_exception_convex}]
\emph{i)} Suppose that, with positive probability, there is some $x$ is such that $\cV_x$ is uncountable. Since there are countably many intervals $[a,b]$ with rational endpoints $0\le a<b<x$, one of them must be such that $\cV_x\cap [a,b]$ is uncountable. By Lemma~\ref{lem:convex_minorants_intersection}, for any point $y\in [b,x]$ we have $\cV_x\cap [a,b] \subseteq \cV_y$, so that the set of exceptional points would have positive Lebesgue measure. 

\emph{ii)} Suppose that, with positive probability, there exists some $x\in \R_+$ such that $\cV_x$ has a an accumulation point $y$ in $(0,x)$. By construction, for any $w\in (y,x]$, we have $\cV_x \cap [0,y] \subseteq \cV_{w} \cap [0,y]$, so that the set of points $w$ for which \emph{ii)} fails has positive Lebesgue measure, a contradition.

\emph{iii)} Finally, suppose that there is some $x\in \R_+$ and $t\in \cV_x\setminus \{0,x\}$ which is not a local minimum. Then, by Lemma~\ref{lem:convex_minorants_intersection}, for any $y\in [t,x]$ we have $t\in \cV_y$, and the proof is complete since $t<x$.

\emph{iv)} If the slopes at $t$ and $t'$ such that $t'<t<x$ are identical, then the same holds for the convex minorants $c_y$ on the intervals $[0,y]$ for every $y\in [t,x]$, so that the set of exceptional points has positive Lebesgue measure.
\end{proof}

The typical situation is that $\cV_x(W)$ has accumulation points at both $0$ and $x$, but it may also happen that $x$ is not an accumulation point: this happens for instance when $x$ is a local minimum or $x=\zi(t,W)$ for some local minimum $t$. 
Let $(t_i)_{i\in \Z}=(t_i(x))_{i\in \Z}$ denote the vertices in $\cV_x(W)\setminus \{x\}$, indexed in such a way that $t_i\le t_{i+1}$ and $t_0=\arg\min\{W_s:s \in [0,x]\}$. In the case where $x$ is not an accumulation point of $\cV_x(W)$, it is understood that the sequence is only defined for $i\le k$ for some $k \ge 0$. The intervals $[t_i,t_{i+1}]$ where the slope of $c_x(W)$ is constant are called the \emph{faces} of the convex minorant. Let $\gamma_i$ denote the slope of the convex minorant on $[t_i,t_{i+1}]$, and $z_i$ be the intercept associated to $t_i$:
\[\gamma_i=\slo(t_{i+1},W)=\frac{W(t_{i+1})-W(t_i)}{t_{i+1}-t_i}
\quad \text{and} \quad
z_i=\zi(t_i,W)=\inf\{s>t_i:W_s\le W_{t_i}+\gamma_{i-1} (s-t_i)\}\,.
\]
It is possible that $z_i=\infty$ for some $i\le 0$, but a.s.\ $z_i<\infty$ for $i\ge 1$. For every $i\ge 1$, such that $z_i<\infty$, and for $s\ge 0$ let 
\begin{equation}\label{eq:decomp_in_excursions}
g_i(s):=(W_{t_i+s}-W_{t_i}-\gamma_i s) \I{t_i+s\le t_{i+1}}
\quad \text{and}\quad 
h_i(s)=(W_{t_i+s}-W_{t_i}-\gamma_{i-1} s ) \I{t_i+s\le z_i}
\,.
\end{equation}
Let $\bn_\sigma$ be the law of a Brownian excursion of duration $\sigma>0$. The following decomposition lemma is straighforward from Theorem~2.2 of \cite{Groeneboom1983a} (see also Theorem~2 and Corollary~2 of \cite{PiRo2011a}):
\begin{lem}\label{lem:W_decomp}For any $i\in \Z$ such that $z_i<\infty$, conditionally on $(t_j,\gamma_j)_{j<i}$, and $(t_i,z_i)$, the collection of functions $g_j$, $j<i$, and $h_i$ form an independent family with law given respectively by $\bn_{t_{j+1}-t_j}$, $j<i$, and $\bn_{z_i-t_i}$.
\end{lem}

Together with the previous considerations about the decomposition, we are thus let to studying convex minorants of Brownian excursions, which is the subject of the next section.

\subsection{Convex minorants of a Brownian excursion} % (fold)
\label{sub:convex_minorant_of_a_brownian_excursion}

\begin{figure}[tb]
  \centering
  \includegraphics[width=.7\textwidth]{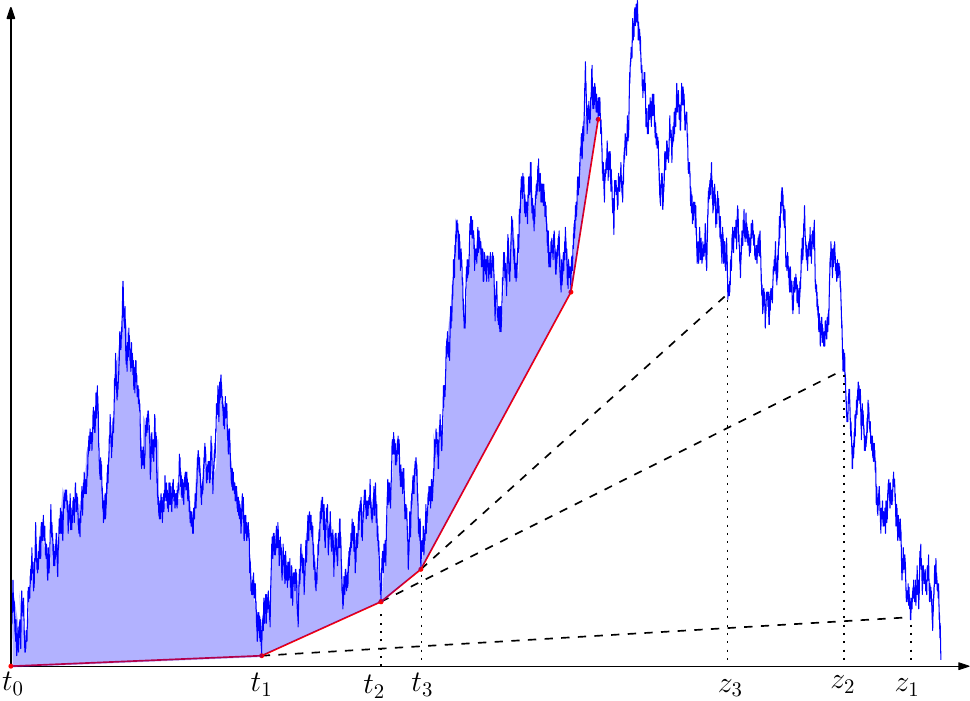}
  \begin{minipage}{.8\linewidth}
  \caption{The greatest convex minorant of a Brownian excursion $\exc$ on some interval $[0,x]$, the sequence of vertices $t_0=0<t_1<\dots < x$ and intercepts $z_1>z_2>\dots > x$. The distribution of these quantities are given in Lemmas~\ref{lem:convex_minorant_law1} and~\ref{lem:convex_minorant_law2} when $x$ is uniform in $[0,1]$.}
  \label{fig:convex_minorant_excursion}
  \end{minipage}
\end{figure}

In this section, we consider a Brownian excursion $\exc$ on $[0,1]$. We use the notation of the previous section with $\omega=\exc$, up to the obvious modifications:
The vertices of the convex minorant of $\exc$ on $[0,x]$ are denoted by $\cV_x=\cV_x(\exc)$, and can be enumerated in increasing order as $(t_i)_{i\ge 0}$ with $t_0=0$, where $t_i=t_i(x)=t_i(x,\exc)$. The slopes $\gamma_i=\gamma_i(x,\exc)$ are defined as before. For each $i\ge 1$, let $z_i=z_i(x)=z_i(x,\exc)=\inf\{s>t_i: \exc_s=\exc_{t_i}+\gamma_{i-1} (s-t_i)\}$. We define $z_0=1$ for convenience. See Figure~\ref{fig:convex_minorant_excursion}. 
 
%The following decomposition of the convex minorant at the first vertex/intercepts, is more essential than the one we have seen before. 
A simple induction yields the description of the restriction of the excursion $\exc$ to the interval $[0,x]$ as a collection of Brownian excursions above the graph of $c_x(\cdot, \exc)$.

\begin{lem}\label{lem:convex_minorant_law1}Let $\exc$ be a standard Brownian excursion on $[0,1]$, and let $x=V$ be an independent random variable uniform on $[0,1]$. Then consider the convex minorant $c_V(\cdot, \exc)$ of $\exc$ on $[0,V]$, with vertices $(t_i)_{i\ge 0}$. Define the functions $\exc_0$ and $\exc_1$ by
\[\exc_0(s):=(\exc(s)-s\cdot \exc(t_1))\I{s\le t_1} 
\qquad \text{and} \qquad 
\exc_1(s):=(\exc(t_1+s)-\exc(t_1)-s\cdot \exc(t_1))\I{t_1+s\le z_1}\,.\]
Then $(t_1,z_1-t_1, 1-z_1)$ is a Dirichlet$(\tfrac 12, \tfrac 12, \tfrac 12)$ random vector,
and conditionally on $(t_1,z_1)$, $(\exc_0,\exc_1, V)$ are independent, $\exc_0$ and $\exc_1$ are Brownian excursions of durations $t_1$ and $z_1-t_1$, respectively, and $V$ is uniform on $(t_1,z_1)$.
\end{lem}
\begin{proof}The claimed properties follow from the decomposition of the Brownian excursion $\exc$ using the line linking $(0,0)$ to  $(t_1,\exc(t_1))$ (notice that $0<t_1<1$ a.s.). For $s\ge 0$, consider the straight line $\{(t,st): t\in [0,1]\}$, and increase the value of $s$ from $0$ until the first value $s=\gamma_1$ at which the location of the first intersection $i(s)=\inf\{t\ge 0: \exc(t)=st\}$ is at most $V$: $\gamma_1=\inf\{s\ge 0: i(s)\le V\}$, $t_1=i(\gamma_1)$. Now, by strong Markov property, $(\exc_{1-t}, z_1\leq t \leq 1)$ is a Brownian meander of duration $1-z_1$ conditioned to end at $\exc_{z_1}$; the path $\exc_0$ coincides with the path above the first face, and is a Brownian excursion; the path $\exc_1$ is the path above the same line, between $t_1$ and $z_1$. Let $\varphi_t(x)=e^{-x^2/(2t)}/\sqrt{t2\pi}$. The vector $(\gamma_1,t_1,z_1,V)$ has a distribution which is absolutely continuous with respect to Lebesgue measure on the set $D:=\R_+\times \{(t,z,u): 0<t<v<z<1\}$, with density $f(s,t,z,v)$ given by 
\begin{align*}
f(s,x,z,v)
&= \I{(s,t,z,v)\in D} \cdot 
       \frac{\varphi_t(ts)}t \cdot  % # 1ere excu
       \frac{\varphi_{z-t}( (z-t) s)}{z-t} \cdot  %# 2eme excu
       \frac{\varphi_{1-z}(z s) zs}{1-z} \cdot  %#   meandre final
         %# masse correspondant à la présence de l'uniforme dans [x,z] ENLEVE POUR avoir la loi de U
t   %# vitesse à laquelle monte la droite penchée à l'abscisse x (quand on augmente de dtheta).
\sqrt{2\pi} %# normalisation
\\
&=\I{(s,t,z,v)\in D}\cdot 
\frac{sze^{- \frac{s^2}{2(1-z)/z}}}{1-z} \times \frac{1} {2\pi\sqrt{t(z-t)(1-z)}} \times \frac 1 {z-t}\,.
\end{align*}  
Integrating for $s\in \R_+$, this yields the claimed distribution.
\end{proof}

\begin{rem}We point out that the distribution of $(t_1,z_1-t_1,1-z_1)$ may also be obtained, without any calculation, using the correspondence with the cut tree and the decomposition of a Brownian continuum random tree into tree pieces that is induced by removing the branch point at the intersection of the geodesics between three random points (see~Section~\ref{sec:dynamics} and \cite{Aldous1994a}). 
\end{rem}

A straightforward induction yields the distribution of the vector of lengths of the faces of the convex minorant $c_V(\cdot, \exc)$ of $\exc$ on $[0,V]$ for an independent uniform point $V$ in $[0,1]$:
\begin{lem}\label{lem:convex_minorant_law2}Let $\exc$ be a standard Brownian excursion on $[0,1]$, and let $x=V$ be an independent random variable uniform on $[0,1]$. Let $(\Delta_{i,1}, \Delta_{i,2}, \Delta_{i,3})_{i\ge 0}$ denote a family of independent Dirichlet$(\tfrac 12, \tfrac 12, \tfrac 12)$ random vectors. Then, for $(t_i)_{i\ge 0}=(t_i(V))_{i\ge 0}$ the sequence of vertices, we have
\[(t_{i}-t_{i-1})_{i\ge 1} \eqdist \bigg(\Delta_{i,1} \cdot \prod_{1\le j<i} \Delta_{j,2}\bigg)_{i\ge 1}\,.\]
\end{lem}

\subsection{Recursive convex minorants of a Brownian excursion} % (fold)
\label{sub:recursive_convex_minorants}

The results of the previous section point out the recursive structure of convex minorants of a Brownian excursion. Here we will use it to construct the tree $\CMT(\exc,\bU)$. This is the first building block of our construction of $\CMT(X,\bU)$ the scaling limit of the minimum spanning tree, and it already reveals some of the main ingredients. Before proceeding to the details, let us explain roughly the strategy: 
\begin{itemize}
  \item for all $x,y\in [0,1]$, we define the set $\llb x,y\rrb$ which is meant to be the collection of points used to go from $x$ to $y$ (somewhat pre-arcs or pre-branches);
  \item we also show that it is possible to assign a ``measure'' $d(x,y)$ to $\llb x,y\rrb$ that induces a $0$-hyperbolic metric space.
\end{itemize}
We will see that the metric space induced by $d$ on $[0,1]$ is connected if we restrict our attention to points at finite distance from $0$, so that the subset of $[0,1]$ with this property, endowed with $d$ is thus an $\R$-tree (in the sense of Section~2.2 of \cite{AdBrGoMi2013a}). Later on, we will show that the metric completion of $([0,1],d)$ is compact, so that no point is put aside. 

The definition of $\llb x,y\rrb$ will be done in stages: first $\llb 0,x\rrb$ with $x$ restricted to some suitable dense subset of $[0,1]$; then, we extend the definition of $\llb 0, x\rrb$ to all $x\in [0,1]$; finally, $\llb x,y\rrb$ is defined in Section~\ref{sec:the_branching_structure} using a notion of common ancestor of $x$ and $y$.

\begin{rem}\label{rem:association_map}
Let $\bn_\sigma$ denote the law of a standard Brownian excursion of duration $\sigma>0$. We will define $\CMT(\exc,\bU)$ as a proper random variable for $\bn_1$-almost every function $\omega$, and almost all sequences $\bU=(U_1,U_2,\dots)$ of independent random variables, uniform on $[0,1]$. For this, the components of $\bU$ are associated to the local minima of $\exc$. This can be done by defining a canonical bijection between $\N$ and the set $\sL(\exc)$ of local minima of $\exc$. For instance, consider an enumeration $I=(I_j,j\geq 0)$ of the (countable) set of all intervals with rational extremities on $[0,1]$. Since a.s.~each local minimum of $\exc$ is a global minimum on at least one interval of $I$, associate with each local minimum $t \in \sL$ the index $j(t)$ of the first interval of $I$ on which $t$ is a global minimum; after that associate with $t$, the uniform random variable $U_{j(t)}$. In the sequel, $j$ is called the association map of $\exc$. The proofs of convergence in Section~\ref{sec:coupling} will need a different, more complex association, but we believe it is not necessary until then. 
\end{rem}

Let $\cU=\bigcup_{n\ge 0} \N^n$ (see Section~\ref{sec:notation}). For any $x\in \sL=\sL(\exc)$, we define recursively a collection $(t_u,\xi_u,\gamma_u, e_u)$, $u\in \cU$, that a priori depends on $x$. Lemma~\ref{lem:no_exception_convex} ensures that, with probability one, the following definition makes sense for all $x\in \sL$.

We first let $t_\varnothing = 0$, $\xi_\varnothing=x$, $\gamma_\varnothing=0$, and $e_\varnothing = \exc$. Almost surely, there are only finitely many vertices of the convex minorant of $\exc$ on $[0,x]$, and they are all elements of $\sL$ and denoted by $t_0=0<t_1< t_2 < \dots < t_k = x$ for some $k\in \N$. For each $i=0,\dots, k-1$, let $\xi_i=t_i + U_{t_{i+1}} |t_{i+1}-t_i|$, $\gamma_i=\slo(t_{i+1}, \exc)=(\exc(t_{i+1})-\exc(t_i))/|t_{i+1}-t_i|$ and let $e_i$ be defined by, for $s\ge 0$,
\[e_i(s)=(\exc(t_i+s)-\exc(t_i)-s\cdot \gamma_i)\I{t_i+s\le t_{i+1}}\,.\]

More generally, suppose now that we have defined $(t_u,\xi_u, \gamma_u, e_u)$ for some $u \in \N^n$, $n\ge 1$. Let $\theta_0^u=0<\theta_1^u<\theta_2^u<\dots$ be the vertices of the convex minorant of $e_u$ on the interval $[0,\xi_u-t_u]$, and set $t_{ui}=t_u+\theta_i^u$ for all $i\ge 0$; observe that the $t_{ui}=t_u+\theta_i^{u}$, $i\ge 0$, are precisely the elements of $(t_j(\xi_u))_{j\ge 0}$ lying in $[t_u, \xi_u]$. Then let $\varphi^u_i=(e_u(\theta^u_{i+1})-e_u(\theta^u_i))/|\theta^u_{i+1}-\theta^u_i|$ be the slope of the convex minorant of $e_u$ on $[\theta^u_i,\theta^u_{i+1}]$. For each $i\in \N$, we let $m_{ui}=|\theta^u_{i+1}-\theta^u_i|=|t_{u(i+1)}-t_{ui}|$, $\xi_{ui}=t_{ui} + U_{t_{u(i+1)}} m_{ui}$, $\gamma_{ui}=(e(t_{u(i+1)})-e(t_{ui}))/m_{ui}=\gamma_u+\varphi_i^u$ and define the function $e_{ui}:[0,m_{ui}]\to \R_+$ by 
\[e_{ui}(s)=(\exc(t_{ui}+s)-\exc(t_{ui})-s\cdot \gamma_{ui})\I{t_{ui}+s\le t_{u(i+1)}}\,.\]
We then define
\begin{equation}
\label{def:0x}\llb 0,x\rrb:=\{x\} \cup \bigcap_{n\ge 0} \overline {\bigcup_{|u|=n} [t_u,\xi_u]}\,,
\end{equation}
which is then a non-empty closed subset of $[0,x]$. 
For each $n\ge 1$, we also let 
\begin{equation}\label{def:distance} d_n(0,x):=\sqrt{\frac \pi 2} \cdot \sum_{|u|=n} m_u^{1/2} \qquad \text{and} \qquad d(0,x):=\limsup_{n\to\infty} d_n(0,x)\,.\end{equation}
\begin{lem}\label{lem:dist_martingale}For each $x\in \sL$, the sequence $(d_n(0,x))_{n\ge 1}$ is a non-negative martingale. As a consequence, with probability one, the sequences $d_n(0,x)$ converge for all $x\in \sL$ to finite limits $d(0,x)$. 
\end{lem}
\begin{proof}Fix $x\in \sL$. Let $\cF_n$ denote the sigma-algebra generated by the random variables $\{(t_u, \gamma_u), |u|\le n\}$; in particular, $(m_u)_{|u|\le n}$ is $\cF_n$-measurable. Conditionally on $\cF_n$, the functions $e_u$, $u\in \N^n$, are independent Brownian excursions of respective durations $m_u$. It follows that 
\begin{equation}\label{eq:recursive_convex_martingale}
d_{n+1}(0,x)=\E{\sum_{|u|=n+1} m_u^{1/2}~\bigg|~\cF_n}=\sum_{|u|=n} m_u^{1/2} \cdot \E{\sum_{i\ge 0}(m_{ui}/m_u)^{1/2}~\Bigg|~\cF_n}\,.
\end{equation}
Let $(\Delta_{i,1},\Delta_{i,2},\Delta_{i,3})_{i\ge 0}$ be iid Dirichlet$(\tfrac 12, \tfrac 12, \tfrac 12)$ random vectors. Then, by Lemma~\ref{lem:convex_minorant_law2}, we have conditionally on $m_u$: 
\begin{equation}\label{eq:dn_conditional}
\Big(\frac{m_{ui}}{m_u}\Big)_{i\ge 0} \eqdist \bigg(\prod_{1\le j<i} \Delta_{j,2} \cdot \Delta_{i,1}\bigg)_{i\ge 0}\end{equation}
From there, it is straightforward to verify by induction that, since $\Ec{\Delta_{i,1}^{1/2} + \Delta_{i,2}^{1/2}}=1$, for each $i\ge 1$, the expectation of the square root of the right-hand side of \eqref{eq:dn_conditional} equals $2^{-i}$. As a consequence, the conditional expectations in the right-hand side of \eqref{eq:recursive_convex_martingale} all equal one almost surely, so that $d_n(0,x)$ is indeed a martingale.
Since $\sL$ is countable, the convergence is almost surely for all $x\in \sL$. 
\end{proof}

The function $d(0, \cdot)$ can be extented to $[0,1]$ as follows. First let $\llb 0,0\rrb = \{0\}$ and $d(0,0)=0$. Then, for each point $x\in (0,1]$, any $t\in \cV_x\setminus \{0,x\}$ is a local minimum, and therefore $\llb0,t\rrb$ and $d(0,t)$ has already been defined in \eqref{def:0x} and \eqref{def:distance}, respectively.  We rely on those to define, for $x\in (0,1]$, 
\begin{equation}\label{eq:def_geodesic_to_zero}
  \llb 0,x\rrb = \overline{\bigcup_{t\in \cV_x\cap \sL} \llb 0,t\rrb} \cup \{x\} \qquad \text{and } \qquad d(0,x)= \sup_{t\in \cV_x\cap \sL} d(0,t)\in \R_+ \cup \{\infty\}\,.
\end{equation}

\begin{rem}\label{rem:points_absolute}
 {\bf i)} The slight subtlety in the definition in \eqref{eq:def_geodesic_to_zero}, where the union is taken on $t\in \cV_x\cap \sL$ rather than $\cV_x\setminus \{x\}$ is to ensure that the definion in \eqref{eq:def_geodesic_to_zero} is consistent with the one in \eqref{def:0x} in the case that $x\in \sL$. For instance, if $t\in \sL$ and $x=\zi(t)$ then $x\in \cV_x$ but almost surely not in $\sL$. 

\noindent {\bf ii)} The recursive construction yields a collection of ``join points'' associated to the local minima. First there is a well-defined face to the left of $t$: almost surely for $t\in [0,1]$, $\li(t)= \sup \cV_t < t$, and $\li(t)\in \sL$, so that $[\li(t), t]$ is a face of the convex minorant $c_t$. The slope $\slo(t)$ is precisely the slope of this face. Furthermore, $U_t$ is used to define a uniform random point in $[\li(t), t]$ that we denote by $\ju(t)$. In the previous decomposition, for any $x\in [0,1]$ and any $u\in \cU$, $i\in \N$ such that $t_{u(i+1)}=t$, we have $t_{ui}=\li(t)$ and $\xi_{ui}=\ju(t)$. 
\end{rem}

Before going further, let us prove the following lemma, that will be useful later (Lemma~\ref{lem:geodesic_restriction}). Observe first that \eqref{eq:def_geodesic_to_zero} allows to extend the definition of $t_u(x), \xi_u(x), m_u(x)$, $u\in \cU$, to all $x\in [0,1]$: let $t_u(x),\xi_u(x)$ and $m_u(x)$ coincide with $t_u(t_i),\xi_u(t_i)$, $m_u(t_i)$ for all $u$ of the form $u=jv$, with $j<i$ and $v\in \cU$. When $\cV_x$ is a finite set, it is understood $m_{jv}(x)$ is only defined for the relevant values of $j$.  

\begin{lem}\label{lem:geodesic_to_zero}Almost surely, for every $x\in [0,1]$ we have $\cV_x\subseteq \llb 0, x\rrb$ and furthermore:
\begin{compactenum}[i)] {}
    \item for every $u\in \cU$, $t_u(x),\xi_u(x)\in \llb 0,x\rrb$,
    \item $\sup_{|u|=n} m_u(x) \to 0$, as $n\to\infty$, and thus
    \item $\llb 0,x\rrb$ is the closure of $\{t_u(x): u\in \cU\}\cup \{x\}$, in particular, if $x=\ri(t)$, then $\llb 0,x\rrb = \llb 0,t\rrb \cup \{x\}$.
\end{compactenum}
\end{lem}
\begin{proof}The first claim is clear from \eqref{eq:def_geodesic_to_zero}. We first prove \emph{i)} for $x\in \sL$. 
For $u\in \cU$, $t_u\in \llb 0,x\rrb$ by definition: indeed, for each $v\in \cU$, $t_{v0}=t_v$, and thus $t_u\in \cap_{n\ge |u|} \cup_{|v|=n}[t_v, \xi_v]$. For $\xi_u$, note that, almost surely $\cV_{\xi_u}$ has an accumulation point at $\xi_u$. It follows that $\xi_u$ lies in the closure of $\{t_{uk}: k\ge 0\}$. By the previous argument, all these points lie in $\llb 0, x\rrb$ which is closed, and thus $\xi_u\in \llb0,x\rrb$ as well. Now, since $\cU$ is countable, this is true for every $u$, and, hence, for every $x\in \sL$. Finally, this is true for all $x\in [0,1]$ by definition of $\llb 0, x\rrb$ in \eqref{eq:def_geodesic_to_zero}. 

\emph{ii)} We restrict our attention to the set of probability one where $m_u(y)\to 0$ as $|u|\to\infty$ for all $y\in \sL$. Fix any $x\in [0,1]$ and $\epsilon>0$. There is an $i\in \N$ large enough that $\sup \cV_x\cap \sL \le t_i+\epsilon$. Then, for any $u$ of the form $jv$ with $j\ge i$ and $v\in \cU$ either $m_{jv}(x)\le \epsilon$, or $m_{jv}(x)$ is not defined. On the other hand, for $u$ of the form $jv$ with $j<i$ and $v\in \cU$, we have $m_{jv}(x)=m_{jv}(t_i)$. It follows that $\sup\{m_u(x): |u|=n\}\le \epsilon$, which completes the proof since $\epsilon>0$ was arbitrary. 

\emph{iii)} follows readily from \emph{i)}, \emph{ii)} and the definition.
\end{proof}

\subsection{The branching structure and the convex minorant tree} % (fold)
\label{sec:the_branching_structure}

We now move on to the branching structure. Let $x,x'\in [0,1]$. With the ultimate objective of defining $d(x,x')$ we first define $x\wedge x':=\sup (\llb 0,x\rrb \cap \llb 0,x'\rrb)$. It should be understood as the closest common ancestor of $x$ and $x'$, when $0$ is seen as the root. It follows readily that the definition that the sets $\llb 0,x\rrb$ enjoy the following restriction property:
\begin{lem}\label{lem:geodesic_restriction}
Almost surely, for any $x\in [0,1]$ and $y\in \llb 0,x\rrb$, we have $y=x\wedge y$, $\llb0,y\rrb = \llb 0,x\rrb \cap [0,y]$, and $d(0,y)\le d(0,x)$
\end{lem}
\begin{proof}We restrict our attention to the set of probability one on which the events of Lemma~\ref{lem:geodesic_to_zero} all occur. 
If $y=x$, the claim is clear, so suppose that $y<x$, which implies that $y\le \sup(\cV_x\cap \sL)$. If $y=\sup(\cV_x \cap \sL)$ then $\llb 0,y\rrb$ is the closure of $\bigcup_{i\ge 1}\llb 0,t_i(x)\rrb$, so that the claim holds by \eqref{eq:def_geodesic_to_zero}. Finally consider the last case $y<\sup(\cV_x\cap \sL)$, and let $t_u=t_u(x)$, and $m_u=m_u(x)$, $u\in \cU$, defined in the previous section. For any $n\ge 1$ there exists some $u\in \cU$ with $|u|=n$ such that $y\in [t_u,t_u+m_u]$. It follows easily that $\llb 0,x\rrb \cap [0,t_u] \subseteq \llb 0,y\rrb$. Since $m_u\to 0$ by Lemma~\ref{lem:geodesic_to_zero}, we have $\llb 0,y\rrb=\llb 0,x\rrb \cap [0,y]$. The claim about the distance follows readily. 
\end{proof}

The extension of $\llb \cdot, \cdot\rrb$ and $d(\cdot,\cdot)$ to $[0,1]^2$ will require the following lemma, that will allow us to bring the (nice) points of $\sL$ back in the game:
\begin{lem}\label{lem:path_to_ancestors}
With probability one, for every $x,y\in [0,1]$ with $x>y$, there exists some $t\in \cV_x\cap \sL$ such that $x\wedge y = t \wedge y \in \llb 0, t\rrb$. 
\end{lem}
\begin{proof}We work on set $\Omega^\star$ of probability one where all the events of Lemma~\ref{lem:no_exception_convex} all occur.
Since $x>y$, let $(t_i)_{i\ge 0}$ be the vertices of $\cV_x\setminus \{x\}$, which might be a finite sequence. Then $[t_i,t_{i+1})$, $i\ge 0$, together with $[\sup_i t_i, x)$ forms a partition of $[0,x)$. On the event $\Omega^\star$, it suffices to consider the following two cases. (a) If $\sup_i t_i=x$, then there exists some $i\in \N$ for which $y\in [t_i, t_{i+1})$. By definition, $x\wedge y = t_{i+1} \wedge y \in \llb 0, t_{i+1}\rrb$, and $t_{i+1}\in \sL$. (b) Otherwise there are only finitely many vertices $t_0, t_1,\dots, t_k$, all of which are in $\sL$; then since $y<x$, we have $x\wedge y = t_k \wedge y \in \llb 0,t_k\rrb$. 
\end{proof}

The following lemma makes formal the branching structure of the sets $\llb 0,x\rrb$, $x\in [0,1]$.

\begin{lem}\label{lem:intersection_arcs}
There exists a set of probability one on which for any $x,x'\in (0,1)$, $x\wedge x'>0$ and 
\[\llb 0,x\rrb \cap [0, x\wedge x' ] = \llb 0, x'\rrb \cap [0,x\wedge x'] 
\qquad \text{and} \qquad 
\llb 0,x\rrb \cap \llb 0,x'\rrb \cap (x\wedge x',1] = \varnothing\,.\]
\end{lem} 
\begin{proof}
We work on a set $\Omega^\star$ of probability one where the events of Lemmas~\ref{lem:no_exception_convex} and~\ref{lem:geodesic_to_zero} all occur.
Let $(t_u)_{u\in \cU}$ and $(t'_u)_{u\in \cU}$ denote the recursive collections of points introduced before, for the points $x$ and $x'$, respectively. For $n\ge 1$, define 
\[x\wedge_n x':=\sup\{\{t_u: |u|\le n\}\cap \{t_u': |u|\le n\}\}\ge 0\,.\] 
Then Lemma~\ref{lem:geodesic_to_zero} \emph{iii)} implies that for every $n\ge 1$ and all $y<x\wedge_n x'$, we have $y\in \llb 0,x\rrb$ if and only if $y\in \llb 0,x'\rrb$. The sequence $(x\wedge_n x')_{n\ge 1}$ is non-decreasing and taking the limit as $n\to\infty$, it follows that $\llb 0, x\rrb$ and $\llb 0,x'\rrb$ coincide on $[0,\sup_n x\wedge_n x')$.

On the other hand, by definition of $x\wedge x'=\sup \llb 0,x\rrb \cap \llb 0,x'\rrb$, the sets $\llb 0,x\rrb$ and $\llb 0,x'\rrb$ are disjoint on $(x\wedge x', \infty)$. So to complete the proof, it suffices to prove that $\sup_n x\wedge_n x'=x\wedge x'$. For every $n\ge 1$, we have $x\wedge_n x'\in \llb 0,x\rrb \cap \llb 0,x'\rrb$ so that $x\wedge_n x' \le x\wedge x'$. To prove the converse inequality, consider an arbitrary point $y\in \llb 0,x\rrb \cap \llb 0,x'\rrb$, and observe that for any $n\ge 1$,  
there exists $u,v\in \cU$ with $|u|=|v|=n$ such that $y\in [t_u,\xi_u]$ and $y\in [t'_v, \xi'_v]$, and necessarily $t_u,t'_v\le x\wedge_n x'$. It follows that 
\[y\le x\wedge_n x' + \sup\{m_u(x): |u|=n \}\,.\] 
It follows that from Lemma~\ref{lem:geodesic_to_zero} that $y\le \sup_n x\wedge_n x'$. Since $y\in \llb 0,x\rrb \cap \llb 0,x'\rrb$ was arbitrary, we may take it as close to $x\wedge x'$ as we want, which proves that $\sup_n x \wedge_n x' = x\wedge x'$.  

Finally, we show that $x\wedge x'>0$. Without loss of generality, we assume that $x'<x$. If $x'\ge t_1$, then $x\wedge x'\ge t_1>t_0=0$. More generally, for any $n\ge 1$, if $x'\ge t_{0^{(n)}1}$, then $x'\wedge x \ge t_{0^{(n)}1}$, where $0^{(n)}1$ is the sequence formed by $n$ consecutive $0$ followed by a $1$. But for every $n\ge 1$, $t_{0^{(n)}1}$ is distributed like $m_1\times \prod_{2\le i\le n} \Delta_i$, where $(\Delta_i)_{i\ge 2}$ is a family of i.i.d.\ random variables with distribution Beta$(\tfrac 12,1)$; as a consequence, $t_{0^{(n)}1}>0$ a.s.\ for every $n\ge 1$ and $t_{0^{(n)}1}\le \sup\{m_u(x): |u|=n+1\} \to0$ by Lemma~\ref{lem:geodesic_to_zero}. Since $x'>0$, there is some $n\ge 1$ for which $0<t_{0^{(n)}1}\le x'$ which proves that $x\wedge x'>0$. The latter decomposition depends on $x$, but either $x'\ge t_1(x)>0$, or $x\wedge x'=t_1(x) \wedge x'$ so that it suffices to consider the decomposition at the set of local minima, which is countable; it follows that, almost surely, for every $x,x'\in (0,1)$, $x\wedge x'>0$.
\end{proof}

We are now ready to define $\llb x,y\rrb$ and $d(x,y)$ for all $x,y\in [0,1]$. Observe first that, by Lemma~\ref{lem:path_to_ancestors}, almost surely, for all $x\ne y$, we have $x\wedge y \in \llb 0,t\rrb$ for some $t\in \sL$, so that $d(0,x\wedge y)<\infty$. Now, if $x=y$, we set $d(x,y)=0$, and otherwise
\begin{equation}\label{eq:def_distance_xy}
d(x,y):=d(0,x) + d(0,y) - 2 d(0,x\wedge y) 
\qquad \text{and} \qquad 
\llb x,y\rrb := (\llb 0,x \rrb \cup \llb 0, y \rrb) \cap [x\wedge y, 1]\,.
\end{equation}
By the previous remark, both $d(\cdot, \cdot)$ and $\llb \cdot, \cdot\rrb$ are well-defined and symmetric on $[0,1]^2$. When necessary, we write $\llb x,y\llb=\llb x, y\rrb \setminus \{y\}$; $\rrb x,y\rrb$ and $\rrb x,y\llb$ are defined similarly. 

Finally, we verify now that $d$ induces a metric space that has the topology of a tree. In the following, we let $(x\cdot y)_0:=\tfrac 1 2 (d(0,x)+d(0,y)-d(x,y))$. Observe that, by definition, we have $(x\cdot y)_0 = d(0,x\wedge y)$. 

\begin{lem}[Triangle inequality and four-point condition]\label{lem:pseudo_metric}
A.s., for every $x,y,z\in [0,1]$, we have 
\begin{compactenum}[i)]
  \item $0\le d(x,y)\le d(x,z)+d(z,y)$, and 
  \item $(x\cdot y)_0 \ge \min \{(x\cdot z)_0, (z\cdot y)_0\}$.
\end{compactenum}
\end{lem}
\begin{proof}We prove \emph{i)} and \emph{ii)} simultaneously. Note that $d(x,y)\ge 0$ by Lemma~\ref{lem:geodesic_restriction}.
By definition, $x\wedge z\in \llb 0, x\rrb$. Suppose first that $x\wedge z\in \llb 0,x\wedge y\llb$. Then, $z\wedge y\in \llb 0, x\wedge y\llb$ as well by Lemma~\ref{lem:intersection_arcs}. It follows readily that $d(0,x\wedge y) \ge d(0,x \wedge z), d(0,y\wedge z)$. Furthermore, by definition,
\[d(x,y)=d(x,x\wedge y)+ d(x\wedge y, y)\le d(x,x\wedge z)+d(x\wedge y,y)\le d(x,z)+d(z,y)\,.\]
If on the other hand, we have $x\wedge z \in \llb x\wedge y, x\rrb$, then Lemma~\ref{lem:intersection_arcs} implies that $z\wedge y = x \wedge y$. As a consequence, we have $d(0,x\wedge y)=d(0,z\wedge y) = \min \{d(0,x\wedge z), d(0,z\wedge y)\}$. Moreover
\begin{equation}\label{eq:triangle_ineq_case2}
d(x,y)=d(x,x\wedge y)+ d(x\wedge y,y) = d(x,x\wedge z)+d(x\wedge z, x\wedge y)+ d(x\wedge y, y)\, , 
\end{equation}
and 
\begin{align*}
  d(x,z)+d(z,y)
  &= d(x,x\wedge z)+d(x\wedge z,z) + d(z,y\wedge z) + d(y\wedge z, y) \\
  &= d(x,x\wedge z)+d(x\wedge z,z) + d(z,x\wedge y) + d(x\wedge y,y)\,,
\end{align*}
which is easily seen to be at least as large as the right-hand side of \eqref{eq:triangle_ineq_case2}.
\end{proof}

By Lemma~\ref{lem:pseudo_metric}, $d$ satisfies the triangle inequality and thus induces a metric on the quotient space: Let $x\sim y$ if $d(x,y)=0$. Let $\sT^\circ:=\{x\in [0,1]: d(0,x)<\infty\}$, and write $\sT$ for the metric completion of the quotient $\sT^\circ/_\sim$; we still write $d$ for the induced metric on $\sT$. Writing $\pi$ for the canonical projection, we let $\rho=\pi(0)$ be the root of $\sT$ and $\mu$ be the push-foward of the Lebesgue measure on $[0,1]$ by $\pi$. We define $\CMT(\exc,\bU)$ as $\frak T:=(\sT,d,\mu, \rho)$.

We will later on identify exactly the distribution of $\CMT(\exc,\bU)$ the Brownian CRT (Theorem~\ref{thm:limit_mst_surplus}); however since the proof requires to introduce a number of additional concepts, it is interesting to first verify that:

\begin{prop}\label{pro:cmt_real-tree}
With probability one, the metric space $(\sT,d)$ is a real tree.
\end{prop}

\begin{rem}
{\bf i)} We define $\sT^\circ$ to ensure that $(\sT,d)$ is connected. We will see later that a.s.~$\sT^\circ=[0,1]$.

\noindent {\bf ii)} It is plausible that $(\sT^\circ,d)$ is already complete; we do not have a short argument for either direction, and we did not try to investigate further since there is no real influence on what follows.
\end{rem}

\begin{proof}[Proof of Proposition~\ref{pro:cmt_real-tree}] By the four-point condition in Lemma~\ref{lem:pseudo_metric} \emph{ii)} and Lemma 3.10 of \cite{Evans2005}, $\sT$ is $0$-hyperbolic \cite[see also][]{Chiswell2001}. Then, by Theorem 3.40 of \cite{Evans2005}, it suffices to prove that $\sT$ is connected to complete the proof.

We show that $\sT$ is path-connected; this relies on the fact, proved in Section~\ref{sub:geodesics_length}, that there exists a measure $\ell$ on $\sT$ such that for all $x,y\in [0,1]$ we have $d(x,y)=\ell(\llb x,y\rrb)$. Let $\pi$ denote the canonical projection from $[0,1]$ onto $\sT$. For any $r\in [0,d(0,x)]$, let $x_r:=\sup\{s\in \llb 0,x\rrb: d(0,s)\le r\}$. Then we claim that the map $\phi$ given by $\phi(r)=\pi(x_r)$ is an isometry from $[0,d(0,x)]$ to $\sT$. To see this, note first that since $d(0,s)$ is non-decreasing for $s\in \llb 0,x\rrb$, and the set  
$\{s\in \llb 0,x\rrb: d(0,s)\le r\}$ is closed, we have $d(0,\phi(r))=d(0,x_r)\le r$. On the other hand, for $s\in \llb 0,x\rrb$, we have $d(0,s)=\ell(\llb 0,s\rrb) = \ell(\llb 0,x\rrb \cap [0,s])$; since $\ell(\llb 0,x\rrb)<\infty$ the right-hand side is continuous if we consider $s\in [0,1]$. 
It follows that $d(0,\phi(r))=r$, and that $\phi$ is an isometry. Therefore, for every $x$, there is a geodesic from $0$ to $x$, and $\sT$ is path-connected and then connected.
\end{proof}

The following consistency property will be useful. It implies in particular that the pairwise distances may be defined using only certain suitable sub-excursions of $\exc$. 

\begin{lem}[Restriction and consistency]\label{lem:def_consistency}
For $x\in [0,1]$, let $(t_i)_{i\ge 0}$ denote the vertices of $\cV_x\cap \sL$, and $z_i=z_i(x)$, $i\ge 0$, the corresponding intercepts. Then
\begin{compactenum}[i)]
  \item for any $i\ge 0$, we have 
  \[\llb 0,x\rrb = \bigcup_{0\le j<i} \llb t_j,t_{j+1}\rrb \cup \llb t_i,x\rrb 
  \qquad \text{and} \qquad 
  d(0,x)=\sum_{0\le j<i} d(t_j,t_{j+1})+ d(t_i,x)\,.\]
  \item for every $i\ge 0$ there exists a vector $\bU_i$ constructed from $\bU$ such that, almost surely, the restriction of $\CMT(\exc, \bU)$ to $\pi([t_i,z_i])$ is isometric to $\CMT(h_i, \bU_i)$, where $h_i$ is the excursion defined in \eqref{eq:decomp_in_excursions}. 

  % we have $y\in [t_i,z_i]$ then $\llb x,y\rrb$ is measurable with respect to the values of $\exc$ on the open interval $(t_i,z_i)$, and of the corresponding points of $\bU$. Furthermore, 
  %   \[d(x,y)=d(x,t_i)+d(y,t_i)-2d(t_i,x\wedge y)\,.\]
\end{compactenum}
\end{lem}
\begin{proof}\emph{i)} By definition, $t_i\in \llb 0,x\rrb$ so that decomposing $\llb 0,x\rrb$ on $[0,t_i] \cup [t_i,x]$, it follows immediately that $\llb 0,x\rrb= \llb 0, t_i\rrb \cup \llb t_i, x\rrb$. A straightforward induction yields the claim. 

\emph{ii)} By Lemma~\ref{lem:intercept}, $t_i\in \cV_y$ so that $t_i\in \llb 0,y\rrb\cap \llb 0, x\rrb$. It follows that $t_i\le x\wedge y$. For the interval $[t_i,z_i]\subseteq [0,1]$, we now define a sequence $\bU_i$ from $\bU$ as follows. Recall Remark~\ref{rem:association_map} about the association map, and let $(I_j)_{j\ge 1}$ the enumeration of the intervals with rational end points there. Recall also the definition of $h_i$ in \eqref{eq:decomp_in_excursions}. We denote by $j(\exc, \cdot)$ and $j(h_i, \cdot)$ the association maps of $\exc$ and $h_i$ respectively. For every $t\in \sL(\exc)$, then $t-t_i\in \sL(h_i)$ (the set of local minima is a.s.\ preserved by the removal of a linear drift). Let 
\[
U_{i,k}:=
\left\{
\begin{array}{ll}
U_{j(\exc,t+t_i)} & \text{ if } k=j(h_i,t) \text{ for some } t\in \sL(h_i)\\
0            & \text{ otherwise}\,.     
\end{array}  
\right.
\]
Let $\bU_i=(U_{i,k})_{k\ge 1}$. Then, the restriction of $\CMT(\exc,\bU)$ to $\pi([t_i,z_i])$ is isometric to $\CMT(h_i, \bU_i)$. Note that the components of $\bU_i$ that we have set to $0$ above are never used in the construction; if one wants to enforce that $\bU_i$ has the same distribution as $\bU$, one can instead use independent uniform random variables to complete the definition of $\bU_i$.  
\end{proof}

\subsection{Geodesics and the length measure} % (fold)
\label{sub:geodesics_length}

In this section, we show that the distance $d(x,y)$ is actually a measurable function of the set $\llb x,y\rrb$. Let $\psi$ be the function defined by $\psi(r)=\sqrt{r|\log|\log r||}$ for $r>0$, and let $m^\psi$ denote the Hausdorff measure constructed on $\R$ using $\psi$ as a gauge function. Recall that the $\psi$-Hausdorff measure $m^\psi$ of a Borel set $E\subseteq \R$ is defined by \cite{Falconer1986,Falconer1990a,Mattila1999a}
\[m^\psi(E):=\lim_{\delta\to 0+} \inf\bigg\{\sum_{i\ge 1} \psi(A_i): E\subseteq \bigcup_{i\ge 1} A_i, |A_i|<\delta\bigg\}\,,\]
where the $A_i$ are intervals and $|A_i|$ are their lengths. 

For any $x,y\in [0,1]$, the distances between pairs of points of $\llb x,y\rrb$ naturally define a measure as follows: for any $z,t\in \llb x,y\rrb$, we have $\llb z,t\rrb\subseteq \llb x,y\rrb$ and we let $\ell_{x,y}(\llb z,t\rrb)=d(z,t)$. More generally, for any compact interval $A\subseteq [0,1]$, we let $\ell^\circ_{\llb x,y\rrb}(A)=\ell^\circ_{\llb x,y\rrb}(\llb x,y\rrb \cap A)=d(\inf \llb x,y\rrb \cap A, \sup \llb x,y\rrb \cap A)$. This defines $\ell^\circ_{\llb x,y\rrb}$ uniquely as a Borel measure on $[0,1]$. 

\begin{lem}\label{lem:measure_geodesic}Let $V$ be a random variable with uniform distribution independent of $(\exc,\bU)$. There exists a constant $a>0$ such that,
 with probability one, for any Borel set $A\subseteq [0,1]$, we have $\ell_{\llb 0,V\rrb}(A)=a \cdot m^\psi(A \cap \llb 0,V\rrb)$. In particular, $d(0,V)=a \cdot m^\psi(\llb 0,V\rrb)$.
\end{lem}

\begin{rem}It would be possible to identify the constant $a$ using Theorem~1 of \citet{Perkins1981a} who strengthened the results of \citet{TaWe1966a} by (among others) identifying the multiplicative constant between the $\psi$-Hausdorff measure and the local time for the zero set of Brownian motion. However, we did not pursue this further.
\end{rem}

\begin{proof}For $V$ uniform on $[0,1]$, the Cantor set $\llb 0,V\rrb$ has a recursive structure that is tractable with the tools developed by \citet*{GrMaWi1988a} and \citet{MaWi1986a}, which will allow us to compare $m^\psi(\llb 0,V\rrb)$ and $d(0,V)$.

Let $t_1=t_1(V)$ be the location of the first vertex of the convex minorant of $\exc$ on the interval $[0,V]$, and let $z_1=z_1(V)$. Then, by Lemma~\ref{lem:convex_minorant_law1}, 
\[(t_1,z_1-t_1, 1-z_1) \sim \text{Dirichlet}(\tfrac 12, \tfrac 12, \tfrac 12)\,,\]
so that, conditionally on $(t_1,z_1)$, $V$ is uniform in $(t_1,z_1)$. On the other hand, the random jump $\xi_1=\ju(t_1)$ is uniform in $(0,t_1)$ and independent of the rest. This implies that the random Cantor set $\llb 0,V \rrb$ has the same distribution as $C$ constructed as follows. 
Let $\cU_2:=\bigcup_{n\ge 0}\{1,2\}^n$, and let $(\Delta_1(u),\Delta_2(u), \Delta_3(u))$, $u\in \cU_2$, be i.i.d.\ copies of a Dirichlet$(\tfrac 12, \tfrac 12, \tfrac 12)$ random vector $(\Delta_1,\Delta_2,\Delta_3)$.  Set $C_\varnothing=[0,1]$ and, for each $u\in \cU_2$, let 
\[C_{u1}:=[\inf C_u, \inf C_u + |C_u|\cdot \Delta_1(u)],
\qquad \text{and} \qquad 
C_{u2}:= [\sup C_{u1}, \sup C_{u1}+ |C_u| \cdot \Delta_2(u)] \,.
\]

Observe that, for each $u$, $C_{u1}$ and $C_{u2}$ are two intervals in $C_u$, with disjoint interior. Then, for $n\ge 0$, we set $C^n=\bigcup_{|u|=n} C_u$ and $C=\bigcap_{n\ge 0} C^n$. Note in particular that no additional randomness is needed, that would correspond to the point $V$: with this definition $\sup C$ is uniformly distributed on $[0,1]$. 

The law of $(\Delta_1,\Delta_2)$ is explicit and its density $\rho(x_1,x_2)$ is given by
\[\rho(x_1,x_2)=x_1^{-1/2}x_2^{-1/2}(1-x_1-x_2)^{-1/2} \frac{\Gamma(3/2)}{\Gamma(1/2)^3} \cdot \I{x_1+x_2\le 1}\,.\] 
Theorem~5.1 of \cite{GrMaWi1988a} applies: one easily verifies that for $\alpha=1/2$ we have $\Ec{\Delta_1^\alpha+\Delta_2^\alpha}=1$, $\pc{\Delta_1^\alpha+\Delta_2^\alpha = 1} =0$, and $\Ec{1/\min\{\Delta_1^\nu,\Delta_2^\nu\}}\le 2 \Ec{\Delta_1^{-\nu}}<\infty$ for all $\nu\in (0,\tfrac 12)$. Furthermore, Condition (5.1) of \cite{GrMaWi1988a} is satisfied for the point $(x_1,x_2)=(\tfrac 12, \tfrac 12)$, since the density $\rho$ is bounded away from zero uniformly. It follows that with probability one, $m^\psi(C)\in (0,\infty)$. Now, Theorem~5.5 there does not directly apply since the $C_{u1}$ and $C_{u2}$ intersect for every $u$, but this is only at one point, and one easily verifies that the proof there still holds since $m^\psi$ assigns measure zero to any countable collection of points. We conclude that there exists a constant $a>0$ such that $d(0,V) = a \cdot m^\psi(C)$. 

Furthermore, the measure $a \cdot m^\psi(\,\cdot \cap C)$ coincides with the construction measure $\nu$ of \citet{MaWi1986a}, which is easily seen to correspond here to the measure $\ell^\circ_{\llb 0,V\rrb}$. 
% Almost surely, for every continuous function $f$ on $[0,1]$, the sequence
% \[F_n(f):=\lim_{n\to\infty} \sum_{|u|=n} f(\inf C_u) \cdot |C_u|^{1/2}\,,\]
For an interval $A\subseteq [0,1]$, the sequence 
\[\nu_n(A):=\sum_{|u|=n, C_u\cap A\neq \varnothing } |C_u|^{1/2}\]
almost surely converges to a limit value $\nu(A)$. This defines the Borel measure $\nu$ on $[0,1]$ of total mass $d(0,V)$.
%  and satisfies 
% \[\nu(A)=\lim_{n\to \infty} \sum_{\substack{|u|=n\\ C_u\cap A\ne \varnothing}} C_u^{1/2}\,.\]
The fact that $\nu(A)=\ell^\circ_{\llb 0,V\rrb}(A)$ should by now be straightforward. 
\end{proof}

The measures $\ell^\circ_{\llb x,y\rrb}$, $x,y\in [0,1]$, are actually the restrictions of a general measure on $[0,1]$ which projects to the length measure on the convex minorant tree. There is a pre-skeleton on $[0,1]$ which is defined by $\Skel([0,1])=\cup_{x\in \sL}\rrb 0,x\llb = \cup_{x,y\in \sL} \rrb x,y\llb$. 
Let $\ell^\circ$ be the Borel sigma-finite measure on $[0,1]$ uniquely defined by
\begin{enumerate}[i)]
  \item $\ell^\circ(\Skel([0,1])^c)=0$, and 
  \item for every $x,y\in \sL$, and every interval $A$ of $[0,1]$, $\ell^\circ(A \cap \llb x,y\rrb)=\ell^\circ_{\llb x,y\rrb}(A)$.
\end{enumerate}
Then the push-foward measure $\ell=\pi_* \ell^\circ$ is the length measure on the convex minorant tree $\CMT(\exc,\bU)$. Finally, we verify that this corresponds to the push-forward of $m^\psi$ (up to a multiplicative constant). This is essentially just the fact that the skeleton $\Skel([0,1])$ is a countable union of segments, that we can rewrite in terms of a sequence of i.i.d.\ uniform points on $[0,1]$

\begin{prop}\label{pro:length_measure}
Let $(V_i)_{i\ge 1}$ be i.i.d.\ uniform on $[0,1]$, also independent of $(\exc,\bU)$. Then
\begin{compactenum}[i)]
  \item $\skel([0,1]) \subseteq \cup_{i\ge 1} \llb 0, V_i\rrb$, and 
  \item with the constant $a>0$ of Lemma~\ref{lem:measure_geodesic},  the measures $\ell^\circ$ and $a\cdot m^\psi$ almost surely coincide.
\end{compactenum}
\end{prop}
\begin{proof}
\emph{i)} Fix any $x\in \sL$. Almost surely, there is an $i\in \N$, such that $x=t_i(x)$ and $z_i(x)>t_i(x)$. For any $y\in [t_i,z_i)$ we have $x=t_i(x) = t_i(y)$. In particular, a.s.\ there exist infinitely many $n\ge 1$ such that $x=t_i(V_n)$ for some $i\ge 1$. It follows that $\Skel([0,1])\subset \cup_{n\ge 1} \llb 0, V_n\rrb$. 

\emph{ii)} For each $n\ge 1$, let $B_n := \llb 0, V_n\rrb \setminus \bigcup_{1\le j<n} \llb 0,V_j\rrb$. Then, $\Skel([0,1])$ is contained in the union of the $B_n$, $n\ge 1$, which are disjoint sets, and for any interval $A\subseteq [0,1]$, we have
\begin{align*}
\ell^\circ(A) = \sum_{n\ge 1} \ell^\circ(A \cap B_n) = \sum_{n\ge 1} a\cdot m^\psi(A \cap B_n) = m^\psi(A)\,,
\end{align*}
so that the measures $\ell^\circ$ and $a\cdot m^\psi$ indeed concide.
\end{proof}

\begin{rem}\label{rem:one-point_function}
We note the decomposition for the distance $d(0,V)$ identifies its distribution: indeed, by Brownian scaling if $D$denotes the random variable $d(0,V)$, and $D_{1}$ and $D_2$ are two independent copies of $D$, then we have 
\[D \eqdist \sqrt{\Delta_1} D_1 + \sqrt{\Delta_2} D_2\,\]
which implies that, up to a deterministic multiplicative constant, $D$ has the Rayleigh distribution (see for instance Proposition~2.1 of \cite{AlGo2015a}). This can be seen as a first step towards the identification of the law of $\CMT(\exc,\bU)$; see Section~\ref{sec:a_new_point_of_view_on_the_continuum_random_tree} for a full proof of this fact.
\end{rem}

% subsection geodesics_and_the_length_measure (end)subs

\subsection{Recursive convex minorants of Brownian motion with parabolic drift} % (fold)
\label{sub:recursive_convex_minorants_BPT}

We now move on to the definition of the main object of the paper, the tree $\CMT(X,\bU)$. A straightforward application of the Girsanov Theorem shows that, for any $x\in \R_+$, the law of $(X_{s\in [0,x]})$ is absolutely continuous with respect to that of $(W_s)_{s\in [0,x]}$.
As a consequence, ``local properties'' that hold almost surely for $W$ also hold almost surely for $X$ as well.
Since we are only interested in a definition in this section, we may focus on the case of a Brownian motion. In the following, we use the same notation as for the Brownian excursion, we believe that it should not cause any confusion.

\medskip
\noindent
\noindent\textbf{The convex minorant tree associated with a Brownian motion.} We consider $(W_s)_{s\ge 0}$ a standard Brownian motion. 
Fix $x\in \R_+$ and consider the recursive convex minorants of $W$ on $[0,x]$. Recall that $\cV_x(W)$ a.s.\ has an accumulation point at $0$; let $(t_i)_{i\in \Z}=(t_i(x,W))_{i\in \Z}$ as defined in Section~\ref{sec:convex_minorant_of_brownian_paths}. 
The sequence $(t_i)_{i\in \Z}$ is bi-infinite, but one can write for a fixed $i$
\begin{equation}\label{eq:def_distance_brownian-motion}
\llb 0,x\rrb := \{0\} \cup \bigcup_{j:j\le i} \llb t_{j-1},t_{j}\rrb \cup \llb t_i,x\rrb\,.  
\end{equation}
By Lemma~\ref{lem:W_decomp}, the sets $\llb t_{j-1},t_j\rrb$, for $j$ such that $j\le i$ and $\llb t_i,x\rrb$ are well-defined by the construction of Section~\ref{sec:recursive_convex_minorants} (for the Brownian excursion). Furthermore, Lemma~\ref{lem:def_consistency} ensures that the value of $\llb 0,x\rrb$ is independent of $i\in \Z$, so that $\llb 0,x\rrb$ is well-defined as well. 
To define the distance $d(0,x)=d_\subW(0,x)$, we shall verify that the sum of distances given by the decomposition in \eqref{eq:def_distance_brownian-motion} converges (a priori, 0 could be at infinite distance from every point $x>0$). Observe that, still from Lemma~\ref{lem:W_decomp}, conditionally on $(t_i)_{i\in \Z}$, for any $i\in \Z$, $d(t_i,t_{i+1})$ is distributed like $|t_{i+1}-t_i|^{1/2}$ times the distance between $0$ and $1$ in $\CMT(\exc,\bU)$. Also by Theorem~1, Corollary~1 of \cite{PiRo2011a} and Brownian scaling, $(|t_{i+1}-t_i|)_{i \in \Z}$ has the same distribution as $(x_i/\sum_{j\in \mathbb{Z}} x_j)_{i\in \Z}$,  where $(x_j)_{j\in \Z}$ denote the points of a Poisson point process of intensity $e^{-x}/x dx$  on $\R_+$.
Straightforward calculation shows that $\pc{\sum_i x_i >0} = 1$ and $\Ec{\sum_i x_i^p}=\int_{0}^{+\infty} x^{p-1}e^{-x} dx<\infty$ for $p\in(0,1]$ so that a.s. $\sum_i \sqrt{x_i}<+\infty$ and  $\sum_i {x_i}<+\infty$ which implies that $\sum_{i} |t_{i+1}-t_i|^{1/2} < \infty$ almost surely. 

So, for the distance we may define $d(0,x)=d_\subW(0,x)$ by 
\[d_\subW(0,x):= \sum_{j\le i} d(t_{j-i}, t_j) + d(t_i,x) = \sum_{j\in \Z} d(t_{j-1},t_j)\,,\]
which is almost surely finite for almost all $x\in \R_+$. In particular, with probability one $d_\subW(0,x)<\infty$ for every $x\in \sL(W)$ by Lemma~\ref{lem:no_exception_convex}. %\JF{en contradiction avec ce qui suit, en magenta}

Finally, for any $x,y>0$, we can define $d(x,y)=d_\subW(x,y)$ as follows. Writing $x\wedge y=\sup \llb 0,x\rrb \cap \llb 0,y\rrb$ as before, we have $x\wedge y \in \llb0,t\rrb$ for some $t\in \sL$, by the obvious extension of Lemma~\ref{lem:path_to_ancestors} to the case of Brownian motion. Therefore, $d(0,x\wedge y)<\infty$. We may thus define
\[d_\subW(x,y):=d_\subW(0,x)+d_\subW(0,y)-2 d_\subW(0,x\wedge y)\,,\]
and we will prove that it is a.s.\ finite for all $x,y$. 
Note also that, assuming without loss of generality that $y<x$, there exists some $i\in \Z$ such that $y\in [t_i,t_{i+1})$. In particular, $x,y\in [t_i,z_i]$ and we may equivalently define $d(x,y)$ by
\[d(t_i,x)+d(t_i,y)- 2d(t_i,x\wedge y)\,,\]
and any $i$ of which $t_i<y$ would yield the exact same value. 

\medskip
\noindent\textbf{The convex minorant tree associated with $X$: The Brownian parabolic tree. } 
At last, we consider $X$, the Brownian motion with parabolic drift. 
By absolute continuity, the sets $\llb x,y\rrb=\llb x,y\rrb_\subX$ and $d(x,y)=d_\subX(x,y)$ are also well-defined for every $x,y\in \R_+$. The triangle inequality and four-point condition are satisfied by construction (Lemma~\ref{lem:pseudo_metric}). Let $x\sim y$ if $d_\subX(x,y)=0$, and let $(\sM,d)$ denote the metric completion of the quotient metric space; define $\rho=\pi(0)$. For the mass measure, one needs some rescaling and we shall admit for now that the collection of measures $(x^{-1} \pi_* \Leb|_{[0,x]})_{x>0}$ conve{}rges weakly with probability one to a probability measure $\mu$. The proof of this fact is the topic of Section~\ref{sub:mass_measure}. Finally, we let $\CMT(X,\bU)$ denote the pointed measured complete metric space $\frak M:=(\sM,d,\mu,\rho)$, and we call it the Brownian parabolic tree.

% subsection recursive_convex_minorants_BPT (end)

% \begin{alert}Say explicity somewhere that we construct an explicit metric space, and we work with it. So $\CMT(e,\bU)$ and $\CMT(X^0,\bU)$ should not be equivalence classes. So the statement should be understood as the equilance class of $\CMT(X^0,\bU)$ has the same distribution as the scaling limit... The precise representation with $\R_+$ in $\sM$ is important.
% \end{alert}

% \section{A dynamic point of view and the law of $\CMT(e,\bU)$}
% \label{sec:dynamics}
% \input{dynamics}

\section{A dynamic point of view and the law of $\CMT(\exc,\bU)$} % (fold)
\label{sec:a_new_point_of_view_on_the_continuum_random_tree}
%!TEX root = MST_brownian.tex

% \begin{alert}Decide about the orientation of time in this secion. It could make sense to use $Z^\tau$, $\tau\ge 0$,  rather than $Z^\lambda$, $\lambda \le 0$. 

% $Z^\tau_-$

% Insert ?
% \begin{compactitem}
%   \item \citet{Pitman1999b} coalescent random forests ? discrete version of \cite{AlPi1998a}
%   \item \cite{Bertoin2012a} which contains the origins of the cut tree in \cite{BeMi2013a}
% \end{compactitem}

% \end{alert}

In this section, we study the convex minorant tree of a standard Brownian excursion. We prove Theorem~\ref{thm:limit_mst_surplus} in the case where $s=0$ which says that $\CMT(\exc,\bU)$ is distributed like the Brownian continuum random tree, and Theorem~\ref{thm:additive_coalescent} which relates $\CMT(\exc,\bU)$ to the additive coalescent. We are interested here in the case of excursions, and the natural range of interest for $Z^\lambda(\omega)$ is then $\lambda \in (-\infty, 0]$, and we shall therefore rather work with $(Z^{-\tau}(\omega))_{\tau \ge 0}$, which also turns out to be a cadlag process (see Lemma~\ref{lem:properties_Z}). We still occasionally use the parameterization with $\lambda$.

% In the new encoding of the Brownian continuum random tree, we can express the distances using an explicit Hausdorff measure for the geodesics; this is the topic of Section~\ref{sub:geodesics_length}.

\subsection{A fragmentation connected to Brownian motion} % (fold)
\label{sub:a_fragmentation_connected_to_brownian_motion}

The properties of $Z^\lambda$ is intimately related to the following operators. For $\lambda \in \R$, define the operator $\Psi_\lambda$ as follows: for a function $f$ continuous on an interval $D\subseteq \R_+$, and $t \in D$ 
\begin{equation}\label{eq:def_shear}
  \Psi_\lambda f(t):=f(t)+\lambda t -\inf\{f(s)+\lambda s: s\in D, s\le t\}\,.
\end{equation}
Then, $Z^\lambda(\omega)=\{s\in D: \Psi_\lambda \omega (s) = 0\}$. 
The family of operators $(\Psi_\lambda)_{\lambda \in \R}$ enjoys the following composition property, which is a straightforward reformulation of the arguments leading to Theorem~1 i) of \cite{Bertoin2000a}. For $t\ge 0$, let $\frak S_t$ denote the shift operator defined by $\frak S_t f(s)= f(t+s)$, for all $s\ge 0$. 
\begin{lem}\label{lem:shear_composition}
  Let $f$ be a continuous function on $D\subseteq \R_+$ and suppose that, for some $\lambda \in \R$ and $t\in D$, we have $\Psi_\lambda f(t)=0$. Then, for all $h,s\ge 0$ with $t+s\in D$ one has
  \[\frak S_t \Psi_{\lambda-h} f (s) = \Psi_{\lambda-h}f(t+s) = \Psi_{-h} \frak S_t \Psi_\lambda f(s)\,.\]
  In particular, $\Psi_{\lambda-h}f(t)=0$ for all $h\ge 0$. 
\end{lem}

We now go back to the case where $f=\exc$ is a Brownian excursion and write $Z^{\lambda}=Z^\lambda(\exc)$ (in this case, $D=[0,1]$). Lemma~\ref{lem:shear_composition} implies for instance that $Z^{-\tau}=Z^{-\tau}(\exc)$ is non-decreasing in $\tau$ for the inclusion, and thus induces a fragmentation in the sense that the connected components of its complement split as $\tau$ increases. For any $x\in [0,1)$ let $I^\tau(x)$ be the maximal interval of the form $[a,b)$ containing $x$ such that for $(a,b)\cap Z^{-\tau} = \varnothing$. For $x,y\in [0,1)$, we let $x\sim_\tau y$ if $I^\tau(x)=I^\tau(y)$. Observe that, for every $\tau\ge 0$, the collection of $I^\tau(x)$ forms a partition of $[0,1)$. 

By Lemma~\ref{lem:properties_Z}, a.s, for every $\tau\ge 0$, $[0,1]\setminus Z^{-\tau}$ consists in countably many open intervals, whose lengths we denote by $F_1(\tau), F_2(\tau), \dots$ in the decreasing order. Then, let $F(\tau)=(F_i(\tau))_{i\ge 1}$. 
The main result of \citet{Bertoin2000a} is that the process $(F(\tau), \tau\ge 0)$ has the same distribution as another remarkable fragmentation introduced by Aldous and Pitman \cite{AlPi1998a}, where a Brownian continuum random tree is logged along its skeleton at the points of an (independent) Poisson point process of unit intensity; the process of interest is the sequence of sorted masses of the fragments. This shows in particular that, up to a time change, the time reversal of $(F(\tau))_{\tau\ge 0}$ is the classical standard additive coalescent. 

Although there is no obvious coupling between the two representations directly in the continuous, this shows that the fragmentation of $[0,1]$ constructed by Bertoin corresponds to a fragmentation of a certain Brownian continuum random tree. This section will show that (one choice for) this tree is the convex minorant tree $\CMT(\exc,\bU)$. We will also identify the collection of points/times where/when it should be cut and thereby, provide a coupling between the two representations. 

The rest of the section is organized as follows. In Section~\ref{sec:dynamics}, we make explicit the correspondence between $\CMT(\exc,\bU)$ and the dynamics related to the process $Z^{-\tau}$ described above. In Section~\ref{sub:the_genealogy_and_the_marked_cut_tree}, following \citet{BeMi2013a}, we introduce the cut tree which encodes the genealogy of the fragmentation $(F(\tau))_{\tau\ge 0}$. The cut tree is a crucial ingredient since it provides the link between the fragmentation and the recovery of ``the tree what was logged'' through the \emph{inverse cut tree transform} that has been studied in \cite{AdBrGo2010,BrWa2017b}. In Section~\ref{sub:distribution_CMT}, we make the connection between the cut tree, the inverse transform and $\CMT(\exc,\bU)$ and complete the proofs of Theorem~\ref{thm:limit_mst_surplus} (with $s=0$) and Theorem~\ref{thm:additive_coalescent}.

\subsection{Making the dynamics explicit}
\label{sec:dynamics}

In this section, we provide another point of view on the convex minorant tree that makes explicit its relation with the fragmentation of $[0,1]$ induced by $Z^{-\tau}=Z^{-\tau}(\exc)$.  

\begin{lem}\label{lem:fragment_of_x}Almost surely, the following holds for every point $x\in [0,1]$. Let $(t_i)_{i\ge 0}$ and $(\gamma_i)_{i\ge 0}$ be the vertices and the slopes of the convex minorant of $\exc$ on $[0,x]$. Then, setting $\gamma_{-1}=0$ for convenience,  we have for all $i\ge 0$, 
\begin{compactenum}[i)]
  \item $\inf I^{\tau}(x) = t_i$ for all $\tau \in [\gamma_{i-1}, \gamma_{i})$, and 
  \item $\sup I^{\tau}(x) = z_i$ for $\tau = \gamma_{i-1}$. 
\end{compactenum}
\end{lem}
\begin{proof}We work on a set $\Omega^\star$ of probability one where all the events of Lemma~\ref{lem:no_exception_convex} occur, in particular, the slopes $(\gamma_i)_{i\ge 0}$ are strictly increasing for every $x\in [0,1]$. The rest of the proof is deterministic, and we proceed by induction on $i\ge 0$. 

Write $\exc^\lambda$ for the function $s\mapsto \exc(s)+\lambda s$. 
For $i=0$, by construction of the convex minorant, for every $\tau \in [0,\gamma_0)$, $\exc^{-\tau}$ is positive on $(0,x]$ and thus $\inf I^{\tau}(x)=0=t_0$. For $\tau=\gamma_0$ we have $\exc^{-\gamma_0}(t_1)=\exc^{-\gamma_0}(z_1)=0$, and $\exc^{-\gamma_0}(s)>0$ for $s\in (t_1,z_1)$. It follows that $t_1,z_1\in Z^{-\gamma_0}$ and that $\sup I^{\gamma_0}(x)=z_1$. 

Suppose now that, for some $j\ge 0$, the claims in \emph{i)} and \emph{ii)} both hold for all $0\le i\le j$, and that $t_{j+1},z_{j+1}\in Z^{-\gamma_{j}}$. By expressing $\Psi_{-\gamma_j-h}\exc$ for $h\ge 0$ in terms of $\Psi_{-\gamma_j}\exc$, Lemma~\ref{lem:shear_composition} allows us to proceed. First note that the vertices of the convex minorant of $\Psi_{-\gamma_{j}}\exc$ on $[0,x]$ that are in $[t_{j+1},1]$ are precisely $(t_{j+k})_{k\ge 1}$ and the corresponding slopes are $(\gamma_{j+k}-\gamma_j)_{k\ge 1}$. The argument we have just used for $j=0$ applies to $\frak S_{t_{j+1}}\Psi_{-\gamma_j}\exc$ and yields that for all $h\in [0, \gamma_{j+1}-\gamma_j)$, we have $\inf I^{\gamma_j +h}(x) = t_{i+1}$ and $\sup I^{\gamma_{i+1}}(x)=z_{i+1}$. Furthermore, for $h=\gamma_{j+1}-\gamma_j$, $t_{j+2}$ and $z_{j+2}$ are both zeros of $\Psi_{-\gamma_{j+1}}\exc$, while the latter is positive on $(t_{j+2},z_{j+2})$. This completes the proof.
\end{proof}

For $x,y\in [0,1)$, define $\tau(x,y)=\sup\{\tau\ge 0: x\sim_{\tau} y\}$. Note that, by the left-continuity of $Z^\lambda$, we have $Z^{-\tau(x,y)}\cap [x,y]\ne \varnothing$. Recall the definition of $\xi_m$ from Section~\ref{sub:recursive_convex_minorants}, which is also the point $\ju(t_m)$ as defined in Remark~\ref{rem:points_absolute}. 

\begin{lem}\label{lem:cut_point-time}
Almost surely for every $x\ne y\in [0,1)$, we have the following: let $(t_i)_{i\ge 0}$ be the vertices of the convex minorant of $\exc$ on $[0,\max\{x,y\}]$. Then, $m:=\min\{i\ge 1: t_i>\min\{x,y\}\}<\infty$, and:
\begin{compactenum}[i)]
  \item $Z^{-\tau(x,y)}\cap [x,y]$ consists of the single point $\kappa(x,y)=t_m$ that we call a cut point;
  \item $\tau(x,y)=\gamma_{m-1}$;
  % \item $I^{\tau(x,y)-}(x)=I^{\tau(x,y)-}(y)=[t_{m-1}, z_m)$; and
  \item $I^{\tau(x,y)}(\min\{x,y\})=[t_{m-1}, t_m)$ and $I^{\tau(x,y)}(\max\{x,y\})=[t_m, z_m)$.
\end{compactenum} 
Furthermore, we let $\eta(x,y)=\xi_m=\ju(t_m) \in (t_{m-1},t_m)$; conditionally on $I^{\tau(x,y)}(\min\{x,y\})=S$, $\eta(x,y)$ is uniformly distributed on $S$.
\end{lem}
\begin{proof}The set of probability one is $\Omega^\star$ where all the events of Lemma~\ref{lem:no_exception_convex} occur for every point of $[0,1]$. The points \emph{i)} to \emph{iii)} are straightforward consequences of Lemma~\ref{lem:fragment_of_x}, applied to the fragment containing $I^\tau(\max\{x,y\})$ until the time when it does not contain $\min\{x,y\}$ any longer. The statement concerning the distribution of $\eta(x,y)$ is a consequence of fact that $\eta(x,y)$ is then $\ju(t_m)$, which is uniform in $[t_{m-1},t_m]$ conditionally on $t_{m-1},t_m$.
\end{proof}

Observe that Lemma~\ref{lem:cut_point-time} implies that, almost surely for every $x\ne y$, we have
\begin{equation}\label{eq:frag_at_cuttime}
I^{\tau(x,y)-}(x)=\bigcap_{\tau<\tau(x,y)} I^{\tau}(x) = I^{\tau(x,y)}(x) \sqcup I^{\tau(x,y)}(y)\,. 
\end{equation}

We are now ready to move on to the main objective of this section, namely proving that both $\llb x,y\rrb$ and $d(x,y)$ may be defined using an alternative binary decomposition where the intervals containing a pair of marked points are split at the corresponding cut point, just as in \eqref{eq:frag_at_cuttime} above.

Let $\cU_2=\bigcup_{n\ge 0} \{0,1\}^n$, where it is understood that $\{0,1\}^0=\{\varnothing\}$. 
Fix now $x,y\in (0,1)$. We define recursively $(\Pi_u,\tau_u,\kappa_u, A_u, B_u)_{u\in \cU_2}$, where $\Pi_u$ is an interval,  $A_u\le B_u$ are two points in the closure of $\Pi_u$, and the values $\tau_u\in \R_+$, $\kappa_u\in [0,1]$ are always such that $\tau_u=\tau(A_u,B_u)$, $\kappa_u=\kappa(A_u,B_u)$. It is understood that all these random variables depend on $x,y$, so we actually have $\Pi_u(x,y), \tau_u(x,y), \kappa_u(x,y)$, $A_u(x,y)$, $B_u(x,y)$, for $u\in \cU_2$, but we usually omit the reference to $x,y$. Set $\Pi_\varnothing=(0,1)$, $\tau_\varnothing(x,y)=\tau(x,y)$, $\kappa_\varnothing(x,y)=\kappa(x,y)$ and $A_\varnothing=\min\{x,y\}$, $B_\varnothing = \max\{x,y\}$. Let $\Pi_0=I^{\tau_\varnothing}(A_\varnothing)$ and $\Pi_1=I^{\tau_\varnothing}(B_\varnothing)$. 

Assuming that we have defined $(\Pi_u, \tau_u, \kappa_u, A_u, B_u)$ for some $u\in \cU_2$ we then set $\Pi_{u0}=I^{\tau_u}(A_u)$, $\Pi_{u1}=I^{\tau_u}(B_u)$, $A_{u0}=\min\{A_u,\eta(A_u,B_u)\}$, $B_{u0}=\max\{A_u,\eta(A_u,B_u)\}$ and $A_{u1}=\kappa_u=\inf \Pi_{u1}$, $B_{u1}=B_u$. We finally define $\tau_{ui}=\tau(A_{ui},B_{ui})$ and $\kappa_{ui}=\kappa(A_{ui},B_{ui})$ for $i\in \{0,1\}$. 

\begin{lem}\label{lem:cut_points}Almost surely for every $x,y\in[0,1]$, for every $u\in \cU_2$, we have $A_u,B_u\in \llb 0,x\rrb \cup \llb 0,y\rrb$.
\end{lem}
\begin{proof}This is a straightforward induction. For $u=\varnothing$, we have $\{A_\varnothing,B_\varnothing\}=\{x,y\}$ and the claim holds by definition. Assume now that it holds for some $u\in \cU$. We have $\{A_{u0}, B_{u0}, A_{u1}, B_{u1}\}=\{A_u,B_u, \kappa(A_u,B_u), \eta(A_u,B_u)\}$. By Lemma~\ref{lem:cut_point-time}, $\kappa(A_u,B_u)$ is a vertex on the convex minorant of $\exc$ on the interval $[0,\max\{A_u,B_u\}]$ and thus lies in $\llb 0, \max\{A_u,B_u\}\rrb \subseteq \llb 0,x\rrb \cup \llb 0,y \rrb$ by the induction hypothesis and Lemma~\ref{lem:geodesic_restriction}. The same holds for $\eta(A_u,B_u)$ by Lemma~\ref{lem:cut_point-time} and the definition of $\llb \cdot, \cdot \rrb$.
\end{proof}

To avoid any difficulties, we define $\hat d_\lambda$ only for almost every pair of points. This will be enough to exhibit the dynamic properties we have in mind, and settle the foundations for the coupling of Section~\ref{sec:coupling} that allows to identify the law of $\CMT(X,\bU)$.  
In the following, $\overline{\Pi}_u$ denotes the closure of $\Pi_u$. Let 
\[\Pi(x,y) = \bigcap_{n\ge 0} \bigcup_{|u|=n} \overline{\Pi}_u 
\qquad \text{and} \qquad {}
\hat d(x,y) = \limsup_{n\to \infty} \sum_{|u|=n} |\Pi_u|^{1/2}\,. 
\]
The set $\Pi(x,y)$ is well-defined and non-empty, but it is so far unclear whether $\hat d(x,y)$ is finite. 
% \begin{alert}If $\Pi(x,y)$ and $\hat d(x,y)$ are only defined for a.e.\ $(x,y)$, we cannot have something a.s.\ for every pair. FIX
% \end{alert}
\begin{prop}\label{pro:distance_pairs}For any $x,y\in [0,1]$, we have almost surely
\begin{compactenum}[i)]
  \item $\Pi(x,y) = \llb x,y\rrb$, and
  \item $\hat d(x,y) = d(x,y)$.
\end{compactenum}
\end{prop}
\begin{proof}\emph{i)} Recall that, by definition, $\llb x,y \rrb$ is the union of $\llb 0,x\rrb \cap [x\wedge y, 1]$ and $\llb 0,y\rrb \cap [x\wedge y, 1]$. We follow the binary decomposition defining $\Pi(x,y)$; for each $n\ge 0$, let $0^n$ be the left-most node in $\cU_2$ at level $n$; we agree that, in this context, $0^0=\varnothing$. We show that, for each $n\ge 0$, the two sets $\Pi(x,y)$ and $\llb x,y\rrb$ coincide on $[\kappa_{0^n},1]$; we will then show that $\kappa_{0^n}=\kappa_{0^n} \downarrow x\wedge y$ as $n\to\infty$.   

For $n=0$, we have $\kappa_{0^0}=\kappa_{\varnothing}=\kappa(x,y)$, $A_\varnothing = \min\{x,y\}$ and $B_\varnothing = \max\{x,y\}$. By Lemma~\ref{lem:cut_point-time} \emph{i)} and the definition of $\llb 0, B_\varnothing\rrb$ in Equation~\eqref{def:0x}, the set $\Pi(x,y)\cap [\kappa_{0^0}, 1]$ is contained in $\llb 0, B_\varnothing\rrb$; furthermore, since $x\wedge y \le \min\{x,y\}=A_\varnothing$, it is also the case that $\Pi(x,y) \cap [\kappa_{0^0}, 1]$ is  contained in $\llb 0,B_\varnothing\rrb \cap [x\wedge y, 1]$. On the other hand, by Lemma~\ref{lem:cut_points}, $\kappa(x,y)\in \llb 0, \max\{x,y\}\rrb$, and one easily sees that $\Pi(x,y)\cap [\kappa_{0^0},1]=\llb 0, B_\varnothing\rrb \cap [\kappa_{0^0},1]$. Indeed, we may now expand $\Pi_{1}$ on the right using the recurrence relation: writing $0^i1^j$ for the node at level $i+j$ in $\cU_2$ obtained by walking $i$ steps left, and then $j$ steps right from the root, and it should be plain that the points $\kappa_{1^k}$, $k\ge 0$, are simply the vertices of the convex minorant of $\exc$ on $[0,B_\varnothing]$ that are larger than $\kappa_\varnothing = \kappa_{0^0}$. It follows that the sets $\Pi_{10}, \Pi_{1^20},\dots, \Pi_{1^i0}, \dots $ all explicitly appear in the decomposition defining $\llb 0, B_\varnothing\rrb$ on the interval $[\kappa_{0^0},1]$. 

Now for any $n\ge 1$, assuming that we have treated the part of $\Pi(x,y)$ lying in $[\kappa_{0^n},1]$, we are left with the portion of $\Pi(x,y)$ that lies in $[0,\kappa_{0^n}]$, which is constructed from $\Pi_{0^{n+1}}$. By Lemma~\ref{lem:cut_points}, we have $A_{0^{n+1}},B_{0^{n+1}}\in \llb 0,x\rrb \cup \llb 0, y\rrb$, and we have $\kappa_{0^{n+1}}=\kappa(A_{0^{n+1}},B_{0^{n+1}})$. To the right, we have the set $\Pi_{0^{n+1}1}$, that we may expand from the right using the recurrence relation. The arguments above imply that the $\kappa_{0^n1^k}$, $k\ge 0$, are the vertices of the convex minorant on the interval $[0,B_{0^{n+1}}]$ that are at least $\kappa_{0^{n+1}}$. Therefore, $\Pi(x,y)$ and $\llb x,y\rrb$ coincide on $[\kappa_{0^{n+1}}, B_{0^{n+1}}]$ and thus on $[\kappa_{0^{n+1}}, \kappa_{0^n}]$, and in turn on $[\kappa_{0^{n+1}}, 1]$ by the induction hypothesis. 

% On the right, the set $\Pi \cap [\kappa(x,y), \max\{x,y\}]$ is contained in $\llb 0,\max\{x,y\}\rrb$, and this $x\wedge y \le \min\{y,y\}$, $\Pi\cap [\kappa(x,y),\max\{x,y\}]$ is also contained in $\llb 0,\max\{x,y\}\rrb \cap [x\wedge y,\max\{x,y\}]$. Then, a simple induction yields that for every $u_n=0\dots0$, with $|u_n|=n$, we have the same on the interval $\Pi_{u_n}$: the set $\llb[\kappa_{u_n}, B_n\rrb]$ is 

% The set $\llb \kappa(A_{u_n},B_{u_n}), B_{u_n}\rrb$ is constructed as $\Pi \cap [\kappa_{A_{u_n}, B_{u_n}}, 1] \cap \Pi_{u_n1}$.
% \begin{alert}Need a simple argument for $\Leb(\Pi_{u_n})\to 0$: This is a bit of a pain to write: essentially, either (1) for some finite $N\in \N$, and all $n\ge N$, we have $\Pi_{u_n}$ is the interval containing two random points, and the the length decreases as a product of iid beta; $N$ is the first time when $A_{u_n}\not\in \{x,y\}$. Or (2) $N=\infty$ and then $\min\{x,y\}=x\wedge y$ so we are fine as well. 

% We can avoid the problem by taking only dealing with almost every $x,y$... 
%  \end{alert} 

% with $|v_n|=n$ and $v_n=0\dots 0$, we have 
Then, note that for each $n\ge 0$, $x\wedge y \in \Pi_{0^n}$. To see this, it suffices to note that for each $n \ge 0$, one of $(A_{0^n}, B_{0^n})$ or $(B_{0^n},A_{0^n})$ lies in $\llb 0, x\rrb \times \llb 0,y \rrb$. This is clearly true for $n=0$, and carries on because at each step we replace $\max\{A_{0^n}, B_{0^n}\}$ by $\eta(A_{0^n}, B_{0^n})$ which lies in $\llb 0, \max\{A_{0^n}, B_{0^n}\}\rrb$. Lemma~\ref{lem:geodesic_restriction} them implies that $\inf \Pi_{0^n} \in \llb 0,x\rrb \cap \llb 0,y\rrb$, which proves the claim. Since $\inf \Pi_{0^n}$ is non-decreasing, it would suffice to prove that $\diam(\Pi_{0^n})=\Leb(\Pi_{0^n})\to 0$ in order to show that  $\kappa_{0^n} \to x\wedge y$, which would complete the proof of \emph{i)}. So let us now this why $\Leb(\Pi_{0^n})\to 0$. For every $u$, we have $A_u\in [\inf \Pi_u, \kappa_u]$. Then, either $A_u<\eta_u$ and $|\Pi_{u0}|\le |\Pi_u| \cdot U$ where $U$ is uniformly random on $[0,1]$, or $A_u\ge \eta_u$, and then $\Pi_{u0}$ contains two uniform random points so that, $|\Pi_{u00}|\le |\Pi_u| \cdot M$, where $M$ is a Beta$(\tfrac 12, 1)$ random variable by Lemma~\ref{lem:convex_minorant_law1}. Since all the random variables are independent, it is straightforward that $|\Pi_{0^n}|\to 0$ with probability one as $n\to\infty$. 

\emph{ii)} The correspondence between the sets that are used to define $\Pi(x,y)$ and $\llb x,y\rrb$ in the proof of \emph{i)}, also yields a way of rewriting the sums which proves that $\hat d(x,y)=d(x,y)$. We omit the details.
\end{proof}

\black

\subsection{The cut tree and the reconstruction problem} % (fold)
\label{sub:the_genealogy_and_the_marked_cut_tree}

The fragmentation we have presented in Section~\ref{sub:a_fragmentation_connected_to_brownian_motion} has a remarkable genealogy, which can be encoded into a \emph{cut tree} introduced by \citet{BeMi2013a}, and which turns out to be distributed like a Brownian continuum random tree. 

Let $(\zeta_i)_{i\ge 1}$ be i.i.d.\ uniform points in $[0,1]$, which are also independent of $(e,\bU)$. Almost surely, for all $i\ne j$, we have $i\sim_0j$. Then, for distinct $i,j\ge 1$ let $\tau_{ij}=\inf\{\tau\ge 0: \zeta_i \not \sim_\tau \zeta_j\}$ be the first time when $\zeta_i$ and $\zeta_j$ are separated by a point of $Z^{-\tau}$. Then, we define a function $\delta$ on $\N_0\times \N_0$ as follows:
\begin{equation}\label{eq:def_cut-tree}
  \delta(0,i)=\int_{0}^\infty |I^\tau(\zeta_i)| d\lambda 
  \qquad \text{and}\qquad
  \delta(i,j) = \int_{\tau_{ij}}^\infty |I^\tau(\zeta_i)|d\tau + \int_{\tau_{ij}}^\infty |I^\tau(\zeta_j)| d\tau \,,
\end{equation}
where $|\cdot|$ denotes the Lebesgue measure on $[0,1]$. 
It is known that $\delta$ defines a real tree \cite{BeMi2013a}: let $\sC$ denote the completion of $\N_0$ with respect to $\delta$, and let $\nu$ denote the weak limit of probability rescaled counting measure on $\{0, 1,2,\dots, n\}$; then $(\sC,\delta, \nu, 0)$ is a measured real tree rooted at $0$ that we call the \emph{cut tree}; $\N_0=\{0,1,2, \dots\}$ should be seen as a collection of marks in $\sC$. The measured tree $(\sC, \delta, \nu, 0)$ is distributed like a Brownian continuum random tree, and the collection of points $\N\subseteq \sC$ is an i.i.d.\ sequence with common distribution $\nu$ \cite{BeMi2013a,AdDiGo2019a,BrWa2017b}.

% {\nic We let $i\wedge j$ denote point of the unique path between $i$ and $j$ in $\sC$ at minimum distance from $0$. Then 
% \begin{equation}\label{eq:def_branch-point_in_cut-tree}
%  \delta(i,i\wedge j) = \int_{\tau_{ij}}^\infty |I^\tau(V_i)| d\tau
%  \qquad \text{and} \qquad 
%   \delta(j,i\wedge j) = \int_{\tau_{ij}}^\infty |I^\tau(V_j)| d\tau 
%  \qquad \text{and} \qquad 
% \end{equation}
% }

For each $s\in [0,1]$, let $\Gamma_\tau(s):=\{i\in \N: \zeta_i \sim_\tau s\}$. Then, each $i\in \N$ is the image of $\zeta_i$ in the cut tree $\sC$ in the sense that $\Gamma_\tau(\zeta_i)$ converges in $\sC$ as $\tau\to \infty$ to the singleton $\{i\}$ (see \cite{AdDiGo2019a}). Every branch point of $\sC$ corresponds to a fragmentation event, just as reflected by the definition in \eqref{eq:def_cut-tree}. For any $i,j\in \N$, let $i\curlywedge j$ be the common ancestor of $i$ and $j$ in $\sC$, that is the point at distance 
\[\int_0^{\tau_{ij}} |I^\tau(\zeta_i)| d\tau=\int_0^{\tau_{ij}} |I^\tau(\zeta_j)|d\tau\]
from $0$ on the paths between $0$ and $i$, and between $0$ and $j$. Here, $i\curlywedge_\sC j$ corresponds to the (unique) fragmentation event that occurs at time $\tau_{ij}$, and that separates $\zeta_i$ from $\zeta_j$. Let $\sC_{i\curlywedge j}^i$ and $\sC_{i\curlywedge j}^j$ be the two subtrees of $\sC$ above the point $i\curlywedge j$ that contain respectively $i$ and $j$; then for every $k\in \N$ we have $\zeta_i\sim_{\tau_{ij}} \zeta_k$ precisely if $k\in \sC_{i\curlywedge j}^i$. Furthermore, the interval $I^{\tau_{ij-}}(\zeta_i)=I^{\tau_{ij-}}(\zeta_j)$ which contains all the $\zeta_k$ for which $\zeta_k\sim_{t} \zeta_i$ for all $i<\tau_{ij}$ splits into the two intervals $I^{\tau_{ij}}(\zeta_i)$ and $I^{\tau_{ij}}(\zeta_j)$ by the removal of the unique point of $Z^{-\tau_{ij}}$ lying in the interior of $I^{\tau_{ij-}}(\zeta_i)$. 

Observe that the cut tree is only constructed from the process of masses of the fragments containing a sequence of i.i.d.\ uniform points; this is crucial since the ``identities" of the fragments seen as subsets of $[0,1]$ retain some information (for instance, only neighbouring intervals can merge). More precisely, we can do so using only the process of masses, by exchangeability of $(\zeta_i)_{i\ge 1}$.

If we see the fragmentation $(F(\tau))_{\tau\ge 0}$ from the point of view of Aldous and Pitman in \cite{AlPi1998a}, the cut tree is the genealogy of the fragmentation of a Brownian continuum random tree, and it is natural to try to ask whether one can recover the initial tree $(\sT,d,\mu)$ from $\mathfrak C = (\sC,\delta, \nu, 0)$, or if not, what minimal additional information is necessary. This question has been studied by \citet*{BrWa2017b} and \citet*{AdDiGo2019a} (see also~\cite{AdBrHo2014a} for a partial result). Quite naturally, since the cut tree is constructed from the process of masses only, the locations of the cuts are lost, and reconstruction is impossible without additional information. The main result of \cite{BrWa2017b,AdDiGo2019a} is that these locations is the only information that is lost, and that one can recover $(\sT,d,\mu)$ from $(\sC,\delta,\nu)$ plus this additional information.

Since the fragmentation is binary, for every fragmentation event, there should correspond two points, one in each of the two fragments created. It turns out that these points can be given through their images in $\sC$: the additional information comes in the form of a countable collection of marks in the cut tree, and the only relevant information to us is its distribution conditionally on $(\sC,\delta,\nu)$. Let $\Br(\sC)$ denote the set of branch points of $\sC$. Almost surely, for each $b\in \Br(\sC)$, there are precisely three connected components to $\sC\setminus \{b\}$, and we denote by $\sC_b'$ and $\sC_b''$ the two which are not containing $0$, agreeing that $\nu(\sC_b')>\nu(\sC_b'')$. Let $\bV=\{(V_b',V_b''): b\in \Br(\sC)\}$ be an independent family of random variables such that, for each $b\in \Br(\sC)$, $(V_b',V_b'')$ has distribution 
\begin{equation}\label{eq:def_distribution_hook-points}
\frac{\nu(\cdot \cap \sC_b')}{\nu(\sC_b')} \otimes \frac{\nu(\cdot \cap \sC_b'')}{\nu(\sC_b'')}\,.
\end{equation}
The inverse cut tree transform then goes as follows: there exists a (measurable) map $\Phi$ that associates, to a pair $(\mathfrak C, \bV)$ a measured real tree that is distributed like a Brownian CRT. We will verify that $\CMT(\exc,\bU)$ turns out to be $\Phi(\mathfrak C, \bV)$ for a suitable collection $\bV$, but for now, let us describe the procedure if $\bV$ is given and has the distribution described above (this follows \cite{AdDiGo2019a}).

For $i,j\in \N$, we can recursively identify a collection of branch points in $\sC$, which are meant to correspond to the cut points on the path between $\zeta_i$ and $\zeta_j$ in $\sT$. With this goal in mind, we now define a collection $(C_u, p_u^0, p_u^1)$, $u\in \cU_2$, where $C_u$ is a subtree of $\sC$, and $p_u^0,p_u^1\in C_u$. First set $C_\varnothing = \sC$ and let $p_\varnothing^0=i$, $p_\varnothing^1=j$. Then, given $(C_u, p_u^0, p_u^1)$, and writing $b=p_u^0\curlywedge p_u^1$, let $C_{u0}$ (resp.\ $C_{u1}$) be the one among $\sC_b'$ and $\sC_b''$ which contains $p_u^0$ (resp.\ $p_u^1$). Let $p_{u0}^0=p_u^0$, $p_{u1}^1=p_u^1$ and then define $p_{u0}^1$ (resp.\ $p_{u1}^0$) be the one of $V_b'$ and $V_b''$ that lies in $C_{u0}$ (resp.\ $C_{u1}$). Then for each $n\ge 0$ define 
\[Y_n(i,j)=\sqrt{\frac \pi 2} \cdot \sum_{|u|=n} \nu(C_u)^{1/2}\,.\]
% Then $Y_n(i,j)$ converges almost surely to a finite limit $Y(i,j)$, and the collection $(Y(i,j):i,j\in \N)$. More generally, for $u\in \cU$, one may define 
% \[Y_n^u=\sum_{|w|, u\le w} \mu(\sC_w)\,.\]
% Then $Y_n^u\to Y^u$ almost surely. 
Almost surely for all $i,j$, $Y_n(i,j)\to Y(i,j)$ as $n\to \infty$. 
Then the collection of random variables $(Y(i,j):i,j\in \N)$ has the same distribution as $(\delta(i,j), i,j\in \N)$. Seen as a matrix of pairwise distances, this defines uniquely an isometry class of a random compact real tree, which is a Brownian continuum random tree.  

% {\red 
% \begin{equation}\label{eq:def_cut-point-from_cut-tree}
%   Y^0(i,j) = \lim_{n\to\infty} \sqrt{\frac \pi 2} \cdot \sum_{|u|=n, 0\preceq u} \nu(C_u)^{1/2} 
%   \qquad \text{and} \qquad 
%   Y^1(i,j) = \lim_{n\to\infty} \sqrt{\frac \pi 2} \cdot \sum_{|u|=n, 1\preceq u} \nu(C_u)^{1/2}\,,
% \end{equation}
% are respectively the distances between $i$ and $j$ and the branch point $i\wedge j$; they are used to identify in $\sT$ the image of the branch point $i\wedge j$, to prove that we get the Poisson point process. 
% }

% \begin{alert}HERE BE A BIT CAREFUL: the law of the matrix identifies uniquely a random real tree (a probability measure on a GP-equivalence classes of compact real trees). The annealed measure is that of the CRT. However, the random tree that must be a measurable function of $e,\bU$. 
% \end{alert} 

% subsection the_genealogy_and_the_marked_cut_tree (end)

\subsection{The convex minorant tree as the inverse cut tree transform} % (fold)
\label{sub:distribution_CMT}

% \JF{On pourrait dire qu'en discret, on le voit directement par notre papier sur Prim, et aussi par le papier Minmin-Me}

From the previous considerations, proving that $\CMT(\exc,\bU)$ is indeed a Brownian CRT boils down to verifying that it can be seen as obtained from the inverse cut tree transform from $\mathfrak C$ using a certain collection of points that we will denote by $\{(\beta_b',\beta_b''), b\in \Br(\sC)\}$. Our collection is in part constructed as a measurable function of $\exc$ alone, and the main task consists in verifying that it has indeed the same distribution as $\bV$ defined above in \eqref{eq:def_distribution_hook-points}.

%The points $(\beta_b',\beta_b'')$, $b\in \Br(\sC)$, are constructed from points in $[0,1]$ using $(\exc,\bU)$. 
We start with a canonical exploration of the fragmentation. We construct a process $(S_u,\tau_u, \epsilon_u)_{u\in \cU}$ where $S_u$ is a half-open interval of $[0,1)$, $\tau_u$ is the unique time when there exists $x\in [0,1]$ such that $S_u=I^{\tau_u}(x)$ (that is the interior of $S_u$ is a connected component of $[0,1)\setminus Z^{-\tau_u}$); furthermore, writing $\ell_u=|S_u|$, $\epsilon_u$ is a continuous function on $[0,\ell_u]$ with $\epsilon_u(0)=\epsilon_u(\ell_u)=0$ and $\epsilon_u(r)>0$ on $[0,\ell_u]$. It will also be convenient to write $a_u=\inf S_u\in S_u$. The precise order in which the intervals and times are associated with the elements of $\cU$ is key to control the independence structure which turns out to be crucial. 

We first set $S_\varnothing = [0,1)$, $\tau_\varnothing =0$, and $\epsilon_\varnothing = \exc$; we then have $\ell_\varnothing=1$ and $a_\varnothing=0$. For $x\in [0,1)$ and $t\ge 0$, let $R_t(x)=\sup I^t(x)$. The process $(R_t(a_\varnothing))_{t\ge \tau_\varnothing}=(R_t(0))_{t\ge 0}$ has countably many negative jumps. We let $\ell_1>\ell_2>\dots\ge 0$ denote their ranked sizes (in absolute value); then $\sum_i \ell_i = 1$ almost surely. For each $i\ge 1$, we let $\tau_i$ be the unique $t\ge \tau_\varnothing=0$ with $R_{t-}(a_\varnothing)-R_t(a_\varnothing)=\ell_i$, and define $S_i=[R_{\tau_i}, R_{\tau_i-})$; one then has $a_i=R_{\tau_i}$. We then let $\epsilon_i:[0,\ell_i]\to \R_+$ be defined for $r\ge 0$ by 
\[\epsilon_i(r)=e^{\tau_i}(a_i+r) \I{0\le r\le \ell_i}\,.\]

Now, for each $u\in \cU$, given $S_u$ and $\tau_u$, let $(\tau_{ui},\ell_{ui})_{i\ge 1}$ denote the times $\tau_{ui}\ge \tau_u$ and sizes of the jumps of the process $(R_t(a_u))_{t\ge \tau_u}$ sorted in such a way that $\ell_{u1}>\ell_{u2}>\dots \ge 0$. Write $S_{ui}=[R_{\tau_{ui}}(a_u), R_{\tau_{ui}-}(a_u))$, $a_{ui}=\inf S_{ui}$, $\ell_{ui}=|S_{ui}|$ and define $\epsilon_{ui}:[0,\ell_{ui}]\to \R_+$ for $r\ge 0$ by 
\[\epsilon_{ui}(r)
=e^{\tau_{ui}}(a_{ui}+r) \I{r\le \ell_{ui}} 
= \epsilon_u^{\tau_{ui}-\tau_u}\Bigg(\sum_{j\ge 1} \ell_{uj} \I{a_{uj}<a_{ui}}+r\Bigg)\I{0\le r\le \ell_{ui}}\,.\]

Let $\cF_\varnothing$ be the sigma-algebra generated by $(R_t(0))_{t\ge 0}$. Then $(S_i, \ell_i, \tau_i, a_i)_{i\ge 1}$ is $\cF_\varnothing$-measurable while, conditionally on $\cF_\varnothing$, the $(\epsilon_i)_{i\ge 1}$ are independent Brownian excursions of durations $\ell_1>\ell_2>\dots \ge 0$. More generally, let $\cF_u$ be the sigma-algebra generated by $\{(R_t(a_v))_{t\ge \tau_v}, v\preceq u\}$. Then $(S_{vi},\ell_{vi}, \tau_{vi}, a_{vi})_{v\preceq u, i\ge 1}$ is $\cF_u$-measurable while, conditionally on $\cF_u$, the functions $(\epsilon_{ui})_{i\ge 1}$ are independent Brownian excursions of durations $\ell_{u1}>\ell_{u2}>\dots\ge 0$.

The recursive exploration $(S_u, \tau_u, \epsilon_u)_{u\in \cU}$ we have just defined yields a canonical recursive spinal decomposition of the cut tree $\sC$; by canonical we mean that the random points that are used are constructed from $\exc$ only. We say that a point $s\in [0,1]$ has an image $x\in \sC$ if $\Pi_t(s)=\overline{\{j\in \N:\zeta_j \sim_t s\}}$ decreases to the singleton $\{x\}$ as $t\to\infty$. We let $\sC_\varnothing=\sC$ and $b_\varnothing=0$. Working towards the definition of $(\beta_b',\beta''_b)$, $b\in \Br(\sC)$, we start by defining a collection $\eta_u$, $u\in \cU$. In the following, for $x,y\in \sC$, $\llb x,y\rrb_\sC$ denotes the range of the unique geodesic in $\sC$ between $x$ and $y$.

\begin{lem}\label{lem:images_marks}With probability one, the points $(a_u)_{u\in \cU}$ have images in $\sC$ that we denote by $(\eta_u)_{u\in \cU}$. They are defined inductively and satisfy:
\begin{compactitem}[\textbullet]
  \item $\eta_\varnothing$ is the image of $a_\varnothing$;
  \item given $\sC_u$ and $\eta_u\in \sC_u$ the points $b_{ui}$ are the points of $\llb b_u, \eta_u\rrb_\sC$ at distance $\int_{\tau_u}^{\tau_{ui}} |I^t(a_u)|dt$ from $b_u$;
  \item $\sC_{ui}$ is the subtree of $\sC_u\setminus \{b_{ui}\}$ which contains neither $b_u$ nor $\eta_u$;
  \item $\eta_{ui}$ is the image of $a_{ui}$ in $\sC$, which turns out to be in $\sC_{ui}$.
\end{compactitem}
Furthermore, the family $(\eta_u)_{u\in \cU}$ is independent and for each $u\in \cU$, $\eta_u$ has distribution $\nu(\,\cdot \, \cap \sC_u)/\nu(\sC_u)$.
\end{lem}

% \begin{alert}
% Attention with the $\llbracket x,y\rrbracket$; we need something else since this is in the cut tree... $\boldsymbol [\!\boldsymbol[x,y \boldsymbol ]_\sC$ versus $[a,b]$, or $[[a,b]]$, $\langle x,y\rangle$
% except for $b_\varnothing$, the $b_u$ are all branch points; $b_{ui}$ is the common ancestor of $\eta_u$ and $\eta_{ui}$ 
  
% $\sC_u$ is the subtree of $\sC$ induced by the $\{i\in \N: \zeta_i\in S_u\}$ 

% $i\wedge_\sC j$
% \end{alert}

\begin{proof}The proof is by induction. It is proved in \cite{Bertoin2000a} that the process $(|I^t(0)|)_{t\ge 0}$ has the same distribution as the process $(|I^t(\zeta_1)|)_{t\ge 0}$. Therefore with $\Pi_t(0):=\overline{\{i\in \N: \zeta_i\in I^t(0)\}}$, we have almost surely $\sup\{\delta(i,j): i,j\in \Pi_t(0)\}\to 0$ as $t\to\infty $ so that there is a limit point that we denote by $\eta_\varnothing$ such that $\Pi_t(0)\to \{\eta_\varnothing$\}; by definition $\eta_\varnothing$ is the image of $a_\varnothing$ in $\sC$. It also follows that $\eta_\varnothing$ has distribution $\nu$ in $\sC$, since $1$, the image of $\zeta_1$ in $\sC$, does. The points $b_i$, $i\ge 1$, are precisely the branch points of $\sC$ along the segment $\llb 0,\eta_\varnothing\rrb_\sC$, sorted in the decreasing order of the masses $\nu(\sC_i)=\ell_i$ of the subtrees of $\sC$ hanging from the segment. 

Observe now that for $u\in \cU$ and $i\ge 1$, conditionally on $\cF_u$, the process $(I^{\tau_{ui}+t}(a_{ui}))_{t\ge 0}$ is precisely the process of masses of the fragment containing $0$ in the fragmentation of the excursion $\epsilon_{ui}$. As a consequence, the image $\eta_{ui}$ of $a_{ui}$ is well-defined. Furthermore, the distribution of $\eta_{ui}$ is the rescaled mass measure $\nu$ in the image of $S_{ui}$ in $\sC$, which is precisely $\sC_{ui}$. Finally, conditionally on $\cF_u$, the functions $(\epsilon_{u_i})_{i\ge 1}$ are independent, and so are the $(\eta_{ui})_{i\ge1}$: for any collection of bounded continuous functionals $(f_i)_{i\ge1}$, we have
\[\bE\Bigg[\prod_{i\ge 1} f_i(\eta_{ui})~\bigg|~\cF_u\Bigg]=\prod_{i\ge 1} \int_{\sC_{ui}} f_i(x_i) \frac{\nu(dx_i)}{\nu(\sC_{ui})}\,.\]
The claim follows by induction. 
\end{proof}
\begin{rem}Observe that, $\sC_u$ is the complete subtree of $\sC$ induced by $\{i\in \N: \zeta_i\in S_u\}$. Furthermore, except for $b_\varnothing$, the $b_u$ are all branch points in $\sC$; more precisely $b_{ui}$ is the common ancestor of $\eta_u$ and $\eta_{ui}$, namely $b_{ui}=\eta_u \curlywedge \eta_{ui}$.
\end{rem}

The collection $(\eta_u)_{u\in \cU}$ only provides part of the marks we shall need in the cut tree $\sC$. The remaining marks are the images of the random points constructed using the sequence of uniform random variables $\bU$, and which are associated to the local minima of $\exc$. 

\begin{lem}\label{lem:vertices_branchpoints}There is a one-to-one correspondence between the local minima of $\exc$ and the branch points of the cut tree $\sC$: every branch point of $\sC$ is of the form $\eta_u\curlywedge \eta_{ui}$ for $u\in \cU$ and $i\ge 1$, and the corresponding local minimum is $a_{ui}$.
\end{lem}
\begin{proof}
For each $t\in \sL(e)$, let $t_0=0<t_1<t_2< \dots <t_i=t$ be the vertices of the convex minorant of $\exc$ on the interval $[0,t]$. Let $z_0=1>z_1>\dots> z_k$ and $\gamma_1<\gamma_2<\dots< \gamma_k$ be the corresponding intercepts and slopes. Then, at time $\gamma_k$ the interval $[t_{k-1},z_k)$ is split into the pair $[t_{k-1},t_k)$, $[t_k,z_k)$. Let $j_1=\inf\{j\in \N: \zeta_{j_1}\in [t_{k-1},t_k)\}$ and $j_2=\inf\{j\in \N: \zeta_{j_2}\in [t_k,z_k)\}$. To make the correspondence more explicit, we exhibit the two points $\eta_u$ and $\eta_{ui}$ in $\sC$ such that the branch point corresponding to $t$ is $\eta_u\curlywedge \eta_{ui}=j_1 \curlywedge j_2$. It shall be noted that the branch point corresponding to $t=t_i$ is not the image of $t$ in the cut tree $\sC$, the latter being almost surely the leaf $\eta_{ui}$ that we will exhibit. The path to follow in $\cU$ is given by the convex minorant. Let $i_1\ge 1$ be the unique index such that $\tau_{i_1}=\gamma_1$; then, let $i_2$ be the unique index such that $\tau_{i_1i_2}=\gamma_2$, and so on which yields a point $u=i_1i_2\dots i_{k-1}$ with $\tau_u=\gamma_{k-1}$, and $a_u=t_{k-1}$. Finally, let $i$ be the unique index such that $\tau_{ui}=\gamma_k$; then we have $a_{u}=t_k=t$ while $z_k=R_{\tau_{ui}-}$. The images $\eta_u$ and $\eta_{ui}$ of $t_u$ and $t_{ui}$ in $\sC$ are such that $\eta_u \curlywedge \eta_{ui}$ is the branch point $j_1\curlywedge j_2$.
  
Conversely, the sequence of sets $\{\llb 0, \eta_u\rrb_\sC: |u|\le n\}$, $n\ge 1$, increases to $\sC$ and thus exhausts all the branch points. In particular, every branch point $b$ of $\sC$ is of the form $\eta_u\curlywedge \eta_{ui}$ for some $u\in \cU$ and $i\ge 1$. Now, for such a branch point, $a_{ui}\in [0,1]$ is the local minimum of $\exc$ that separates the points from $[a_u,a_{ui})$ from $S_{ui}$ at time $\tau_{ui}$. 
\end{proof}

Finally, we complete the definition of the set of marks in the cut tree $\sC$. Consider a branch point $b$ of $\sC$; by Lemma~\ref{lem:vertices_branchpoints}, it is of the form $\eta_u\curlywedge \eta_{ui}$ for some $(u,i)\in \cU\times \N$ and $a_{ui}$ is the corresponding local minimum. Recall now the join point $\ju(a_{ui})$ associated to $a_{ui}\in \sL$ (Remark~\ref{rem:points_absolute} on page~\pageref{rem:points_absolute}). Observe that, by construction, at time $\tau_{ui}$, the two intervals that get separated are $S_{ui}$ to the right, and $[a_u,a_{ui})$, to the left. The subtree of $\sC$ above the branch point $b$ is therefore the completion of $\{i\in \N: \zeta_i\in [a_u, \sup S_{ui})\}$, and two intervals $[a_u,a_{ui})$ and $S_{ui}$ correspond to the two subtrees of $\sC$ above the branch point $b$, that we previously denoted by $\sC_b'$ and $\sC_b''$. 

Recall that $\cF_u$ is the sigma-algebra generated by $\{(R_t(a_v))_{t\ge \tau_v}, v\preceq u\}$, and that, as a consequence, $a_u, a_{ui}$ and $S_{ui}$ are $\cF_u$-measurable. By Lemma~\ref{lem:images_marks}, conditionally on $\cF_u$, the image $\eta_{ui}$ of $a_{ui}$ in $\sC$ is distributed like $\nu(\cdot \cap \sC_{ui})/\nu(\sC_{ui})$. Let $(\beta_b',\beta_b'')\in \sC_b'\times \sC_b''$ be the pair of points formed by $\eta_{ui}$ and the image of $\ju(a_{ui})$ in $\sC$ (which a.s.\ exists since $\ju(a_{ui})$ is uniform in $[a_u,a_{ui})$). Then, conditionally on $\cF_u$, and by Lemma~\ref{lem:images_marks}, the collection $(\beta_b',\beta_b'')$, $b\in \Br(\sC)$ has the same distribution as $\bV$:

\begin{lem}\label{lem:full_dist_marked_cuttree}
The marked cut tree $(\sC, \{\beta_b',\beta_b'': b\in \Br(\sC))$ is such that:
\begin{compactenum}[i)]
    \item $\{(\beta_b',\beta_b''): b\in \Br(\sC)\}$ are independent conditionally on $\sC$, and 
    \item for each $b$, $(\beta_b',\beta_b'')$ are independent random variables with distribution $\nu(\,\cdot\, \cap \sC'_b)/\nu(\sC_b')\otimes \nu(\,\cdot\, \cap \sC_b'')/\nu(\sC''_b)$.
\end{compactenum}
\end{lem}

The points $(\beta'_b,\beta''_b)$, $b\in \sC$, now being defined, we are ready to verify that $\Phi(\mathfrak C, \bbeta)$ and $\CMT(\exc,\bU)$ are almost surely isometric. The arguments above should already make this pretty clear: indeed, for each $b\in \Br(\sC)$, the set of marks $\{\beta_b',\beta_b''\}$ is precisely the image in $\sC$ the set of points $\{\ju(a_{ui}), a_{ui}\}$ which are identified at time $\tau_{ui}$. To make this formal, fix any $i,j\in \N$, and consider $Y_n(i,j)$ and $\hat d_n(\zeta_i,\zeta_j)$. The choice of the marks $(\beta_b',\beta''_b)$, $b\in \Br(\sC)$, is precisely made so that, for every $n\ge 1$, sorting the sets $\{\nu(C_u), |u|=n\}$ and $\{|\Pi_u|, |u|=n\}$ in decreasing order yields the same sequence, and therefore
\[Y_n(i,j)= \sqrt{\frac \pi 2} \sum_{|u|=n} \nu(C_u)^{1/2} = \sqrt{\frac \pi 2}\sum_{|u|=n} |\Pi_u|^{1/2} = \hat d_n(\zeta_i,\zeta_j)=d_n(\zeta_i,\zeta_j)\,,\]
where the last step follows from Proposition~\ref{pro:distance_pairs}.
Taking the limit as $n\to\infty$, this implies that, for each $k\ge 1$ the metric spaces $(\{\zeta_i, 1\le i\le k\}, d)$ and $([k], Y)$ are isometric (with the correspondence $(i,\zeta_i)$, $i\in [k]$). Since $(\{\zeta_i, 1\le i\le k\},d)$ increases to $\CMT(\exc,\bU)$ (Proposition~\ref{pro:length_measure}), the claim follows by taking the limit as $k\to\infty$.

Finally, we are ready to prove Theorem~\ref{thm:additive_coalescent} which shows that the convex minorant tree provides a coupling between the two classical constructions of the additive coalescent by Aldous \& Pitman \cite{AlPi1998a} on the one hand, and Bertoin \cite{Bertoin2000a} on the other. Let $\cP=\{(\pi(x),\slo(x)): x\in \sL(e)\}$. 

\begin{proof}[Proof of Theorem~\ref{thm:additive_coalescent}]Observe that, by Lemma~\ref{lem:vertices_branchpoints}, with probability one, all the local minima of $\exc$ are of the form $\eta_u\curlywedge \eta_{ui}$ defined in Lemma~\ref{lem:cut_point-time}, and therefore, almost surely, $\cP=\{(\pi(a_{ui}), \tau_{ui}): u\in \cU,i\in \N\}$. Since, a.s.\ for all $u\in \cU$, we have $\ell_u>0$, this can equivalently be put as $\cP=\{(\pi(\kappa(\zeta_i,\zeta_j)), \tau(\zeta_i,\zeta_j)): i\ne j \in \N\}$. From there, the claim is an easy consequence of Theorem~16, and Corollaries~17-18 of \cite{AdDiGo2019a} (it is even simpler since we do not need to infer the $\tau(\zeta_i,\zeta_j)$ from the cut tree, they can be read directly from the fragmentation). 
\end{proof}

% subsection the_aldous_pitman_fragmentation_and_bertoin_s_representation (end)

% subsection the_recursive_convex_minorant (end)

\section{Compactness of the Brownian parabolic tree $\CMT(X, \bU)$}
\label{sec:compactness}
%!TEX root = MST_brownian.tex

In this section, we prove the compactness of $\CMT(X,\bU)$ constructed in Section~\ref{sub:recursive_convex_minorants_BPT}. The completeness is plain from the definition and we only need to verify that $(\sM,d)$ is totally bounded.

We will proceed by controlling the growth of a well-chosen sequence of subspaces that increase to $\sM$ using a chaining argument. This leads us to a process that is reminiscent to a certain version of Prim's exploration at the discrete level, and that also turns out to be crucial in the calculation of the Hausdorff dimension (see Section~\ref{sec:mass_hausdorff}). The general strategy is inspired from the arguments of \citet{CuHa2017a} for the compactness of trees constructed by aggregation of segments. 

\subsection{The growth process} % (fold)
\label{sub:the_growth_process}

In the entire section, we consider the process $X$, and the random variables $Z^\lambda=Z^\lambda(X)$ refer to this case. We may see the metric space $(\sM,d)$ as obtained from the coalescent process induced by $Z^\lambda$ on $\R_+$, which turns out to be the standard multiplicative coalescent \cite{BrMa2015a,Armendariz2001}. In this process, fragments only merge by pairs, but obtaining fine quantitative estimates is delicate since for any $\lambda \in \R$ and $h>0$, $Z^{\lambda}\setminus Z^{\lambda+h}$ is a.s.\ not contained in any compact interval. We shall thus track a single connected component as $\lambda$ increases.

\begin{rem}The most natural choice of a connected component to track is the largest one, but this leads to some inconvenient conditioning. One could also track a connected component containing a fixed point (at the discrete level), but, without any additional structure, such a node must be uniformly random, and thus the corresponding component would be too small to lead to anything interesting. Here, the structure imposed by the representation on $\R_+$ allows us to track any fixed point \emph{among the ones that do matter} even though they are a negligible for the mass measure (that is, the nodes $v_i$ for $\epsilon n^{2/3}\le i \le C n^{2/3}$ for constants $0<\epsilon<C$).
\end{rem}

For each $\lambda\in \R$, let $L_\lambda=\sup Z^\lambda \cap [0, 1)$ and $R_\lambda= \inf Z^\lambda \cap (1,\infty)$, and define $H_\lambda=[L_\lambda, R_\lambda)$. So, up to inclusion of the left-most point, $H_\lambda\subseteq \R_+$ is the interval of $\R_+\setminus Z^\lambda$ which contains the point~$1$; since $1\not\in \bigcup_\lambda Z^\lambda$ with probability one, this is well-defined for all $\lambda\in \R$. 
We have the following asymptotics, whose proofs are found in Section~\ref{sub:the_position_of_the_connected_component}.

\begin{lem}\label{lem:position_Hlambda}
There exist constant $c>0$ and $x_0,\lambda_0$ such that, for all $x>x_0$ and $\lambda\ge \lambda_0$ we have 
\[\pc{L_\lambda \ge \tfrac x {\lambda^2}} \le e^{-cx}
\qquad \text{and} \qquad \pc{|R_\lambda-2\lambda| > 1} \le e^{-c\lambda}\,.\]
\end{lem}

It follows that $H_\lambda \uparrow (0,\infty)$ as $\lambda \to \infty$, so that $H_\lambda$ provides a suitable increasing family of subspaces of $\sM$.
So for any $x\in \R_+\setminus \{0,1\}$, we let $\lambda(x) = \inf\{\lambda: x\in H_\lambda\}$ be the time at which $x$ joins the connected component containing $1$. The intervals that join $H_\lambda$ play a different role depending on whether they lie to the left or to the right, and we define $\Lambda^\ssp=\{\lambda \in \R: R_\lambda> R_{\lambda-}\}$ and $\Lambda^\ssm:=\{\lambda \in \R: L_\lambda<L_{\lambda-}\}$; then $\Lambda^\ssp$ and $\Lambda^\ssm$ are both countable. Furthermore, they are almost surely disjoint; this is because the standard multiplicative coalescent is binary \cite{Aldous1997}, and could also be proved directly from the representation with $X$ (see the proof of Proposition~\ref{prop:large_merges}). We let $\Lambda=\Lambda^\ssp\cup \Lambda^\ssm$. 

For $x,y\in \R_+$, we let $x\leftrightarrow y$ if $\lambda(x)=\lambda(y)$, and let $q_\lambda$, $\lambda \in \Lambda$, be the equivalence classes of this relation. For $\lambda\in \Lambda^\ssm$, we have $q_\lambda=[L_\lambda, L_{\lambda-})$, while $q_\lambda=[R_{\lambda-}, R_\lambda)$ for $\lambda\in \Lambda^\ssp$. Because of this, $(q_\lambda)_{\lambda\in \Lambda^\ssm}$ and $(q_\lambda)_{\lambda\in \Lambda^\ssp}$ respectively define partitions of $(0,1)$ and $(1,\infty)$ into countably many disjoint intervals. 
%\JF{y a du flou autour du fait que les $q_{\lambda}$ sont ouverts/fermés, d'un coté ou de l'autre. Est-ce une partition, ou une partition qui évite les points de $\cup Z^\lambda$? Le truc qui n'est pas bien clair, c'est que vaut $\lambda(x)$ lorsque $x$ est l'un des points de $\cup Z^\lambda$ (un des points de sauts de $H_{\lambda}$). Ce que je comprends, c'est que la définition de $\lambda(x)$ avec l'inf pousse à définir différemment qui est dans $q_{\lambda}$ différemment à droite et à gauche de 1; mais je crois que du coup, c'est la même remarque que la précédente, où je dis que je ne comprends pas.}.

For an interval $I\subset \R_+$, we let $\sM|_I$ be the \emph{intrinsic} metric space induced by $d$ on $I$: this is the metric space $(I,d_I)$ where $d_I(x,y)=d(x,y)$ if $\llb x,y\rrb \subseteq I$ and $d_I(x,y)=+\infty$ otherwise. So in general, $\sM|_I$ might be disconnected. Let $\sM_\lambda:=\sM|_{H_\lambda}$; then $\sM_\lambda$ is connected for all $\lambda\in \R$. For $\lambda\in \Lambda^\ssp$ (resp.\ $\Lambda^\ssm$), we also let $\sT^\tsp_\lambda:=\sM|_{q_\lambda}$ (resp.\ $\sT_\lambda^\tsm:=\sM_{q_\lambda}$). Quite plainly, and up to the metric completion, the metric space $\sM$ is obtained by combining the $\sT^\tsp_\lambda$, $\lambda\in \Lambda^\ssp$, and $\sT^\tsm_\lambda, \lambda\in \Lambda^\ssm$, using the identifications performed during the construction (using the random points constructed from $\bU$). This process actually turns out rather agreeable: the metric spaces $\sT^\tsp_\lambda$ and $\sT^\tsm_\lambda$ are easy to understand because they are small as $\lambda \to \infty$, and the way they are put together is also easy to control. Informally, the dynamics as $\lambda$ increases are as follows: 
\begin{itemize}
	\item at time $\lambda\in \Lambda^\ssp$, $\sT^\tsp_\lambda$ merges with $\sM_{\lambda-}$ by identifying $\inf q_\lambda$ with a uniform point in $H_{\lambda-}$;
	\item at time $\lambda\in \Lambda^\ssm$, $\sM_{\lambda-}$ connects with $\sT^\tsm_\lambda$ by identifying $L_{\lambda-}$ with a uniform point in $\sT^\tsm_\lambda$.  
\end{itemize}

The $(\sM_{k^3})_{k\ge 1}$ is the convenient sequence of subspaces of $\sM$ that we mentioned before. The following decomposition which takes advantage of these dynamics will be useful. What matters for now is the global picture, we will fill out the details later on.
\begin{itemize}
	\item \emph{The annuli of forests to the right.} 
	For any $k\ge 1$ define $\Lambda_k:=\{\lambda \in \Lambda^\ssp: k^3<\lambda \le (k+1)^3\}$, and let $\sF_k$ be the intrinsic metric space induced by $\sM$ on $(R_{k^3}, R_{(k+1)^3}]$. For each $k\ge 1$, $\sF_k$ is a forest consisting of  infinitely many connected components obtained when only the identifications within $(R_{k^3},R_{(k+1)^3}]$ are performed. Our aim is to bound the maximum diameter of the connected component $\sF_k$ in order to control the worst case accumulation of length when putting all the $\sF_k$ together. Formally, for $\lambda,\lambda'\in \Lambda_k$ we write $\lambda \equiv_k \lambda'$ if $R_{\lambda-}\wedge R_{\lambda'-} > R_{k^3}$, which implies that $\sT_\lambda^\tsp$ and $\sT_{\lambda'}^\tsp$ are connected within $\sF_k$. We will prove that, almost surely, every equivalence class of this relation is finite. So for each $\lambda\in \Lambda_k$, we may define $\rho_\lambda:=\min\{R_{\lambda'-}: \lambda' \equiv_k \lambda\}$ as the leftmost point of the connected component containing $\sT^\tsp_\lambda$ within $\sF_k$. More generally, the equivalence relation $\equiv_k$ naturally extends as follows: for $x,y\in (R_{k^3}, R_{(k+1)^3}]$ we let $x\sim_k y$ if $x\in q_\lambda$, $y\in q_{\lambda'}$ and $\lambda \equiv_k \lambda'$. Defining the diameter of a potentially disconnected metric space as the supremum of the diameters of its connected components, we therefore have $\diam(\sF_k)=\sup\{d(x,y): x,y\in (R_{k^3},R_{(k+1)^3}], x\equiv_k y\}$. 
 
	\item \emph{The chain of beads to the origin.} For any $\lambda\in \R$, we let $\sP_\lambda$ denote the intrinsic metric space induced by $\sM$ on $[0,L_\lambda)$. This metric is almost surely connected and has the structure of a ``string of beads'' that we now describe. For $a\in \R$, let $\Lambda_a^\ssm:=\Lambda^\ssm\cap (a,\infty)$.
	Almost surely for any $a\in \R$, $\Lambda^\ssm_a$ only contains finitely many points in any compact interval of $(0,+\infty)$, so that we may enumerate its elements in increasing order as $(\lambda_i)_{i\ge 1}$ (Lemma~\ref{lem:no_exception_convex}). 
	The metric space $\sP_\lambda$ is obtained by putting together the metric spaces $\sT^\tsm_{\lambda_j}$, $j\ge 1$, into a chain by connecting $\sT^\tsm_{\lambda_i}$ to a uniform random point in $\sT^\tsm_{\lambda_{i+1}}$ for each $i\ge 1$.  
\end{itemize}

By giving an estimate of the extent of $H_\lambda$, Lemma~\ref{lem:position_Hlambda} provides an effective way to control the locations of the ``gluing points'' which lies at the core of the proofs of the compactness and of the computation of the Hausdorff dimensions. The contribution of $\sP_\lambda$ is easily treated separately, and the crucial steps consists in controlling the diameters of the $\sF_k$, $k\ge 1$. 

\begin{prop}[Diameter of annuli forests]\label{pro:bound_annulus}There exists $k_0\in \N$ such that for all $k\ge k_0$ we have
\[\p{ \diam(\sF_k) > k^{-3/2}} \le 11 k^{-5/4}\,.\]
\end{prop}

\begin{prop}[Diameter of the string of beads]\label{pro:left-end}Almost surely, 
 \[\lim_{\lambda \to \infty }\diam(\sP_\lambda) = 0\,.\]
\end{prop}

Taking Propositions~\ref{pro:bound_annulus} and~\ref{pro:left-end} for granted for now, the proof of compactness is then straightforward.

\begin{prop}[Compactness of $\sM$]\label{pro:space_compact}The metric space $(\sM,d)$ is almost surely compact.
\end{prop}
\begin{proof}By Proposition~\ref{pro:bound_annulus} and the Borel--Cantelli lemma, with probability one, there exists an almost surely finite random variable $K$ such that $\diam(\sF_k)\le k^{-3/2}$ for all $k\ge K$. In particular, for all $k\ge K$, 
\[\dH(\sM_{k^3}, \sM) \le 2 k^{-1/2} + \diam(\sP_{k^3})\,.\]
Fix any $\epsilon>0$. By Proposition~\ref{pro:left-end}, $\diam(\sP_{k^3})<\epsilon/3$ for all $k$ large enough; it follows that there exists $k\ge K$ large enough such that $\dH(\sM_{k^3}, \sM)< \epsilon/2$. Recall that, for any fixed $\lambda\in \R$, the restriction $\sM_\lambda$ of $\sM$ to $H_\lambda$ is almost surely compact by absolute continuity with the Brownian continuum random tree. So there exists a cover of $\sM_{k^3}$ by finitely many balls of radius $\epsilon/2$; increasing the radius of each ball to $\epsilon$ yields a finite cover of $\sM$. We have thus proved that $\sM$ is totally bounded. Since it is complete by definition, it is compact.
 \end{proof} 
 
The remainder of the section is devoted to the proof of Propositions~\ref{pro:bound_annulus} and~\ref{pro:left-end}. We first prove Lemma~\ref{lem:position_Hlambda} in Section~\ref{sub:the_position_of_the_connected_component}. The forests $\sF_k$ are made of the trees $\sT^\tsp_\lambda$, for $\lambda\in \Lambda_k$. For all $k\ge 1$, $\Lambda_k$ is infinite, which causes some difficulties. Still, we expect that, for large $k$, the components $\sT_\lambda^\tsp$ with $\lambda \in \Lambda_k$ should be rather small; Section~\ref{sub:distances_in_swallowed_components} deals with the question of uniform bounds on distances in the $\sT^\tsp_\lambda$ in terms of the lengths $|q_\lambda|$. We then obtain in Section~\ref{sub:statistics_small_components} the relevant statistics about the connected components $\sT_\lambda^\tsp$ for $\lambda\in \Lambda_k$ which includes information about the lengths $|q_\lambda|$ and  their diameters. In Section~\ref{sub:contribution_of_intervals} we put together all the pieces and prove Proposition~\ref{pro:bound_annulus} which essentially says that the diameter of $\sF_k$ is comparable to the maximum diameter of the $\sT_\lambda^\tsp$, $\lambda \in \Lambda_k$. Finally, we prove Propostition~\ref{pro:left-end} in Section~\ref{sub:string_of_beads}.

\subsection{The position of the component containing $1$: Proof of Lemma~\ref{lem:position_Hlambda}} % (fold)
\label{sub:the_position_of_the_connected_component}

Recall that for a continuous process $\omega = (\omega_t)_{t\ge 0}$, we let $\underline \omega$ and $\overline \omega$ denote respectively the running infimum and supremum processes: $\underline \omega_t :=\inf\{\omega_s : 0\le s\le t\}$ and $\overline \omega_t = \sup\{\omega_t: 0\le s \le t\}$. 

Recall that $(W_t)_{t\ge 0}$ denotes a standard Brownian motion. We will use repeatedly the following simple fact (see, e.g., \cite{KaSh1988} page 96, consequence of the fact that $\overline W_t$ has same law as $|W_t|$, for a fixed $t$): for all $x\ge 0$, we have
\begin{equation}\label{eq:tail_max_bm}
\pc{\overline W_t \ge x} = \pc{\underline{W\!}_t \le -x }\le e^{-x^2/(2t)}\,.
\end{equation}

Let $\lambda\ge 0$. Let us first focus on the upper bound. For any $t\in \R_+$, we have $R_\lambda > t$ if only if there is an excursion of $X^\lambda$ above its running minimum that straddles both $1$ and $t$, that is if $1\sim_\lambda t$. Thus
\begin{align}\label{eq:re-upper-bound1}
\pc{R_\lambda > 2\lambda +1} 
& = \pc{\underline{X}_1^\lambda = \underline X_{2\lambda+1}^\lambda} \notag \\
& = \pc{\underline{X}_1^\lambda \le \underline{X}_{2\lambda+1}^\lambda} \notag \\ 
%& \le \inf_x \{\pc{\underline{X}_1^\lambda \le x} + \pc{\underline{X}_1^\lambda\le \underline{X}_{2\lambda+1}^\lambda, \underline{X}_1^\lambda > x}\} \notag \\ 
& \le \inf_x \{\pc{\underline{X}_1^\lambda \le x} + \pc{\underline{X}_{2\lambda+1}^\lambda > x}\}\,.
\end{align}
A quick inspection of the expected values leads to the choice $x=-\tfrac \lambda 2$. On the one hand $\underline{X}_1^\lambda \ge -\frac 1 2 + \underline{W\!}_1$ so that, by \eqref{eq:tail_max_bm}, we obtain
\begin{align}\label{eq:re-upper-bound2}
\pc{\underline{X\!}^\lambda_1 \le -\tfrac \lambda 2}
 \le \pc{-\tfrac 12 + \underline{W}_1\le -\tfrac \lambda 2}
%\le e^{-(\lambda-1)^2/8} 
\le e^{-\lambda^2/9}\,,
\end{align}
for all $\lambda$ large enough. On the other hand, $\underline{X}_{2\lambda+1}^\lambda \le X^\lambda_{2\lambda+1} = -\lambda - \tfrac 12 + W_{2\lambda+1}$ which implies that 
\begin{align}\label{eq:re-upper-bound3}
\pc{\underline{X}_{2\lambda+1}^\lambda > - \tfrac \lambda 2 }
 \le \pc{W_{2\lambda +1} > \tfrac \lambda 2} \le e^{-\lambda/20}\,,
\end{align}
for all $\lambda$ large enough. Putting together \eqref{eq:re-upper-bound1}--\eqref{eq:re-upper-bound3} completes the proof of the upper bound.

For the lower bound, observe that 
\begin{align}\label{eq:re-lb1}
\p{R_\lambda < 2\lambda-1}
&=\pc{\underline{X}^\lambda_1 > \underline{X}^\lambda_{2\lambda-1}} \notag\\
&\le \inf_x \big\{\pc{X_1^\lambda < x} + \pc{\underline{X}_1^\lambda> \underline{X}^\lambda_{2\lambda-1}, X_1^\lambda\ge x} \big\}\,. 
\end{align}
Considering the fact that $X_1^\lambda = -\tfrac 12 + \lambda + W_1$, we are lead to choosing $x=\tfrac\lambda2$. We have $\pc{X_1^\lambda <\frac \lambda 2} \le e^{-\lambda^2/9}$ for all $\lambda$ large enough. To deal with the second part of the right-hand side of \eqref{eq:re-lb1}, we use Markov's property at time $1$ and observe that for $s\ge 1$,  $X_s^\lambda - X^\lambda_1$ is distributed like $X^{\lambda-1}_{s-1}$: 
\begin{align}\label{eq:re-lb2}
  \pc{\underline X_1^\lambda > \underline{X}_{2\lambda -1}^\lambda~|~X_1^\lambda \ge \tfrac \lambda 2} 
  & \le \pc{\underline X^{\lambda-1}_{2\lambda-2} < -\tfrac \lambda 2}\,.
\end{align}
However, since $s\mapsto -\tfrac {s^2}2 +(\lambda-1)s$ is non-negative on $[0,2(\lambda-1)]$, we have
\begin{align*}
	\underline{X}^{\lambda-1}_{2\lambda-2} 
	= \inf\{-\tfrac {s^2}2 + (\lambda -1) s + W_s: s\le 2\lambda -2 \} 
	% & \ge \inf\{-\tfrac {s^2}2 + (\lambda-1)s: s\le 2\lambda -2\} + \underline{W}_{2\lambda -2} 
	\ge \underline{W}_{2\lambda -2}\,,
\end{align*}
so that, by \eqref{eq:tail_max_bm},
\begin{align*}
  \pc{\underline X^{\lambda-1}_{2\lambda-2} < -\tfrac \lambda 2} 
  \le \pc{\underline{W}_{2\lambda -2} < -\tfrac \lambda 2} 
  \le e^{-\lambda/20}\,, 
\end{align*}
for all $\lambda$ large enough. Putting this together with \eqref{eq:re-lb1} and \eqref{eq:re-lb2} yields the lower bound on $R_\lambda$.

Finally, we deal with the lower bound on $L_\lambda$. Observe that 
$\Ec{X^\lambda_s}=-\tfrac {s^2}2 + \lambda s 
%\ge s (\lambda -\tfrac 12) 
\ge s \lambda / 2$ for all $s\in [0,1]$ provided that $\lambda$ is large enough. 
So writing $t=a/\lambda^2$, we have, since $\underline X^\lambda_t\le 0$, 
\begin{align*}
\pc{L_\lambda \ge t} 
= \pc{\underline{X}_t^\lambda > \underline{X}_1^\lambda}
%&\le \pc{\underline{X}_t^\lambda > \underline{X}_1^\lambda, X^\lambda_t \ge \tfrac \kappa{4\lambda} } + \pc{X^\lambda_t < \tfrac \kappa {4\lambda}}\\
&\le \pc{\exists s\in [t,1]: X^\lambda_s \le 0}\\
&\le \pc{\exists s\in [t,1]: W_s \le - \lambda s / 2}\,.
\end{align*}
Observe that $\{(s,y): s\in [t,1], y\le -\lambda s/2\} \subseteq \bigcup_{i\ge 1} \{(s,y): s\le (i+1)t, y\le -\lambda i t/ 2\}$. So writing $\tau(-x):=\inf\{s\ge 0: W_s<-x\}$, we have
\begin{align*}
\pc{\exists s\in [t,1]: W_s \le - \lambda s/ 2} 
 &\le \sum_{i\ge 1} \pc{\tau(-\lambda it/2) \le (i+1) t}\\
 &\le \sum_{i\ge 1} \pc{\underline{W}_{(i+1)t} \le -\lambda i \tfrac t2 }
%& \le \sum_{i\ge 1} \exp(-i a /16)\\
\end{align*}
which, using \eqref{eq:tail_max_bm} is at most $e^{-a/20}$ for all $a$ large enough. This completes the proof of Lemma~\ref{lem:position_Hlambda}.
% The event in the right-hand side above corresponds to sample path large deviations for Brownian motion and its probability can be estimated thanks to Schilder's theorem \cite[Theorem 5.2.3, p. 185][]{DeZe1998}, a special case of the Freidlin--Wentzell theorem for diffusions:
% \begin{align*}
% \pc{\exists s\in [\frac \kappa {\lambda^2},1]: \lambda^{-1} W_s \le \tfrac s 2}
% & \le \exp\left( \lambda^2 \inf\{\int_0^1 |\phi|^2\}\right)\,,
% \end{align*}
% where $\phi$ ranges in the set of absolutely continuous functions on $[0,1]$ that hit the region $\{\}$ 

\subsection{Distances in small aggregated components} % (fold)
\label{sub:distances_in_swallowed_components}

In this section, we are interested in the intrinsic metric space induced by $d$ on $q_\lambda$. Write $D_\lambda = \sup\{d(x,y): x,y\in q_\lambda\}$ for the diameter of this metric space, and for $\zeta_\lambda$ a random variable with uniform distribution in $q_\lambda$, independent of everything else, let $Y_\lambda:=d(l_\lambda, \zeta_\lambda)$ where $l_\lambda=\inf q_\lambda$.

Recall that $X^\lambda_t=W_t-\tfrac {t^2}2 + \lambda t$ and $B^\lambda_t = X^\lambda_t - \underline{X}^\lambda_t$ is the process reflected in the running infimum. We define $m_\lambda=|q_\lambda|$ and $\varepsilon_\lambda:\R_+ \to \R_+$ by 
\begin{equation}\label{eq:e_alpha}
	\varepsilon_\lambda(s):= B^{\lambda-h}(l_\lambda+s) \I{0\le s\le m_\lambda}\,.
\end{equation}
The excursion $\varepsilon_\lambda$ has duration $m_\lambda$ and encodes the metric space $\sT_\lambda^\tsp$ or $\sT_\lambda^\tsm$ supported by $q_\lambda$. For any fixed $\lambda'\in \R$, and $q$ an interval of $\R_+\setminus Z^{\lambda'}$, conditionally on $|q|=\sigma$, the distribution of the excursion of $B^{\lambda'}$ straddling $q$ is given by 
\begin{equation}\label{eq:exp_tilted_excursion}
	\tilde\bn_{\sigma}(A)=
	\frac {\int \I{\omega\in A} \exp(\int_0^\sigma \omega(r)dr) \bn_\sigma(d\omega) }
	{\int \exp(\int_0^\sigma \omega(r) dr) \bn_\sigma(d\omega)}\,
\end{equation}
where $\bn_\sigma$ is the law of a Brownian excursion of duration $\sigma>0$. Consider now, $\lambda\in \Lambda$, which is random. For any $\lambda'<\lambda$, $l_\lambda\in Z^{\lambda'}$; furthermore, as $\lambda'\uparrow \lambda$, we have $\inf\{s>0: B^{\lambda'}(l_\lambda+s)=0\}\uparrow m_\lambda$ and the excursion of $B^{\lambda'}$ starting at $l_\lambda$ converges to $\varepsilon_\lambda$. It follows that, for a bounded continuous functional $\phi$, 
\begin{equation}\label{eq:dist_varepsilon_lambda}
	\Ec{\phi(\varepsilon_\lambda)~|~m_\lambda=\sigma}=\int \phi(\omega) \tilde \bn_\sigma(d\omega)\,.
\end{equation}
We note further that, by the strong Markov property, the excursions $\varepsilon_\lambda$, $\lambda\in \Lambda$ are mutually independent conditionally on $m_\lambda$, $\lambda \in \Lambda$. 
% {\red 
% By construction, this is an excursion in the sense that $e_\alpha(s)>0$ for $s\in q_\alpha^o$ while $e(s)=0$ for $s\in \{\inf q_\alpha, \sup q_\alpha\}$. 
% The law of the excursions is explicit: for any $\alpha\in \cal A$, the excursions $e_\alpha$ are condtionally independent given $(q_\alpha)_{\alpha\in \cal A}$ and 
% \begin{equation}\label{eq:law_e-alpha}
% \tilde \bn_\sigma(A) = \p{\varepsilon_\lambda \in A~|~m_\lambda=\sigma} = \frac{\E{\I{e\in A} \exp(\int_0^\sigma e(s)ds)}}{\E{\exp(\int_0^\sigma e(s)ds)}}.
% \end{equation}
% }

\begin{prop}[Distances in small components]\label{pro:diameter_small}
For any $\epsilon>0$, there exist
\begin{compactenum}[i)]
	\item a sub-Gaussian random variable $D^\star$ such that, if $m_\lambda\in (0,\epsilon]$ then $m_\lambda^{-1/2} D_\lambda \le_{st} D^\star$, and
	\item a random variable $Y^\star$ with $\Ec{1/Y^\star}<\infty$ and such that if $m_\lambda\in (0,\epsilon]$ then $Y^\star \le_{st} m_\lambda^{-1/2} Y_\lambda$.
\end{compactenum}
%%% NIC : I have removed the part with the convergence in distribution
% {\nic Furthermore, conditionally on $m_\lambda=\sigma$, $m_\lambda^{-1/2} Y_\lambda$ converges in distribution and in every $L^p$ as $\sigma\to 0$, to a Rayleigh random variable $R$ with density $xe^{-x^2/2}$ on $\R_+$.}
% % {\red Furthermore, $D^\star_\epsilon,Y^\star_\epsilon$ converge in distribution as $\epsilon\to 0$. }
% % \JF{pas clair: est-ce que tu veux dire que l'on peut construire des processus  $D^\star_\epsilon,Y^\star_\epsilon$ indexé par $\epsilon$ t.q. blabla?}
\end{prop}
\begin{proof}
\emph{i)} Let $(\omega,\bu)\mapsto d_{\omega,\bu}$ a deterministic measurable function which gives the metric  of $\CMT(\omega,\bu)$ for $\bn\otimes d\bu$-a.e. $(\omega,\bu)$, and let $\bar d_{\omega,\bu}=\sup_{x,y} d_{\omega,\bu}(x,y)$ be the corresponding diameter.
From the law of $\varepsilon_\lambda$ in \eqref{eq:exp_tilted_excursion}--\eqref{eq:dist_varepsilon_lambda} and the Cauchy--Schwarz inequality, for any $r\in \R$,
\begin{align}\label{eq:diameter_to_brownian}
\E{e^{r m_\lambda^{-1/2} D_\lambda}~\Big|~m_\lambda=\sigma} \notag
& = \int \exp(r \sigma^{-1/2} \bar d_{\omega,\bu}) \tilde \bn_\sigma(d\omega) d\bu  \notag \\
& = \int \exp(r \sigma^{-1/2} \bar d_{\omega,\bu}) \frac{\exp(\int_0^\sigma \omega(s) ds) \bn_\sigma(d\omega)} {\int \exp(\int_0^\sigma \omega(s) ds) \bn_\sigma(d\omega) } d\bu\notag \\
% & = \frac{\E{\exp(\gamma D_{e,\bu}} \exp(\int_0^\sigma e(s)ds)}{\E{\exp(\int_0^\sigma e(s)ds)}}\\
& \le \frac{\sqrt{\int \exp(2 r \sigma^{-1/2} \bar d_{\omega,\bu}) \bn_\sigma(d\omega) d\bu } \cdot \sqrt{\int \exp(2\int_0^\sigma \omega(s)ds) \bn_\sigma(d\omega)}}{\int \exp(\int_0^\sigma \omega(s)ds) \bn_\sigma(d\omega)}.
% & \le 2 e^{2\gamma^2}\cdot \frac{\phi(-2\sigma^{3/2})^{1/2}}{\phi(-\sigma^{3/2})},
\end{align}
In the first factor above, the metric space is encoded by a Brownian excursion, and is therefore exactly distributed like a Brownian continuum random tree of mass $\sigma$ (Theorem~\ref{thm:limit_mst_surplus} with $s=0$). It follows by Brownian scaling that, under $\bn_\sigma \otimes d\bu$, $\sigma^{-1/2} \bar d_{\omega,\bu}$ is dominated by  $2\|e\|$, twice the supremum of a standard Brownian excursion, which is well-known to be sub-Gaussian (see, for instance, \cite{DuIg1977,Kennedy1976}). On the other hand, observe that $\psi(z):=\int \exp(z\int_0^1 \omega(s)ds) \bn_1(d\omega)=\Ec{\exp(z\int_0^1 e(s)ds)}$, the Laplace transform of Brownian excursion area \cite{Janson2007}, is continuous and positive on any interval $[0,\epsilon]$. It follows that there exist constants $A=A_\epsilon$ and $v>0$ such that, for any $\sigma\in [0,\epsilon]$,
\begin{equation}\label{eq:Dlambda_subgaussian}
	\E{\exp(r m_\lambda^{-1/2} D_\lambda)~\Big|~m_\lambda=\sigma} \le e^{r^2/(2v)} \cdot \sup_{\sigma \in [0,\epsilon]} \frac{\sqrt{\psi(2\sigma^{3/2})}}{\psi(\sigma^{3/2})} \le A e^{r^2/(2v)}\,.
\end{equation}

One then easily constructs (the law of a variable) $D^\star$ by inverse transform. For $\sigma>0$ let $F_\sigma(x):=\pc{D_\lambda \le x~|~m_\lambda=\sigma}$, and $F_\star(x):=\inf\{F_\sigma(x): \sigma \le \epsilon \}$. Then for $U$ a $[0,1]$-uniform random variable, $F_\star$ is the distribution function of a random variable $D^\star:=F_\star^{-1}(U)$ that dominates all the $m_\lambda^{-1/2} D_\lambda\I{m_\lambda\le \epsilon}$. The random variables $F^{-1}_\sigma(U)$ are uniformly sub-Gaussian by \eqref{eq:Dlambda_subgaussian}, and so is $D^\star$. This completes the proof of \emph{i)}. 

The proof of \emph{ii)} about the distance to a random point is similar: we only discuss the adaptation of the arguments in \emph{i)} to bound $\Ec{m_\lambda^{1/2}/Y_\lambda~|~m_\lambda=\sigma}$. Instead of Cauchy--Schwarz, using Hölder's inequality (with exponents $3/2$ and $3$) provides an upper bound similar to \eqref{eq:diameter_to_brownian} where the main term involves the $d_{\omega,\bu}(0,\xi)$ for $\xi$ independent and uniform in $[0,\sigma]$ under $\bn(d\omega)$: this is the distance between two random points in a unit mass Brownian CRT, which is a random variable $Y$ with density $x e^{-x^2/2}dx$ on $\R_+$ \cite{Aldous1991b}, so that $\Ec{Y^{-3/2}}<\infty$. The multiplicative error term is bounded just as above, and the random variable $Y^\star$ is constructed similarly, using the supremum of the distribution functions instead of the infimum. We omit the details. 
%%% NIC: part with the convergence in distribution removed
% {\nic For the last claim, observe that \emph{i)} implies that $m_{\lambda} Y_\lambda$ is tight in every $L^p$. The convergence is straightforward at this point: for any continuous bounded functional $f$, we have, by Brownian scaling for $\bn_\sigma$
% \begin{align*}
% \left|\int f(\omega) \tilde \bn_\sigma(d\omega) - \int f(\omega) \bn_\sigma(d\omega)	\right|
% & \le \|f\| \cdot \int  \left|\frac{\exp\big(\sigma^{3/2} \int_0^1 \omega_1(r)dr\big)}{\Psi(\sigma^{3/2})}-1\right| d\bn_1(d\omega_1)\,,
% \end{align*}
% which tends to zero as $\sigma\to0$ by bounded convergence. 
% }
\end{proof}

\subsection{Statistics of the aggregated components $\sT_\lambda^\tsp$} % (fold)
\label{sub:statistics_small_components}

From the results of the previous section, especially the law of the $\varepsilon_\lambda$, it is crucial to understand the distribution of the $m_\lambda$. In the following we let $\sE$ be the space of continuous excursions, that is, the functions $f\in \cC(\R_+,\R_+)$ with $f(0)=0$ such that there exists $\sigma\in [0,\infty)$ such that $f(s)>0$ for $s\in (0,\sigma)$ and $f(s)=0$ for $s\ge \sigma$. 

Fix any $\lambda\in \R$ and consider the excursions of $X^\lambda$ away from its running infimum $\underline X^\lambda$. For $y\in \R_+$, let $\tau^\lambda_y:=\inf\{s>0: -\underline{X}_s^\lambda>y\}$. Then, $\{(y,\tau^\lambda_y-\tau^\lambda_{y-}): \tau_y^\lambda>\tau_{y-}^\lambda\}$ is a Poisson point process on $\R_+\times \R_+$ of intensity $dy \tilde\varrho^\lambda_{\tau_{y-}}$, where the inhomogeneous measure $\tilde \varrho_x^\lambda$ is defined by (see, e.g., \cite{Aldous1997}, Section 5.2)
\begin{equation}\label{eq:excursion_length_measure}
\tilde \varrho^\lambda_x(l,\infty):=\lim_{\epsilon\to 0} \frac 1 \epsilon \p{\inf\{s>0:X^\lambda_{x+s}\le 0\} > l~\Big|~X_x^\lambda=\epsilon }\,.
\end{equation}

\begin{prop}\label{pro:distribution_R-process}
The process $\{(\lambda, \varepsilon_\lambda): R_{\lambda}>R_{\lambda-}\}$ is a Poisson point process on $\R\times \sE$ of intensity $R_{\lambda-} d\lambda \tilde \varrho^{\lambda}_{R_\lambda-}(d\sigma) \tilde\bn_{\sigma}(d\omega)$.
\end{prop} 
%\JF{la notation $\varrho^{\lambda}_{R_\lambda-}$ n'est pas encore introduite... (elle l'est 10 lignes plus bas)}
\begin{proof}The distribution of the excursions $\varepsilon_\lambda$ conditionally on their duration is known from the previous section, and we only need to deal with the sizes of the jumps of $(R_\lambda)_{\lambda\in \R}$. These are formed by agglomeration of some of the excursion lengths of $B^\lambda$, for $\lambda\in \R$, which are described by the excursion length measure in \eqref{eq:excursion_length_measure}: As $\lambda$ increases, the excursions of $B^\lambda$ away from $0$ merge together until they eventually join the connected component containing $1$. 

We proceed geometrically using the process $X^0$ only. The excursions $\varepsilon_\lambda$ are simply read from $X^0$: for any $\lambda$ such that $R_{\lambda}>R_{\lambda-}$, $\varepsilon_\lambda$ is obtained as 
\[\varepsilon_\lambda(s) = (X^0(R_{\lambda-}+s)-X^0(R_{\lambda-})+s\lambda) \I{R_{\lambda-} + s\le R_\lambda}
\qquad \text{for }s\ge 0.\]
Here, notice that $\varepsilon_\lambda$ is indeed obtained from the agglomeration of countably many excursions of $X^{\lambda -h}-\underline X^{\lambda-h}$, which might be described using the process involving straight lines with slopes $\lambda-h$ (at time $\lambda$, the excursions of interest are those of $X^0$ above the process $s\mapsto \underline X_s^\lambda - \lambda s$). Now, knowing the intensity of jumps of $\tau^\lambda$ for each $\lambda\in \R$, it is routine to deduce the intensity of excursions of jumps of $(R_\lambda)_{\lambda \in \R}$: at time $\lambda$, we always have an excursion, the increase in local time is $d\lambda R_{\lambda-}$, and this gives rise to excursions whose durations are governed by $\tilde \rho_{R_{\lambda}-}^\lambda$. 
% {\red For each $y$ such that $\tau^\lambda_y>\tau^\lambda_{y-}$ define 
% \[f_y:=u\mapsto f_y(r)=X^\lambda(\tau^\lambda_{y-}+r)-X^\lambda(\tau^\lambda_{y-}) \I{\tau^\lambda_{y-}+r\le \tau^\lambda_y}\,.\]}	
\end{proof}

For each $k\ge 1$ we say $\lambda\in \Lambda_k$ is of level $i\ge 0$ and write $\lambda \in \Lambda_{k,i}$ if its duration satisfies $m_\lambda \in [k^{-6-i}, k^{-5-i})$. We define the $M_{k,i}$ total duration (mass) of excursions of level at least $i$ 
\[M_{k,i}=\sum_{j\ge i}\sum_{\lambda \in \Lambda_{k,j}} m_\lambda\,. \]

\begin{lem}[Statistics for fragments $\sT^\tsp_\lambda$, $\lambda\in \Lambda^\ssp$]\label{lem:sizes_fragments-k}For any $k\ge 1$ and $i\ge 0$, there exists an event $B_{k,i}$ of probability at most $8\min\{k^{-1-i/4}, k^{-5/4}\}$ such that outside of $B_{k,i}$ we have 
\begin{compactenum}[i)]
\item the longest excursion: $\sup\{m_\lambda: \lambda \in \Lambda_k\} \le k^{-5}$
\item total duration of excursions of level at least $i$: $M_{k,i} \le \min \{k^{7/2-i/4}, 7k^2 \},$
\item number of excursions of level $i$: $\#\Lambda_{k,i}\le 7 k^{8+i}$,
\item maximum diameter of an excursion of level $i$: $D_{k,i}=\sup\{D_\lambda: \lambda\in \Lambda_{k,i}\}\le k^{-2 -i/4}$.
\end{compactenum}
In particular, $\bigcup_{i\ge 0} B_{k,i}$ occurs with probability at most $10 k^{-5/4}$.
\end{lem}

The proof of Lemma~\ref{lem:sizes_fragments-k} is based on upper bounds on the durations $|q_\lambda|$, $\lambda\in \Lambda_k$. The relevant calculations are simplified if we upper bound the Brownian with parabolic drift by a Brownian with a suitable linear drift. This is why the following is especially useful. In the following, we let $\varrho^{\mu}$ be the excursion length measure for Brownian motion with linear drift $\mu$, $W^\mu:=s\mapsto W_s + \mu s$. We are mostly interested in what happens at large positions, for which the drift is negative: we note that, for $\mu\ge 0$, 
\[\varrho^{-\mu}(x,\infty)
= \lim_{\epsilon\to 0} \frac 1 \epsilon \p{\inf\{s>0: W^{-\mu}_s\le 0\}>x\Big|~W^{-\mu}_0=\epsilon}
= \int_x^\infty \frac {e^{-\mu^2 r/2}} {\sqrt{2\pi r^3}}  dr\,.\]
The following proof relies this and the fact that $X^{-\mu}_s\le W^{-\mu}_s$ for all $s\ge 0$ in the natural coupling.

\begin{proof}[Proof of Lemma~\ref{lem:sizes_fragments-k}] Let $B_k$ be the event that $R_{k^3}> 2k^3-1$ or $R_{(k+1)^3}< 2(k+1)^3+1$. Then, by Lemma~\ref{lem:position_Hlambda}, $\pc{B_k}\le \exp(-ck^3)$ for some constant $c>0$. Now, on the one hand, the quantity of local time $LT_k$ corresponding to $\Lambda_k$ is  
\[LT_k:=\int_{k^3}^{(k+1)^3} R_{\lambda-}d\lambda = \int_{k^3}^{(k+1)^3} R_{\lambda}d\lambda \le R_{(k+1)^3}[(k+1)^3-k^3]\,,\]
so that $LT_k\le 7k^5$ for all $k$ large enough on the complement event $B_k^c$.  On the other hand, $\sup\{\lambda-\inf q_\lambda: \lambda \in \Lambda_k\}\le (k+1)^3-R_{k^3}\le -k^3/2$ for large $k$ on $B_k^c$. 

By Proposition~\ref{pro:distribution_R-process}, it follows that, on the event $B_k^c$, for every $x>0$, $\#\{\lambda\in \Lambda_k: m_\lambda>x\}$ is stochastically dominated by a Poisson random variable with parameter $\Delta_k \varrho^{-\mu_k}(x,\infty)$ with $\Delta_k=7k^5$ and $\mu_k=k^3/2$. The properties \emph{i)} to \emph{iv)} in the statement then follow easily. 

% {\red Putting together Proposition~\ref{pro:bound_poisson-measure} and Lemma~\ref{lem:local_time} with $\Delta_k=8k^5$, $\mu_k=k^3/2$, and letting $B_k$ be the bad event that one of the properties in Lemma~\ref{lem:local_time} is violated counting as well the bound on the position of $R_{k^3}, R_{(k+1)^3}$ from Lemma~\ref{lem:position_Hlambda}, then $\pc{B_k}\le \exp(-c k^3)$, and on $B_k^c$ we have a domination of $(m_\alpha: \alpha \in \cA_k)$ by the atoms of a Poisson point measure. }

\noindent\emph{i) Longest excursion.} We have
\begin{align*}
	\bE\Bigg[\sum_{\lambda\in \Lambda_k} \I{m_\lambda > k^{-5}} \I{B_k^c} \Bigg]
	\le 7k^5 \varrho^{-\mu_k}(k^{-5},\infty) 
	= 7k^5 \int_{k^{-5}}^\infty \frac{e^{-k^6 t/8}}{\sqrt{2\pi t^3}} dt
	\le e^{-c k}\,,
\end{align*}
for some constant $c$ and all $k$ large enough. Markov's inequality then implies the claim. 

\noindent \emph{ii) Total length of excursions of level $i$.} We proceed similarly for the upper bound on $M_{k,i}$. We have
\begin{align*}
\bE\Bigg[\sum_{\lambda\in \Lambda_k} m_\lambda \I{m_\lambda \le k^{-5-i}} \I{B_k^c} \Bigg]
\le 7 k^5 \int_0^{k^{-5-i}} \!\!\!\!\!\! x \varrho^{-\mu_k}(dx)  
\le 7 k^5 \int_0^{k^{-5-i}} \frac{xe^{-k^6 x/8}}{\sqrt{2\pi x^3}}  dx
%\le 7/2 k^5 \int_0^{k^{-5-i}} x^{-1/2} dx \\ 
\le 7 k^{5/2 - i/2}\,.
\end{align*}
Moreover, we also always have the bound $\sum_{\lambda\in \Lambda_k} m_\lambda \le R_{(k+1)^3}-R_{k^3}$ which is at most $7k^2$ on $B_k^c$. Markov's inequality then yields
\[\pc{M_{k,i} \ge k^{7/2-i/4}, B_k^c} \le \min\{7 k^{-1-i/4}, 7k^{-3/2+i/4}\}\,.\]

\noindent \emph{iii) Cardinality of $\Lambda_{k,i}$.} This is a deterministic bound on $B_k^c$: we have $\#\Lambda_{k,i}\le 7 k^2 \cdot k^{6+i}=7 k^{8+i}$.

\noindent \emph{iv) Maximum diameter in level $i$.} By Proposition~\ref{pro:diameter_small}, when the event of \emph{i)} occurs, $D_{k,i}$ is stochastically dominated by 
\[k^{-5/2-i/2} \cdot \max\{D^\star_j: 1\le j\le \# \Lambda_{k,i}\}\,,\]
where $D^\star_j$ are iid copies of the random variable $D^\star$ that is sub-Gaussian (choose $\epsilon=1$). Using the bound for $\#\Lambda_{k,i}$ in \emph{iii)} above, it follows easily that  
\begin{align*}
	\pc{D_{k,i} \ge k^{-2-i/4}, B_k^c} 
	& \le 7 k^{8+i} \cdot \pc{D^\star \ge k^{1/2+i/4}} \\
	& \le 7 k^{8+i} \cdot \exp(-\tfrac{k^{1+i/2}}{2v})\,,
\end{align*}
for some constant $v>0$. 

Finally, writing $B_{k,i}$ for the event that either $B_k$ or any of the bad events in \emph{i)}--\emph{iv)} occur, we have $\pc{B_{k,i}}\le 8\min \{7k^{-1-i/4},k^{-5/4}\}$ for large $k$ (this is essentially limited by the event in \emph{ii)}). The union bound yields the last claim. 
\end{proof}

% subsubsection the_distribution_of_sizes_of_swallowed_components (end)

\subsection{The accumulation of length in an annulus: Proof of Proposition~\ref{pro:bound_annulus}} % (fold)
\label{sub:contribution_of_intervals}

Each component $\sF_\lambda$ gets connected to some point to its left (in $\R_+$), which falls within some $\sF_{\lambda'}$, for some $\lambda'<\lambda$, and so on. The proof of Proposition~\ref{pro:bound_annulus} consists in bounding the accumulation of these lengths before a connection to $\sM_{k^3}$ eventually occurs. We are only interested here in the points of $(1,\infty)$.

Recalling the notation $\ju(\cdot)$ from Remark~\ref{rem:points_absolute}, by construction, for each $\lambda\in \Lambda^\ssp$, the point $l_\lambda=\inf q_\lambda$ is identified with $\ju(l_\lambda)$ in the metric space $(\sM,d)$. Furthermore, $\ju(l_\lambda)$ is uniform in $H_{\lambda-}$. For any $x\in q_\lambda$, with $\lambda \in \Lambda^\tsp$, the segment between $x$ and $1$ must contain the points $\ju(l_\lambda)$ that we may see as a projection of $q_\lambda$ on $H_{\lambda-}$. With this in mind, we let $p(x)=\ju(l_\lambda)$ if $x\in q_\lambda$, $\lambda\in \Lambda^\ssp$. Then, for any point $x>1$, we consider the sequence of successive projections defined by $p^0(x)=x$, and provided that $p^n(x)>1$, $p^{n+1}(x)=p(p^n(x))$, until we eventually find a point in $[0,1]$. 

Fix now some natural number $k\ge 1$ and let $A_k:=(R_{k^3}, R_{(k+1)^3}]$ denote the set of points of the annulus of fragments $\sF_\lambda$, for $\lambda\in \Lambda_k$. For each $\lambda\in \Lambda_k$, and $n\ge 0$, let $\lambda_n$ be such that $p^{n}(l_\lambda)\in q_{\lambda_n}$. Recall that $D_\lambda=\sup\{d(x,y): x,y\in q_\lambda\}$. Clearly, the distance from any point $x\in A_k$ to $\sM_{k^3}$ is at most 
\[\sup_{\lambda \in \Lambda_k} \sup_{x\in q_\lambda} d(x,\sM_{k^3}) \le \sup_{\lambda \in \Lambda_k} \sum_{1\le n\le N_\lambda} D_{\lambda_n} \,,\]
where $N_\lambda=\inf\{n\ge 1: p^n(l_\lambda)<R_{k^3}\}$. However, since $\Lambda_k$ is infinite, we shall refine the analysis and rely on the decomposition into different levels introduced in the previous section. 

Recall that we say that $q_\lambda$ is an interval of level $i$, and write $\lambda\in \Lambda_{k,i}$ if $m_\lambda\in [k^{-6-i},k^{-5-i})$; let $A_{k,i}=\bigcup_{\lambda\in \Lambda_{k,i}} q_\lambda$ be the subset of $A_k$ consisting of the intervals of level $i$. For $i\ge 0$, and $\lambda \in \Lambda_{k,i}$, let $N_i(\lambda):=\inf\{n\ge 1: p^{n}(l_\lambda) \not\in A_{k,i} \}$ be the number of hops until hitting an interval of level lower than $i$, or exiting $A_k$ altogether from the left. We then have
\begin{align}\label{eq:bound_accumulation}
\sup_{\lambda \in \Lambda_k} \sum_{1\le n\le N_\lambda} D_{\lambda_n} 
& \le \sum_{i\ge 0} \sup_{\lambda\in \Lambda_{k,i}} \sum_{1\le n \le N_i(\lambda)} D_{\lambda_n} \notag \\ 
& \le \sum_{i\ge 0} \sup_{\lambda\in \Lambda_{k,i}} N_i(\lambda) \cdot \sup_{j\ge i}\sup_{\lambda\in \Lambda_{k,j}} D_\lambda \notag \\ 
& \le \sum_{i\ge 0} \sup_{\lambda\in \Lambda_{k,i}} N_i(\lambda) \cdot k^{-2-i/4}\,,
\end{align}
provided that the event $\bigcap_{i\ge 0} B^c_{k,i}$ from Lemma~\ref{sub:statistics_small_components} occurs. So it remains only to upper bound $N_{k,i}=\sup\{N_i(\lambda):\lambda \in \Lambda_{k,i}\}$. We do this using the properties of the sequence of the projections. 

\begin{lem}\label{lem:Ni}
For any $k\ge 1$ large enough,  with probability at least $1-11 k^{-5/4}$, we have for every $i\ge 0$, 
\[N_{k,i} := \sup_{\lambda \in \Lambda_{k,i}} N_\lambda < 20\,.\]
\end{lem}
\begin{proof}Let $\cG$ be the sigma-algebra generated by $(X_s)_{s\ge 0}$. Then $(q_\lambda)$, $\lambda \in \Lambda^\ssp$ is $\cG$-measurable while, conditionally on $\cG$, the random variables $\xi(l_\lambda)$ are independent and uniform in $H_{\lambda-}$. Let $B_k$ be the event that $R_{k^3}<2k^3-1$ or $R_{(k+1)^3}>2(k+1)^3+1$. For any $\lambda\in \Lambda_k$, $(1,R_{k^3}] \subseteq H_{\lambda-}$; therefore, on the event $B_k^c$ for any Borel set $A$, $\pc{\xi_\lambda \in A~|~B_k^c,\cG}\le \Leb(A)/k^3$. Furthermore, by Lemma~\ref{sub:statistics_small_components}, on $B^c_{k,i}$ we have $M_{k,i}\le \max\{k^{7/2-i/4}, 7k^2\}$ and $\#\Lambda_{k,i}\le 7k^{8+i}$. It follows by the union bound, that for any natural number $m\ge 1$, we have
\begin{align*}
	\pc{N_{k,i} \ge m~|~B^c_{k,i}, B^c_k, \cG} 
	%& \le \pc{N_{k,i} \ge m, M_{k,i} \le k^{7/2-i/4}}  \\
	& \le \#\Lambda_{k,i} \cdot \sup_{\lambda \in \Lambda_{k,i}}\pc{N_\lambda \ge m ~|~M_{k,i} \le k^{7/2-i/4}} \\
	& \le 7 k^{8+i} \cdot (k^{7/2-i/4}/k^3)^m \\
	& \le 7 k^{8+i} \cdot k^{-m(i-2)/4}\,.
\end{align*}
As a consequence, for $m=20$, we obtain
\[\pc{N_{k,i} \ge 20 ~|~B^c_{k,i}, B^c_k, \cG} \le 7 k^{18-4i}\,,\]
which will be good enough for $i\ge 5$. On the other hand, for $0\le i\le 4$, the alternative bound $M_{k,i}\le 7k^2$ yields a bound of $7k^{8+i-m}\le 7k^{-8}$ for $m=20$. Putting everything together, we have $N_{k,i}\ge 20$ for some $i\ge 0$ with probability at most $10 k^{-5/4} + 28 k^{-8} + 8 k^{-2}\le 11 k^{-5/4}$ for all $k$ large enough. 
\end{proof}

Going back to \eqref{eq:bound_accumulation}, Lemma~\ref{lem:Ni} implies that 
\[\sup_{\lambda\in \Lambda_k} \sum_{1\le n \le N_\lambda} D_{\lambda_n} \le 20 \sum_{i\ge 0} k^{-2 -i/4} \le k^{-3/2}\,\]
with probability at least $1-11k^{-5/4}$ which completes the proof of Proposition~\ref{pro:bound_annulus}.

% subsubsection contribution_of_intervals (end)

% subsubsection distances_in_swallowed_components (end)

\subsection{The diameter of the string of beads: Proof of Proposition~\ref{pro:left-end}} % (fold)
\label{sub:string_of_beads}

By construction, for any $\lambda \in \R$, the diameter of $\sP_\lambda$ is no greater than 
\begin{align}\label{eq:sum_diam_around0}
\sum_{\lambda'>\lambda} \diam(\sT^-_{\lambda'}) \I{\lambda'\in \Lambda^\ssm}
&\le \sum_{\lambda'>\lambda} \sqrt{m_{\lambda'}} \cdot D_\lambda^\star \I{\lambda'\in \Lambda^\ssm}\,,
\end{align}
where $D^\star_\lambda$, $\lambda\in \Lambda^\ssm$, are i.i.d.\ copies of the sub-Gaussian random variable whose existence is guaranteed by Proposition~\ref{pro:diameter_small} with $\epsilon=1$. We have already bounded a similar sum in Section~\ref{sub:recursive_convex_minorants_BPT}; in particular, the arguments there show that almost surely
\begin{equation}\label{eq:bound_around_zero}
\sum_{\lambda'\in \Lambda^\ssm} \sqrt{m_{\lambda'}} < \infty\,.
\end{equation}
Finally, consider the process $M_\lambda(s)$ defined for $s\ge 0$ by 
\[M_\lambda(s):=\sum_{\lambda<\lambda'\le \lambda+s} \sqrt{m_{\lambda'}} \cdot (D^\star_{\lambda'}-\Ec{D^\star_{\lambda'}})\,.\]
Conditionally on the $m_\lambda$, $\lambda \in \Lambda^-$, $(M_\lambda(s))_{s\ge 0}$ is a martingale. Since $D^\star$ is sub-Gaussian, $M_\lambda(s)$ is bounded in $L^2$ and thus converges almost surely to a finite limit as $s\to\infty$. Putting this together with \eqref{eq:bound_around_zero} shows that the right-hand side of \eqref{eq:sum_diam_around0} and hence $\diam(\sP_\lambda)$ tends to zero as $\lambda \to \infty$, which completes the proof of Proposition~\ref{pro:left-end}. 

% subsubsection subsubsection_name (end)

\section{The mass measure and Hausdorff dimension of $\CMT(X,\bU)$}
\label{sec:mass_hausdorff}
%!TEX root = MST_brownian.tex

In this section we prove the lower bound on the Hausdorff dimension of $\sM$. We use the mass distribution principle using the mass measure $\mu$ that is defined in Section~\ref{sub:mass_measure}. The asymptotics for the $\mu$-mass of small balls are provided in Section~\ref{sub:evolution_mass_ball} and relies heavily on the growth process defined in Section~\ref{sec:compactness}.

\subsection{The mass measure} % (fold)
\label{sub:mass_measure}

We start with the construction of the mass measure $\mu$ on $\sM$. The measures in this section will always be seen as Borel measures on $\sM$, the completion of $\R_+$ with respect to $d$. For $t\in \R_+$ we let $\mu_t$ be the rescaled Lebesgue measure on $[0,t]$: $\mu_t(A)=t^{-1} \Leb(A\cap [0,t])$. For each $t$, $\mu_t$ is a probability measure on $\sM$ which charges only a subtree containing the root $0$ (it is easy to see that $[0,t]$ is connected). Let $\cL$ be the set of leaves of $\sM$, that is the set of points $x$ such that $\sM\setminus \{x\}$ is connected. Our aim in this section is the following
\begin{prop}\label{prop:mass_measure}
With probability one, as $t\to\infty$, $\mu_t$ converges weakly to a limit probability measure that we denote by $\mu$ and call the mass measure on $\sM$. Furthermore $\mu(\R_+)=0$ so $\mu(\cL)=1$.
\end{prop} 

Recall the notation in Section~\ref{sub:the_growth_process}. For a subset $S$ of the tree $\sM$ and $x\in \R_+$, we define the projection of $x$ onto $S$ as the point of $S$ that is closest to $x$. Fix $\lambda_0\in \R$. We are interested in the projection onto the subset of $\sM$ consisting of the points $[0,R_{\lambda_0}]$. For $x\in \sM$, we let $[x]_{\lambda_0}$ denote the corresponding point. Observe that, with the notation of the previous section, a.s.\ $[x]_{\lambda_0}=\sup\{p^n(x) \cap [0,R_{\lambda_0}], n\ge 0\}$.

Even though $[0,R_{\lambda_0}]$ is not closed in $\sM$, we will always have $[x]_{\lambda_0}\in [0,R_{\lambda_0})$ for the points $x$ we consider. Rather than working with the measures $\mu_t$, $t\in \R_+$, it will be more convenient to work with $\bar \mu_\lambda := \mu_{R_\lambda}$ for $\lambda \in \R$; Lemma~\ref{lem:position_Hlambda} which says that $R_\lambda\to\infty$ guarantees that taking the limits as $t\to\infty$ or $\lambda\to\infty$ is equivalent.

We define the following process: for a Borel set $S\subseteq [0,R_{\lambda_0}]$ and $\lambda\ge \lambda_0$, 
\begin{equation}\label{eq:def_martingale}
M_\lambda=M_\lambda(S):=\bar \mu_\lambda(\{x\le R_\lambda: [x]_{\lambda_0} \in S\})\,.
\end{equation}
We will consider only the randomness coming from $\bU$ and study $M_\lambda$ conditionally on $\sigma(R_\lambda:\lambda\ge \lambda_0)$. We let $\bP^{\shortdownarrow}$ and $\bE^\shortdownarrow$ be the corresponding probability and expectation. 

\begin{lem}\label{lem:mass_martingale}
The process $(M_\lambda, \lambda \ge \lambda_0)$ is almost surely a supermartingale under $\bP^\shortdownarrow$.
\end{lem}
\begin{proof}For each $\lambda\in \Lambda^\ssp$, $\lambda>\lambda_0$, all the points $x\in q_\lambda$ have the same projection on $[0,R_{\lambda_0}]$ since $p(x)=\xi(l_\lambda)$. Furthermore, in order to determine where an interval $q_\lambda$ projects onto $H_{\lambda_0}$ it suffices to follow the sequence of random projections/jumps $p^n(x)$, $n\ge 1$. Almost surely, $\inf\{n\ge 0: p^n(l_\lambda)\in H_{\lambda_0}\}<\infty$ and the point $[l_\lambda]_{\lambda_0}$ is uniformly random in $H_{\lambda_0}$. In the following, $\lambda_0$ being fixed, we use $\phi(\lambda)$ as a short-hand for $[l_\lambda]_{\lambda_0}$. The points $\phi(\lambda)$, $\lambda>\lambda_0$, are of course not independent because of the coalescence of the trajectories. Then,
\begin{equation}\label{eq:mass_martingale}
M_\lambda = \frac 1 {R_\lambda} \left[{M_{\lambda_0}\cdot R_{\lambda_0}} + \sum_{\lambda_0<\lambda'\le \lambda } m_{\lambda'} \I{\phi(\lambda') \in S}\right].
\end{equation}
It follows that, writing $\cF_\lambda$ for the sigma-algebra generated by $(\ju(l_{\lambda'}): \lambda
'\le \lambda)$, the random variable $M_\lambda$ is bounded and $\cF_\lambda$-measurable. The expression in \eqref{eq:mass_martingale} is amenable to a simple evaluation of the conditional expectations: for $h\ge 0$,
\begin{align*}
\bE^\shortdownarrow[M_{\lambda+h}~|~\cF_\lambda]
  & = \frac 1 {R_{\lambda+h}} \cdot \bE^\shortdownarrow \Bigg[M_{\lambda_0}\cdot R_{\lambda_0} + \sum_{\lambda_0< \lambda'\le \lambda + h} m_{\lambda'}\I{\phi(\lambda') \in S}~\Bigg|~\cF_\lambda \Bigg]\\
  & = \frac 1 {R_{\lambda+h}} \Bigg(R_\lambda \cdot M_\lambda + \sum_{\lambda<\lambda'\le \lambda+h} m_{\lambda'} \bP^\shortdownarrow (\phi(\lambda')\in S~|~\cF_\lambda)\Bigg)\,.
\end{align*}

However, almost surely conditionaly on $\cF_\lambda$, $\phi(\lambda')\in S$ if and only if the first point of the sequence $(p^i(l_{\lambda'}))_{i\ge 1}$ that falls in $H_\lambda$ lies in some interval $q_{\lambda''}$ itself such that $\phi(\lambda'')\in S$. By definition, such a point is the projection of $l_{\lambda'}$ on $H_\lambda$, and is uniform in $H_\lambda \subset [0,R_\lambda]$ and therefore $\bP^\shortdownarrow(\phi(\lambda')\in S~|~\cF_\lambda)\le M_\lambda$. It follows that 
\begin{align*}
\bE^\shortdownarrow [M_{\lambda+h}~|~\cF_\lambda]
& \le \frac{M_\lambda} {R_{\lambda+h}} \cdot \Bigg(R_\lambda + \sum_{\lambda<\lambda'\le \lambda+h} m_{\lambda'}\Bigg)  = M_\lambda\,,
\end{align*}
which completes the proof.
\end{proof}

% \begin{alert}I wrote on some notes ``Need to be careful about the continuity of $\cF_\lambda$'', but not clear why anymore. 
% \end{alert}

\begin{proof}[Proof of Proposition~\ref{prop:mass_measure}]
Since $\sM$ is compact by Proposition~\ref{pro:space_compact}, the collection of measures $(\mu_t)_{t> 0}$ is tight. We prove that it is Cauchy for the Prohorov metric using the super-martingales $M_\lambda$ we have just introduced. Recall that, for two Borel measures $\nu$ and $\nu'$ on $(\sM,d)$, the Prohorov distance is given by 
\[\dP(\nu,\nu')=\inf\{\epsilon>0: \nu(A)\le \nu'(A^\epsilon)+\epsilon, \nu'(A)\le \nu(A^\epsilon)+\epsilon \text{ for all Borel sets } A\}\,\]
where $A^\epsilon=\{x: d(x,A)<\epsilon\}$. 

The arguments for compactness in Section~\ref{sec:compactness} show that for any $\epsilon>0$, there exists a $\lambda_0\in \R_+$ such that $\sup_x d(x,[0,R_\lambda])<\epsilon$ for all $\lambda\ge \lambda_0$. With this choice for $\lambda_0$, it follows that, for any $\lambda>\lambda_0$, 
\[\dP(\mu_\lambda,[\mu_\lambda]_{\lambda_0})<\epsilon\,,\] 
where $[\mu_\lambda]_{\lambda_0}$ denotes the image of $\mu_\lambda$ by the projection onto $[0,R_{\lambda_0}]$. Therefore, for any $\lambda,\lambda'\ge \lambda_0$,
\[\dP(\mu_\lambda,\mu_{\lambda'})\le \dP([\mu_\lambda]_{\lambda_0}, [\mu_{\lambda'}]_{\lambda_0}) + 2 \epsilon\,.\]
To complete the proof, cover $[0,R_{\lambda_0}]$ with finitely many balls of diameter $\epsilon$, say $B_1,B_2,\dots, B_k$. Then, by definition of $M_\lambda(S)$, we can construct a coupling $(X,Y)$ with $X\sim [\mu_\lambda]_{\lambda_0}$ and $Y\sim [\mu_{\lambda'}]_{\lambda_0}$ such that $(X,Y)\not \in \cup_{i} B_i \times B_i$ with probability at most $\sum_{i} |M_\lambda(B_i)-M_{\lambda'}(B_i)|$. Since the diameter of $B_i$ is at most $\epsilon$, the cost of the coupling on $B_i\times B_i$ is at most $\epsilon$. It follows that 
\[\dP([\mu_\lambda]_{\lambda_0},[\mu_{\lambda'}]_{\lambda_0}) \le \sum_{i=1}^k |M_{\lambda}(B_i)-M_{\lambda'}(B_i)| + \epsilon \,,\]
which is at most $2\epsilon$ for all $\lambda,\lambda'$ large enough because of the convergence of the mass super-martingales of Lemma~\ref{lem:mass_martingale}. This completes the proof of convergence.

The two additional properties are straightforward from the definition. First, for any interval $[i,i+1)$, $\mu([i,i+1))=\lim_\lambda \mu_\lambda([i,i+1])= 0$ since $R_\lambda\to \infty$, and thus $\mu(\R_+)=0$. The completion of $\R_+$ with respect to $d$ only adds leaves, so $\sM\setminus \R_+\subseteq \cL$, and therefore $\mu(\cL)=1$. 
\end{proof}

% subsection mass_martingales (end)

\subsection{The mass of balls around zero} % (fold)
\label{sub:evolution_mass_ball}

It is proved in \cite{AdBrGoMi2013a} that the Minkowski dimension of $\sM$ is almost surely equal to 3, and we thus only need to find a lower bound. For this, we aim at using the mass distribution principle with the mass measure. In this direction, one needs to upper bound the $\mu$-mass of balls centered at points with distribution $\mu$. In general, this might be delicate since we need to identify the balls around these points, and they are almost surely not in $\R_+$ (Proposition~\ref{sub:mass_measure}). This is why the following result is crucial; the intuition should be intuitively clear from the discrete setting, where the point $1$ can be replaced in Prim's algorithm by a uniformly random point in $[n]$ without altering the distributions. For a point $x\in \sM$ and $r> 0$ we let $B_x(r)$ denote the open ball of radius $r$ centered at $x$ (for the metric $d$).

\begin{lem}\label{lem:uniform_point}Let $\zeta$ be a point of $\sM$ with distribution $\mu$. Then, the processes $(\mu(B_0(r)))_{r\ge 0}$ has the same distribution as $(\mu(B_\zeta(r)))_{r\ge 0}$. 
\end{lem}
\begin{proof}
For $n\ge 1$, recall from Section~\ref{sec:motivation_and_history} that $v_1,v_2,\dots, v_n$ denote the Prim order on $[n]$ on the complete graph with edge weights $(w_e)$, $e\in E^n$; in what follows, we will occasionally write $v(i)$ instead of $v_i$. 
For $n\ge 1$, let $V^n_\lambda$ be the collection of vertices connected to $v(\lfloor n^{2/3}\rfloor)$ in the random graph with edge weights at most $p_n(\lambda)$, and let $H^n_\lambda$ denote the collection of their Prim ranks, and let $L^n_\lambda=\min H^n_\lambda$. Let $\hat \mu^n_\lambda$ denote the uniform probability distribution on $H^n_\lambda$. 
By Lemma~\ref{lem:dist_left-most-node}, conditionally on $V^n_\lambda$, the vertex $v(L^n_\lambda)$ is uniformly random in $V^n_\lambda$, and independent of the random variables $w_{e}$ associated to the edges with end points in $V^n_\lambda$. 
It follows in particular that, for any $r\ge 0$, 
\begin{align}\label{eq:mass_discrete_ball}
  \mu^n_\lambda(\{u \in H^n_\lambda: d^n_\lambda(L^n_\lambda, u)\le r n^{1/3}\}) 
\eqdist
\mu^n_\lambda(\{u \in H^n_\lambda: d^n_\lambda(\zeta^n_\lambda, u)\le r n^{1/3}\}) \,,
\end{align}
where $\zeta^n_\lambda$ denotes an independent point with distribution $\hat \mu^n_\lambda$ (uniform in $H^n_\lambda$).
By Proposition~\ref{pro:convergence_distances}, $(H^n_\lambda, d^n_\lambda, \mu^n_\lambda, L^n_\lambda, \zeta^n_\lambda)$ converges in distribution in the sense of Gromov--Prokhorov to $(H_\lambda, d, \hat \mu_\lambda, L_\lambda, \zeta_\lambda)$, where $\zeta_\lambda$ is an independent point with distribution $\mu_\lambda$. This implies the convergence of the random variables in \eqref{eq:mass_discrete_ball} as $n\to\infty$ towards
\[
\hat\mu_\lambda(\{u\in H_\lambda: d(L_\lambda,u)\le r\})
\eqdist
\hat\mu_\lambda(\{u\in H_\lambda: d(\zeta_\lambda,u)\le r\})\,.
\]
Now, $d(0,L_\lambda)\le \diam(\sP_\lambda)\to 0$ as $\lambda\to \infty$ by Proposition~\ref{pro:left-end}. Note also that a straightforward coupling yields $\dP(\hat \mu_\lambda, \mu_\lambda)\le L_\lambda/R_\lambda \le L_\lambda$ (independently of the metric since we can match the points exactly on a set of probability $1-L_\lambda/R_\lambda$). Taking the limit as $\lambda\to\infty$, Lemma~\ref{lem:position_Hlambda} and Proposition~\ref{prop:mass_measure} yield the claim for every fixed $r\ge 0$. This is easily extended to the joint convergence for finitely many values $r_1<r_2<\dots<r_k$, which completes the proof.
\end{proof}

In order to upper bound $\mu(B_0(r))$ we will proceed in two steps: we will first upper bound $\bar \mu_\lambda(B_0(r))$ showing that it is of the correct order of magnitude, that is roughly $r^{3}$ for some well-chosen $\lambda$ depending on $r$ ($\lambda$ of order $1/r$); we will then rely on the concentration for the mass supermartingales of the previous section, which controls the evolution of the mass as $\lambda$ increases, to show that $\mu(B_0(r))=\lim_\lambda \bar \mu_\lambda(B_0(r))$ remains of order $r^3$. Once we have the relevant upper bound for a fixed $r$, the proof is easily completed using routine arguments (taking a suitable subsequence and the Borel--Cantelli lemma). 

\begin{prop}\label{pro:mass_distribution-lambda}There exists a constant $c>0$ such that, for any $\epsilon\in (0,1)$ and every $r>0$ small enough,
\[\pc{\bar \mu_{r^{\epsilon-1}}(B_0(r)) > r^{3-\epsilon}} \le r^{c\epsilon}\,.\]
\end{prop}
\begin{proof}Observe that if $d(0,H_\lambda)>r$ then no point of $H_\lambda$ lies within $B_0(r)$, so that $\bar\mu_\lambda(B_0(r))$ is at most $L_\lambda/R_\lambda$. Using this with $\lambda=r^{\epsilon-1}$, it follows that
\begin{align}\label{eq:bound_mubar_start}
\pc{\bar \mu_{r^{\epsilon-1}}(B_0(r)) > r^{3-6\epsilon}}
& \le \p{\bar \mu_{r^{\epsilon-1}}(B_0(r)) > \frac{L_{r^{\epsilon-1}}}{R_{r^{\epsilon-1}}}} + \p{\frac{L_{r^{\epsilon-1}}}{R_{r^{\epsilon-1}}} > r^{3-6\epsilon} } \notag \\
& \le \pc{d(0,H_{r^{\epsilon-1}}) \le r } + \pc{L_{r^{\epsilon-1}} > r^{2-4\epsilon}} + \pc{R_{r^{\epsilon-1}} \le r^{-1+2\epsilon}}\notag \\ 
& \le \pc{d(0,H_{r^{\epsilon-1}}) \le r } + \exp(-r^{-\epsilon})\,,
\end{align}
where the last line follows, for all $r>0$ small enough, from the bounds in Lemma~\ref{lem:position_Hlambda}. 

Most of the work now consists in bounding the first term in \eqref{eq:bound_mubar_start} above. Observe that the geodesic from $H_{r^{\epsilon-1}}$ to $0$ must cross every single one of the metric spaces induced by $\sM$ on the intervals $q_\lambda\subseteq [0,1]$ with $\lambda>r^{\epsilon-1}$. Furthermore, with the notation of Section~\ref{sub:distances_in_swallowed_components}, $d(0,H_{r^{\epsilon-1}})$ decomposes as follows: since here the portion of path in $q_\lambda$ is precisely between $\ju(\sup q_\lambda)\in q_\lambda$ a.s. and $\inf q_\lambda$; we have 
\begin{align*}
d(0,H_{r^{\epsilon-1}}) 
%= \sum_{\lambda\in \Lambda^\ssm} d(\ell_\lambda, \zeta_\lambda) 
= \sum_{\lambda\in \Lambda^\ssm} Y_{\lambda} \I{\lambda > r^{\epsilon-1}}
\ge_{st} \sum_{\lambda\in \Lambda^\ssm} Y^\star_{\lambda} \cdot m_{\lambda}^{1/2} \cdot \I{\lambda> r^{\epsilon-1}}\,,
\end{align*}
where the last inequality is a stochastic minoration that relies on Proposition~\ref{pro:diameter_small}: the $m_\lambda$, $\lambda \in \Lambda^\ssm$, are the sizes of the jumps of $L_\lambda$, and the $Y^\star_\lambda$ are conditionally independent from the entire collection $(m_\lambda: \lambda\in \Lambda^\ssm)$. In order to lower bound $d(0,H_{r^{\epsilon-1}})$ it suffices to focus on a single term of the sum in the right-hand side: if any of those terms is greater than $r$, then $d(0,H_{r^{\epsilon-1}})>r$ as well, thus
\begin{align}\label{eq:bound_mubar}
\pc{d(0,H_{r^{\epsilon-1}}) \le r} 
%& \le \p{\sum Y_\lambda^\star m_\lambda^{1/2} \I{\lambda r^{\epsilon-1}>1} \le r} \notag\\ 
& \le \p{\#\{\lambda>r^{1-\epsilon}: \lambda\in \Lambda^\ssm, Y^\star_\lambda m_\lambda^{1/2} > r\}=0} \notag\\ 
& \le \p{\sup\{m_\lambda: \lambda\in \Lambda^\ssm, \lambda>r^{1-\epsilon}\}\le r^{2-\epsilon}} + \pc{Y^\star \le r^{\epsilon/2}}\,.
\end{align}
The second term is at most $r^{c\epsilon}$ by Proposition~\ref{pro:diameter_small} \emph{ii)} and Markov's inequality. The $(m_\lambda)$ are also the lengths of the faces of the convex minorant of $X$ on the interval $[0,1]$ by Lemma~\ref{lem:fragment_of_x}, and to deal with the first term, we relate it to the convex minorant of a standard Brownian motion $W$. 

Let $(Q_t)_{t\ge 0}$ be defined by 
\begin{align*}
Q_t
&=\exp\left(\int_0^t s dW_s - \frac 1 2 \int_0^t s^2 ds\right)\,,
\end{align*}
Then by the Cameron--Martin--Girsanov formula (Theorem 38.5 of \cite{RoWi2000}), the laws of $X$ and $W$ are related by a change of measure whose density is given by the martingale $Q_t$. For a function $\omega \in \cC(\R_+, \R)$, we consider the convex minorant of $\omega$ on $[0,1]$ and we let $\chi_r(\omega)$ denote the indicator that the longest face with slope (strictly) smaller than $-r^{\epsilon-1}$ has length at most $r^{2-\epsilon/2}$. Then, by Lemma~\ref{lem:fragment_of_x} we have
\begin{align}\label{eq:convex_minorant_CS}
\p{\sup\{m_\lambda: \lambda\in \Lambda^\ssm, \lambda>r^{1-\epsilon}\}\le r^{2-\epsilon}}
& = \Ec{\chi_r(X^0)} \notag \\ 
& = \Ec{\chi_r(W) \cdot Q_1} \notag \\ 
& \le \Ec{\chi_r(W)^2}^{1/2} \cdot \Ec{Q_1^2}^{1/2}\,,
\end{align}
by the Cauchy--Schwarz inequality. Observe that in the right-hand side above, $\Ec{Q_1^2}\le \Ec{\exp(2 \overline{W}_1)}$ is finite and independent of $r$ thanks to the Gaussian tails of $\overline W_1$.

The first factor in \eqref{eq:convex_minorant_CS} can be estimated using the results of Pitman and Ross \cite[Theorem~1]{PiRo2011a} and Brownian scaling. Let $(x_i,s_i)_{i}$ be the points of a Poisson point process with intensity $(2\pi x)^{-1/2} \cdot \exp(-(2+s^2)x/2)dxds$ on $\R_+\times \R$. Then $(x_i,s_i)$ are the lengths and slopes of the faces of the convex minorant of $W$ on the interval $[0,E]$ where $E$ is an independent exponential random variable with mean one, and here $E=\sum x_j$. Therefore, by Brownian scaling,
\begin{align*}
\Ec{\chi_r(W)}
&= \p{\sup\{ x_i : s_i \sqrt E <-r^{\epsilon-1}\} \le E \cdot r^{2-\epsilon}}\\ 
& \le \p{A^c} + \p{\sup\{x_i: s_i r^{1-5\epsilon/6}<-1\} \le r^{2-4\epsilon/3}}\,,
\end{align*}
where $A$ denotes the event that $E=\sum x_j \in [r^{\epsilon/3}, r^{-\epsilon/3}]$.
Since $E$ is exponential with mean one, we have $\pc{A^c}\le 2r^{\epsilon/3}$ for all $r>0$ small enough. We claim that there exists a constant $c>0$ such that the second term above is no larger than $r^{c\epsilon}$, and in order the complete the proof, it suffices to justify that claim. We slightly change the scaling and write $\delta=r^{1-5\epsilon/6}$ and $\gamma=\epsilon/3$ to lighten the notation. Then $r^{2-4\epsilon/3}\le \delta^{2+\gamma}$, and we focus on 
\begin{align}\label{eq:around_zero_main}
\p{\sup\{x_i: s_i \delta <-1\} \le \delta^{2+\gamma}}
% & \le \p{\#\{i: x_i> r^{2+2\delta}, s_i r< - 1\} = 0 } \\
& = \exp\left(-\int\!\!\!\!\int \frac{e^{-(2+s^2)x/2}}{\sqrt{2\pi x}} \I{x\ge \delta^{2+\gamma}, s\delta<-1} ds dx\right)\,.
\end{align}
We just need to lower bound the integral in the right-hand side: consider the subregion $\Sigma$ of $[\delta^{2+\gamma}, \infty)\times (-\infty, -1/\delta]$ where $u=xs^2\le 1$, that is $\Sigma:=\{(x,s): \delta^{2+\gamma}\le x\le \delta^2, -x^{-1/2}\le s\le -1/\delta\}$:
\begin{align*}
\int\!\!\!\!\int \frac{e^{-(2+s^2)x/2}}{\sqrt{2\pi x}} \I{x\ge \delta^{2+\gamma}, s \delta<-1} ds dx 
& \ge \int\!\!\!\!\int_\Sigma \frac{e^{-(s^2+2)x/2}}{\sqrt{2\pi x}}dsdx\\ 
&\ge e^{-\delta^2} \int\!\!\!\!\int_\Sigma \frac{e^{-s^2x/2}}{\sqrt{2\pi x}}dsdx \\ 
% & = \frac 1 {2 \sqrt{2\pi}} \int_{r^{2+2\epsilon}}^{r^2} \frac{\Delta s(x)}{\sqrt x} dx \\ 
& \ge \frac {e^{-\delta^2-1/2}} {\sqrt{2\pi}} \int_{\delta^{2+\gamma}}^{\delta^2}\left[\frac 1 x - \frac 1{\delta\sqrt x} \right] dx \\ 
& \ge \frac {-\gamma \log \delta e^{-\delta^2-1/2}} { \sqrt{2\pi}} \,,
\end{align*}
for all $\delta>0$ small enough.
It follows easily that, there exists a constant $c>0$ such that for all $\delta>0$ small enough the right-hand side of \eqref{eq:around_zero_main} is at most $\delta^{c\gamma}$, which translated into the original parameters yields a bound of $r^{c'\epsilon}$ for the right-hand side of \eqref{eq:convex_minorant_CS}, and in turn for \eqref{eq:bound_mubar} and \eqref{eq:bound_mubar_start}. This completes the proof.
\end{proof}

Finally, using Lemma~\ref{lem:uniform_point}, the following proposition completes the proof of Theorem~\ref{thm:compact_dimH}

\begin{prop}\label{pro:mass_distribution}Let $\zeta$ be a point of $\sM$ with distribution $\mu$. Then for every $\epsilon\in (0,1)$, almost surely, for all $r>0$ small enough we have
\[\mu(B_\zeta(r)) \le r^{3-\epsilon}\,.\]
As a consequence $\DimH(\sM)\ge 3$.
\end{prop}
\begin{proof}By Lemma~\ref{lem:uniform_point}, it suffices to prove the bound for $\mu(B_0(r))$. Fix $\lambda_0=1/r$, set $S=B_0(r)$ and recall the process $M_\lambda=\bar\mu_\lambda(\{x\le R_\lambda: d(0,[x]_{\lambda_0}) \le r \})$ of \eqref{eq:def_martingale}. Then, for any $\lambda\ge \lambda_0$, $\bar\mu_\lambda(B_0(r))$ is no larger than the $\bar\mu_\lambda$-mass of the excursions which are grafted within distance $r$ or the origin: 
we have 
\[\bar\mu_\lambda(B_0(r)) \le M_\lambda\,.\]
Since $M_\lambda$ is bounded, Lemma~\ref{lem:mass_martingale} implies that $M_\lambda$ converges almost surely as $\lambda\to\infty$, but it also implies some concentration results since the increments of $M_\lambda$ are bounded by the $m_\lambda$, $\lambda \in \Lambda^\ssp$. By the Azuma--Hoeffding inequality \cite{Azuma1967,Hoeffding1963,BoLuMa2012a}, for any $x>0$, we have, conditionally on $(m_\lambda, \lambda\in \Lambda)$,
\begin{align}\label{eq:azuma_mass-martingale}
\p{M_\lambda - M_{\lambda_0}>x~|~m_\lambda, \lambda>\lambda_0} 
%&\le \exp\bigg(-\frac{x^2}{2 \sum_{\lambda_0<\lambda'\le \lambda} m_{\lambda'}^2}\bigg)\notag \\ 
&\le \exp\bigg(-\frac{x^2}{2 \sum_{\lambda_0<\lambda'} m_{\lambda'}^2}\bigg)\,.
\end{align}

Bounding the $\ell^2$-norm of $(m_\lambda)_{\lambda>\lambda_0}$ is routine using Lemma~\ref{lem:sizes_fragments-k}. Indeed, for any $k\ge 1$, 
\begin{align*}
\E{\sum_\lambda m_\lambda^2 \cdot \I{\lambda \in \Lambda_k}} 
&\le 8k^5\int_0^\infty \frac{t^2}{\sqrt{2\pi t^3}} e^{-k^6 t/8} dt
%&= \frac{8 k^5 }{\sqrt{2\pi}} \int_0^\infty t^{1/2} e^{-k^6t/8} dt\\
%&= \frac{8^{5/2} k^{-4} }{\sqrt{2\pi}} \int_0^\infty u^{1/2} e^{-u} du 
= 64 \cdot k^{-4}\,,
\end{align*}
and Markov's inequality then yields, for some constant $K$,  
\begin{equation}\label{eq:tail_masses}
  \p{\sum_{\lambda: \lambda r>1} m_\lambda^2 \ge r^{3-\epsilon}} \le r^{\epsilon-3} \cdot \E{\sum_\lambda m_\lambda^2 \cdot \I{\lambda >1/r}} \le K r^\epsilon\,.
\end{equation}
Since $M_{\lambda_0}=\bar\mu_{\lambda_0}(B_0(r))$, it follows from Proposition~\ref{pro:mass_distribution-lambda} and \eqref{eq:azuma_mass-martingale}--\eqref{eq:tail_masses} that 
\begin{align*}
\pc{\mu(B_0(r) \ge 2 r^{3-\epsilon})} 
& \le \pc{\mu_{\lambda_0}(B_0(r)) \ge r^{3-\epsilon}} + \pc{M_\infty - M_{\lambda_0} \ge r^{3-\epsilon}} \\
& \le r^{c\epsilon} + \exp(-c r^{-\epsilon}) + K r^\epsilon\,.
\end{align*}

From there, completing the proof is standard: take a subsequence $r_i=2^{-i}$, $i\ge 1$; the Borel--Cantelli implies that for all but finitely many values of $i\ge 1$, we have $\mu(B_0(r_i))\le 2 r_i^{3-\epsilon}$, and thus $\mu(B_0(r))\le 16 r^{3-\epsilon}$ for all $r>0$ small enough. As a consequence, the mass distribution principle (see, e.g., Proposition~4.9 of \cite{Falconer1990a}) implies that $\DimH(\sM)\ge 3-\epsilon$, which completes the proof since $\epsilon>0$ was arbitrary.
\end{proof}

% subsection subsection_name (end)

% subsection the_mass_measure (end)

\section{Distances in the Brownian parabolic tree}
\label{sec:coupling}
%!TEX root = MST_brownian.tex

% \begin{alert}
% Here we need to prove:
% \begin{compactitem}
% 	\item Theorem~\ref{thm:limit_mst_Kn} about the entire MST
% 	\item Theorem~\ref{thm:limit_mst_surplus} about the MST of a connected graph with given surplus
% 	\item Theorem~\ref{thm:dynamics_X} about the dynamics for the limit random graph and Kruskal processes
% \end{compactitem}
% \end{alert}

% \begin{alert}\textbf{NOTATION}
% 	\begin{compactitem}
% 		\item $\bs=(s_1,s_2,\dots, s_\kappa)$ the points in $(0,\infty)$; will use $v_{\lfloor n^{2/3}s_i \rfloor}$ for the discrete analog;
% 		\item $\underline{\lambda}$ for the time at which replace the fragments by CRT ? So far this is $\tau$
% 		\item $D(\bs)=(d(s_i,s_j))_{1\le i,j\le \kappa}$ for the matrix of distances ? $D^n(\bs)$ for the discrete analog;
% 	\end{compactitem}
% \end{alert}

% \begin{alert}Remarks:
% 	\begin{compactitem}
% 		\item I don't think we need the distances in a forest, just in a tree containing all the points ?
% 	\end{compactitem}
% \end{alert}

All the proofs of convergence will be based on couplings with discrete objects. It would be possible to identify the distribution of $\CMT(X, \bU)$ as that of the scaling limit of the minimum spanning tree constructed in \cite{AdBrGoMi2013a} directly in the continuum using the dynamics as $\lambda$ evolves and the tools developed in \cite{AdBrGoMi2019a}. However, since we need comparisons with discrete objects anyway for Theorems~\ref{thm:limit_mst_surplus} and~\ref{thm:dynamics_X}, we do not pursue this here. All the limit theorems essentially boil down to proving that, in a suitable coupling, and for every $\lambda\in \R$, the restriction of the metric space $\CMT(X,\bU)$ to any interval $(a,b)$ of $\R_+\setminus Z^\lambda$ is the limit (in probability) of the minimum spanning tree of a connected component induced by a vertex set whose node have Prim ranks in an interval $\{a_n, a_n+1,\dots, b_n-1\}$ where $a_n \sim a n^{2/3}$ and $b_n \sim b n^{2/3}$. Our coupling will be ``global'' in the sense that it allows a transparent application to any collection of times $\lambda_1<\lambda_2<\dots < \lambda_k$ and any finite collection of intervals at these times. 

\subsection{Discrete preliminaries} % (fold)
\label{sub:discrete_preliminaries}

In this section, we provide the discrete representation that we will use to prove our limit theorems. They all heavily rely on the Prim order introduced in \cite{BrMa2015a} and its properties. We will in particular give a representation of the minimum spanning tree $M_n$, and of the random graph $G(n,p)$ that we will see as the union of a portion of the minimum spanning tree, the Kruskal forest denoted by $K(n,p)$, together with additional cyclic edges. 

Recall the Prim algorithm and the Prim order $v_1,v_2,\dots, v_n$ discussed in Section~\ref{sub:intuition_and_techniques}. Recall also that $V_k=\{v_1,\dots, v_k\}$. For $k\in [n]$, let $N_k^{n, p}$ be the number of nodes in $[n]\setminus V_k$ which have a neighbour in $V_k$ in the graph $G(n,p)$ whose edge set is $\{e: w_e\le p\}$. For $\lambda\in\R$, set $p_n(\lambda)=1/n+\lambda n^{-4/3}$. Then define, for $t\ge 0$, 
\begin{equation}\label{eq:def_Xnlambda}
X^{n,\lambda}_t
:=n^{-1/3} \left( 
N_{\lfloor t n^{2/3}\rfloor }^{n,p_n(\lambda)}
-\#\{i\le t n^{2/3} : N_i^{n,p_n(\lambda)} = 0\}
\right)\,.
\end{equation}
Let $Z^{n,\lambda}$ be $n^{2/3}$ times the collection of instants when $X^{n,\lambda}$ reaches a new minimum. Then, the points of $Z^{n,\lambda}$ are the Prim ranks of the first vertices of the connected components of $G(n,p_n(\lambda))$. Furthermore, recalling the definitions in Section~\ref{sub:intuition_and_techniques}, the collection of the edges of the minimum spanning tree $M_n$ are precisely $\{e_i=(u_i,v_i): 2\le i\le n\}$. The identities of the nodes can of course not be recovered from $(X^{n,\lambda})_{\lambda\in \R}$, but one may use the Prim ranks to construct a graph on $[n]$ that is isomorphic to $M_n$ using $(X^{n,\lambda})_{\lambda\in \R}$ only. However, the information about the location of the $u_i$ vanishes in the limit, and we shall construct a graph that has the correct distribution of the left-end points $u_i$, conditionally on $(Z^{n,\lambda})_{\lambda \in \R}$. 

We start with an encoding of the merges. Note that there are precisely $n-1$ jumps to the process $(Z^{n,\lambda})_{\lambda \in \R}$, each one corresponding to the appearance of one of the edges $e_i=(u_i,v_i)$ for some $2\le i\le n$. Let $\li_n(i)$ and $\ri_n(i)$ be respectively the Prim ranks of the left-most and right-most vertices of the connected component of $v_{i}$ at time $w_{e_i}$; let also $\slo_n(i)=(1-n w_{e_i})n^{1/3}$ be the discrete analog of the slope of a point $t\in \sL$ in the continuous setting. Then, the set
\begin{equation}\label{eq:def_mergen}
	\Merge_n(X^n):=\{(n^{-2/3}\li_n(i), n^{-2/3}i, n^{-2/3}\ri_n(i), -\slo_n(i)): 2\le i\le n\}
\end{equation}
contains all the information about the merges of connected components. We can rephrase the fact that the extremities $u_i$ of the edges $e_i=(u_i,v_i)$ are uniform in the connected component containing $v_{i-1}$ as follows. 
Let $(U_i)_{i\ge 1}$ be i.i.d.\ uniform on $[0,1]$, also independent of $\Merge(X^n)$. For each $i$, let $\ju_n(i)=\li_n(i)+\lfloor U_i (i-\li_n(i))\rfloor$. Then, $\ju_n(i)$ is uniform in $\{\li_n(i), \li_n(i)+1, \dots, i-1 \}$. The following lemma is a simple reformulation of Lemma~\ref{lem:discrete_merges}.

\begin{prop}[A representation of the minimum spanning forest]\label{prop:random_forest}
Conditionally on $\Merge_n(X^n)$, the collection of Prim ranks of the nodes $(u_i: 2\le i\le n)$ has the same distribution as $(\ju_n(i): 2\le i\le n)$. In particular, up to a relabelling of the nodes of $M_n$ using the Prim ranks: 
\begin{compactenum}[i)]
	\item the graph on $[n]$ with edges $\{\ju_n(i), i\}$, $2\le i\le n$, is distributed like $M_n$;
	\item the graph on $[n]$ with edges $\{\ju_n(i), i\}$, $2\le i\le n$ with $-\slo_n(i) \le \lambda$ is distributed like the Kruskal forest $K(n,p_n(\lambda))$.
 \end{compactenum}
\end{prop}

We now move on the representation of the random graphs. We say that an edge is \emph{cyclic} if it is the maximum weight edge of some cycle. For each $p\in [0,1]$, the graph $G(n,p)$ is formed of the portion of the minimum spanning tree consisting of the edges of weight at most $p$, together with the cyclic edges of weight at most $p$. Observe that while the edges of the minimum spanning tree are all a.s.\ a function of $(X^{n,\lambda})_{\lambda\in \R}$, this is not the case for the cyclic edges (with positive probability some information is lost, even at the discrete level).  Again, rather than collecting the information from the random graph, it is more instructive to construct this information with the correct distribution conditionally on $(Z^{n,\lambda})_{\lambda \in \R}$. 

Let $\{Y_{ij}, 1\le i<j\le n\}$ be i.i.d.\ random variables uniform on $[0,1]$ and independent of everything else (namely $\Merge_n(X^n)$ and $(U_i)_{2\le i\le n}$). For each $1\le i<j\le n$, let $\lambda^n_{ij}=(nY_{ij}-1)n^{1/3}$. We store the information concerning cyclic edges in a point process. Define 
\begin{equation}\label{eq:def_xin}
	\Xi_n :=\Big\{(i n^{-2/3}, jn^{-2/3}, \lambda^n_{ij}): 1\le i<j\le n, Z^{n,\lambda^n_{ij}}\cap \{i+1,i+2,\dots, j\} = \varnothing \Big\}\,,
\end{equation}
so that $\Xi_n$ is the collection of triples $(i/n^{2/3}, j/n^{2/3}, \lambda^n_{ij})$ for which $v_i$ and $v_j$ are in the same connected component of $G(n,p_n(\lambda^n_{ij})-\delta)$ for $\delta>0$ small enough. The total number of cyclic edges, sometimes called the surplus, of a connected component is also a quantity of interest, and can be expressed in terms of $\Xi_n$. Recall that $C^{n,\lambda}_i$, $i\ge 1$, denote the collection of vertex sets of the connected components of the random graph $G(n,p_n(\lambda))$, sorted in decreasing order of their sizes. For a discrete connected component $C^{n,\lambda}_j$, the number of surplus edges in $G(n,p_n(\lambda))$ is given by 
\begin{equation}\label{eq:discrete_surplus}
	\surp^{n,\lambda}_j=\#\big\{(x,y,\lambda') \in \Xi_n: xn^{2/3},yn^{2/3}\in C^{n,\lambda}_j, \lambda'\le\lambda\big\}\,.
\end{equation}

\begin{prop}[A representation of the random graph]\label{prop:random_graph}Up to a relabelling of the nodes with the Prim ranks, the graph on $[n]$ with edge set consisting of the union of 
\begin{compactitem}
	\item the edges $\{\ju_n(i), i\}$, $2\le i\le n$ such that $\slo_n(i)\ge -\lambda$, and 
	\item the edges $\{i,j\}$, $1\le i<j\le n$ such that $\lambda^n_{ij}\le \lambda$ and $Z^{n,\lambda^n_{ij}}\cap \{i+1,i+2,\dots, j\} = \varnothing$
\end{compactitem}
has the same distribution as $G(n,p_n(\lambda))$.
\end{prop}
\begin{proof}Note first that we only care about the distribution of the edges of the second set that are not already in the first one. From Kruskal's algorithm, it is clear that, conditionally on the minimum spanning tree $M_n$, the weights of the edges in the complement are independent. Furthermore, for any fixed pair of nodes $u,v\in [n]$ which are not adjacent in the minimum spanning tree, the weight of the edge between $u$ and $v$ is uniform, conditioned on being larger than the value $p_n(\lambda)$ at which $u$ and $v$ first become part of the same connected component. This is precisely what the second condition says when expressed in terms of the Prim ranks. 
\end{proof}

Recall that $E^n_p=\{e\in E^n: w_e\le p\}$ denotes edge set of the random graph $G(n,p)$. Define similarly $F^n_p\subseteq E^n_p$ the edge set of the minimum spanning forest $K(n,p)$, that is the collection of edges of the minimum spanning tree which have weight at most~$p$. 
\begin{lem}\label{lem:discrete_edge_removal}
Let $p\in [0,1]$. 
\begin{compactenum}[i)]
	\item Conditionally on $E^n_{p}$, $(w_e:e\in E^n_p)$ is a family of i.i.d.\ uniform random variables (r.v.) on $[0,p]$;
	\item Conditionally on $F^n_{p}$, $(w_e:e\in F^n_p)$ is dominated by a family of i.i.d.\ uniform r.v.\ on $[0,p]$.
	% namely: for any $x_e\in [0,1]$, $e\in F^n_p$, we have
	% \[\pc{w_e\le p x_e : e\in F^n_p~|~F^n_p} \ge \prod_{e\in F^n_p} x_e\,.\]
\end{compactenum}
\end{lem}
\begin{proof}\emph{i)} The first assertion is immediate since $\{w_e: e\in E^n_p\}$ is simply a collection of i.i.d.\ uniform r.v.\ on $[0,1]$ conditioned on being at most $p$. \emph{ii)} The second claim is a consequence of Kruskal's algorithm: The set $F^n_p\subseteq E^n_p$ is obtained from $E^n_p$ by iteratively removing the edge with maximum weight that belongs to a cycle, until there are no more cycles. The remaining edge have thus been selected for not being the maximum edge of any cycle; by \emph{i)} the initial weights in $E^n_p$  are i.i.d.\ uniform  random variables on $[0,p]$, and the weights in $F^n_p$ are therefore dominated by a collection of i.i.d.\ uniform r.v.\ on $[0,p]$. 
\end{proof}

\begin{lem}\label{lem:dist_left-most-node}
Fix any $\lambda\in \R$. Conditionally on $a,b\in [n]$, $a<b$, being two successive points of $Z^{n,\lambda}$, 
\begin{compactenum}[i)]
	\item the vertices whose Prim ranks are in $\{a, a+1, \dots, b-1\}$ form a connected component of $G(n,p_n(\lambda))$;
	\item conditionally on $S=\{v_a,\dots, v_{b-1}\}$, the vertex $v_a$ with Prim rank $a$ is uniformly random in $S$, and independent of $G(n,p_n(\lambda))$.
\end{compactenum}
\end{lem}
\begin{proof}The first claim is immediate from the definition of $X^{n,\lambda}$; see Section 4.1 of \cite{BrMa2015a}. The second point is a consequence of the definition of the Prim order. Consider the time in Prim's algorithm when we decide who gets to have rank $a$: conditionally on the event in \emph{i)}, this depends on an edge with weight (strictly) larger than $p_n(\lambda)$; conditionally on having its extremity in the set of vertices with Prim ranks $a,a+a,\dots, b-1$, the end point is uniformly random, and declared to have Prim rank $a$. This completes the proof.
\end{proof}

\subsection{Asymptotic properties of random graphs for $\lambda\to -\infty$}
\label{sec:asymptotic_random_graphs}

{Recall that we identify the nodes with their Prim ranks, so $v_i$ is simply denoted by $i$. 
For points $s_1,\dots, s_k \in (0,\infty)$, let $\Span_n(s_1,\dots, s_k)$ denote the collection of vertices that belong to one of the paths in the minimum spanning tree $M_n$ between some $s_i^n=\lfloor s_i n^{2/3}\rfloor$ and $s_j^n=\lfloor s_j n^{2/3} \rfloor$. For $\lambda\in \R$ and $s_1,s_2,\dots, s_k\in (0,\infty)$, let $J^{n}_{\lambda}(s_1,s_2,\dots, s_k)$ be the collection of indices $j\ge 1$ such that $C^{n,\lambda}_j$ intersects $\Span_n(s_1,\dots, s_k)$. 

The following lemma shows that all the connected components containing part of the path in the minimum spanning tree between a collection of random points have a size of order $n^{2/3}$. 

\begin{prop}\label{pro:local_asympt_small_lambda}
Let $I\subset (0,\infty)$ be any compact interval, and let $s_1,s_2,\dots, s_k\in (0,\infty)$ be i.i.d.\ uniform in $I$. Then, for any $\epsilon>0$ there exists $\lambda\in \R$, and $\delta>0$, such that, with probability at least $1-\epsilon$, all the connected components of $G(n,p_n(\lambda))$ containing nodes of $\Span_n(s_1,\dots, s_k)$ are trees and have size at least $\delta n^{2/3}$.
\end{prop}
\begin{proof}
\emph{i)} We abbreviate $J^n_\lambda(s_1,\dots, s_k)$ as $J^n_\lambda$. For any fixed $\lambda$, the collection of connected components containing any of the $s_i^n$, $1\le i\le k$, have Prim ranks at most $n^{2/3} \sup I  + |C^{n,\lambda}_1|$. With high probability, this is at most $tn^{2/3}$ for some fixed $t$ for all $\lambda\le 0$ (say). However, by the representations in \cite{AdBrGo2012a} or \cite{BrMa2015a}, the number of surplus edges involving pairs of nodes with Prim rank at most $tn^{2/3}$ converges to a Poisson random variable whose parameter the $\int_0^t (X^\lambda_s-\underline X^\lambda_s) ds$, which tends to zero almost surely as $\lambda\to-\infty$. Thus, for any $\epsilon>0$, we can indeed choose $\lambda$ small enough for all the $C^{n,\lambda}_j$, $j\in J^{n}_\lambda$ to be trees with probability at least $1-\epsilon$.

\emph{ii)} We shall prove that the family of random variables $\max\{n^{2/3}/ |C^{n,\lambda}_j|: j\in J^n_\lambda\}$, $n\ge 1$, is tight. The arguments are all routine, and we only provide the main structure of the proof. Fix $\epsilon>0$. First, let $\overline \lambda$ be large enough that $s_1^n,\dots, s_k^n$ are all in the same connected component $H^n$ which also contains the point $\lfloor n^{2/3}\rfloor$ with probability at least $1-\epsilon$ (see for instance, Lemma~\ref{lem:position_Hlambda}). 

By Lemma~\ref{lem:discrete_edge_removal} \emph{ii)}, when decreasing $p$ from $p_n(\overline \lambda)$ to $p_n(\lambda)$, each edge is removed with probability at most $(p_n(\overline \lambda)-p_n(\lambda))/p_n(\overline \lambda)\sim (\overline \lambda-\lambda)n^{-1/3}$ independently of the others. So, the number of edges removed on a prescribed path of length at most $C n^{1/3}$ is dominated by a binomial random variable with parameters $Cn^{1/3}$ and $(\overline \lambda-\lambda)n^{-1/3}$ and is thus tight. Since $n^{-1/3}\diam(H^n)\le n^{-1/3}\diam(M_n)$ which is tight (\cite{AdBrGo2010}), the same holds for the length of the path between any of two of the $\{s_1^n, \dots, s_k^n\}$. This implies the tightness of $(|J^n_\lambda|)_{n\ge 1}$, for any $\lambda\in \R$. This also readily implies that $n^{1/3}$ divided by the smallest distance in the minimum spanning tree between any two removed edges is tight. On the other hand, the minimum distance between any two of the $\{s_1^n,\dots, s_k^n\}$ is itself of order $n^{1/3}$ (this is lower bounded by the distance in the corresponding  graph $G(n,p_n(\overline \lambda))$, and thus follows from the results in \cite{AdBrGo2012a,AdBrGoMi2013a}). It follows that the smallest portion of a path connecting the $\{s_1^n, \dots, s_k^n\}$ in the Kruskal forest $K(n,p_n(\lambda))$ is also of order at least $n^{1/3}$.

Now, by \emph{i)}, let $\lambda$ be small enough that all the involved connected components are trees at time $p_n(\lambda)$ with probability at least $1-\epsilon$. Conditionally on the number of its nodes being $m$, the diameter any such connected component is of order $m^{1/2}$ (\cite{Aldous1991b,Aldous1993a}). Putting this together with the facts that, for this value of $\lambda$, the number of portions of paths is tight $J^n_\lambda$, that each of the portions has length of order $n^{1/3}$, this implies that each of the portions is contained in a connected component whose size is indeed of order $n^{2/3}$ (and no smaller). 
\end{proof}

The following folklore global asymptotic properties for the connected components will be useful. 
\begin{lem}\label{lem:global_asympt_small_lambda}
For any $\epsilon,\delta, \delta'>0$, there exists $\lambda\in \R$ such that, with probability at least $1-\epsilon$,
\begin{compactenum}[i)]
	\item the largest connected component of $G(n,p_n(\lambda))$ contains at most $\delta n^{2/3}$ nodes;
	\item the maximum diameter of a connected component of $G(n,p_n(\lambda))$ is at most $\delta' n^{1/3}$.
\end{compactenum}
\end{lem}
\begin{proof}%It is classical, and we collect the relevant references. 
For any $\alpha\le \delta$, the probability that either \emph{i)} or \emph{ii)} fails is at most
\begin{align*}
	\pc{|C^{n,\lambda}_1| \ge \alpha n^{2/3}}	+ \p{\max_{j\ge 1} \diam(C^{n,\lambda}_j) \ge \delta' n^{1/3}, |C^{n,\lambda}_1|\le \alpha n^{2/3}}\,.
\end{align*}
By Theorem~1.3 of \cite{NaPe2008}, there exists $\alpha>0$ small enough such that the second term is at most $\epsilon/2$. 
Then, choose $\lambda$ small enough that the first term is also at most $\epsilon/2$. The fact that such a $\lambda$ exists follows for instance from the results of \cite{AlLi1998} on the entrance boundary for the standard multiplicative (Theorem~4 there), and the relation between the random graph and the multiplicative coalescent in \cite{Aldous1997} (Proposition~4). 
\end{proof}

\subsection{A global coupling argument} % (fold)
\label{sec:global_coupling_argument}

Before actually proving the convergence of the trees or graphs seen as metric spaces, we verify that the main objects, on which the representations of the previous section rely, do converge. The objective is to eventually construct a rich enough probability space on which enough parameters converge almost surely, in order the make the final proof of convergence of the metric as easy as possible. The starting point is the process $(X^{n,\lambda})_{\lambda\in \R}$ introduced in \eqref{eq:def_Xnlambda}. By Theorem 7 of \cite{BrMa2015a}, we have 
\begin{equation}\label{eq:Convergence_Xnlambda}
(X^{n,\lambda})_{\lambda\in \R} \xrightarrow[n\to\infty]{} (X^\lambda)_{\lambda\in R}\,,
\end{equation}
in distribution in $\mathbb D(\R, \C([0,\infty), \R))$. The first essential ingredient consists in proving that this implies that the macroscopic merges restricted to any compact region of time and space also converge. Define 
\[\Merge(X)=\{(\li(t),t,\ri(t),-\slo(t)): t\in \sL(X)\}\,.\]
We say that $\Merge_n((X^{n,\lambda})_{\lambda\in \R)}\to \Merge(X)$ if for any compact intervals $I\subset (0,\infty)$ and $\Lambda\subset \R$, and any threshold $\epsilon>0$, the subset of $\Merge_n((X^{n,\lambda})_{\lambda\in \R}))$ consisting of points $(l,t,r,-s)$ such that $r-t,t-l>\epsilon$, $l,t,r\in I$ and $-s\in \Lambda$ converges to the corresponding subset of $\Merge(X)$. 

\begin{prop}[Convergence of large merges]\label{prop:large_merges}
Consider a probability space in which $(X^{n,\lambda})_{\lambda\in \R}\to (X^\lambda)_{\lambda\in \R}$ almost surely. Then, in probability,
\[\Merge_n((X^{n,\lambda})_{\lambda\in \R}) \xrightarrow[n\to\infty]{} \Merge(X)\,.\]
\end{prop}

% \begin{alert}Move in technical part about local minima in appendix?
% \end{alert}

% \begin{alert}Can probably simplify: This is convergence as a point process for the Hausdorff distance.
% \end{alert}

\begin{proof}We will use the following fact: a.s., there does not exist three local minima of $t, t', t''\in \sL(X)$ such that the points $(t,X_t)$, $(t',X_{t'})$ and $(t'', X_{t''})$ all lie on the same line; using the representation in \cite{BrMa2015a}, this is essentially equivalent to the fact that the standard multiplicative coalescent is binary. To see that this is indeed the case, note that local minima of a continuous function are also global minima on an interval with rational extremities; then for, three disjoint intervals $[a,b]$, $[a',b']$ and $[a'',b'']$ with rational extremities, the local minima $(t,X_t)$, $(t',X_{t'})$ and $(t'', X_{t''})$ on each of these intervals have a law which absolutely continuous with respect to the Lebesgue measure on $\R^2$, and then, are aligned with probability zero; the union of this countable number of zero probability events also has probability zero. Assume that $(a,b,c,\lambda)\in \Merge(X)$. In this case, $b$ is a local minimum of $X^\lambda$, and $X^\lambda$ is strictly above the (horizontal) line connecting $(a,X_a^\lambda),(b,X_b^\lambda),(c,X_c^\lambda)$ on $(a,b)\cup (b,c)$; because of the property recalled above, since $(a,X_a^\lambda)$ and $(b,X_b^\lambda)$ are local minima of $X^\lambda$, it is (a.s.) not the case for $(c,X_c^\lambda)$, so that, $\inf\{X^\lambda_t: t\in (c,c+\eta)\}<0$ for any $\eta>0$.

Given compact intervals $I\subset (0,\infty)$, $\Lambda\subset \R$ and a threshold $\epsilon>0$, there are only finitely many points in $\Merge(X)\cap I^3\times \Lambda$ with the first three coordinates at least $\epsilon$ apart, and it suffices to consider each one separately. 
Take some $(a,b,c,\lambda)\in \Merge(X)$, so that, in particular $X^\lambda_a=X^\lambda_b=X^\lambda_c$. Consider, for  $\varepsilon'>0$, $\delta>0$, the event
% \begin{align*}
%   \cE((a,b,c,\lambda); \varepsilon',  \delta)
%   =& \{m^\lambda(0,a-\varepsilon')\geq X^\lambda_a+\delta\}\\
% % &\cap&     \{ M^t(a-\varepsilon',a)\geq X^t_a+\delta,  M^t(a,a+\varepsilon')\geq X^t_a+\delta\}\\
%   &\cap\{m^\lambda(a+\varepsilon',b-\varepsilon')>\delta\}\\
% %  &\cap&\{M^t(b-\varepsilon',b)\geq X^t_b+\delta,M^t(b,b+\varepsilon')\geq X^t_b+\delta\}\\
%   & \cap\{m^\lambda(b+\varepsilon',c-\varepsilon')>\delta\}\\
% &\cap \{M^\lambda(c-\varepsilon',c)\geq X^\lambda_c+\delta, m^\lambda(c,c+\varepsilon')\leq X^\lambda_c-\delta\}\,.
% \end{align*}
\begin{align*}
  \cE((a,b,c,\lambda); \varepsilon',  \delta)
  =& \Big\{\inf_{[0,a-\varepsilon']}X^\lambda \geq X^\lambda_a+\delta\Big\}
  \cap \Big\{\inf_{[a+\varepsilon', b-\varepsilon']} X^\lambda> X_a^\lambda+\delta\Big\}\\
  & \cap\Big\{\inf_{[b+\varepsilon',c-\varepsilon']} X^\lambda > X^\lambda_a + \delta\Big\}
\cap \Big\{\inf_{[c,c+\varepsilon']} X^\lambda < X^\lambda_c-\delta\Big\}\,.
\end{align*}
Fix $\varepsilon'>0$, $a,b,c\in (0,\infty)$ such that $|b-a|, |c-b|\ge 2\varepsilon'$, and $\lambda\in \R$. Using the properties of the local minima of the Brownian motion, for any $\epsilon>0$ there exists a $\delta>0$ such that
\[\pc{\cE((a,b,c,\lambda); \varepsilon',\delta )~|~(a,b,c,\lambda)\in \Merge(X)}\geq 1-\epsilon.\]
% Now, since $X^n$ converges to $X$ (\JF{on démontre la convergence dans $D(\R,C([0,+\infty),\mathbb{R}))$, theo.5,  mais dans la tension, on démontre une convergence plus forete, dans la section 6.4, car on démontre que c'est tendu  $(X^{n,\lambda}_x)$ comme fonction en $\lambda$ et $x$, sur les compacts, avec le critère du module de continuité dans $C[\lambda_1,\lambda_2]\times[0,A]$}).
Let us show that, the event $\cE((a,b,c,\lambda);\varepsilon', \delta)$, for all $n$ large enough, there must exist some vector $(a_n,b_n,c_n,\lambda_n)$ close to $(a,b,c,\lambda)$ such that $(a_n,b_n,c_n,\lambda_n)\in \Merge_n((X^{n,\lambda})_{\lambda\in \R})$. To prove this, it suffices to show that for some $\lambda'<\lambda$ close enough to $\lambda$ there are points $(a_n,b_n,c_n)\in Z^{n,\lambda'}$ which are close to $(a,b,c)$, while for some $\lambda''>\lambda$ close enough to $\lambda$ there are points $a_n'',b_n''\in Z^{n,\lambda''}$ with $(a_n'',c_n'')$ close to $(a,c)$ but no other point of $Z^{n,\lambda''}$ between $a_n''$ and $c_n''$. We now proceed with the details. 

On the one hand, for any $\lambda'<\lambda$ there exists $\delta'>0$ small enough such that $X^{\lambda}_b<X^{\lambda'}_a-\delta'$ and $X^{\lambda'}_c<X^{\lambda'}_b-\delta'$. Taking $\lambda'$ close enough to $\lambda$ and $\delta'>0$ even smaller, we may also ensure that $X^{\lambda'}>\delta'$ on $[a+\varepsilon', b-\varepsilon']$ and $[b+\varepsilon', c-\varepsilon']$. The convergence of $X^{n,\lambda'}$ to $X^{\lambda'}$ then ensures that we may find $a_n,b_n, c_n\in Z^{n,\lambda'}$ all within distance $\varepsilon'$ of $a$, $b$ or $c$. 

On the other hand, for any $\lambda'>\lambda$, we have $X^{\lambda'}>\delta'$ on $(a+\varepsilon',c-\varepsilon')$; we may take $\lambda'$ close enough to $\lambda$, and $\delta'>0$ small enough such that we also have $\inf\{X^{\lambda'}: [c,c+\varepsilon']\}<\inf\{X^{\lambda'}: [0,a-\varepsilon']\}-\delta'$. The convergence of $X^{n,\lambda'}$ to $X^\lambda$ then ensures that there exists $a_n\in [a-\varepsilon', a+\epsilon']$ and $c_n\in [c-\varepsilon',c+\varepsilon']$ such that $a_n,c_n\in Z^{n,\lambda'}$ and $Z^{n,\lambda}$ does not have any other point between $a_n$ and $c_n$. 

Consider now a accumulation point $(a,b,c,\lambda)$ of $\Merge_n((X^{n,\lambda})_{\lambda\in \R})$ with $a<b<c$. 
Then, there exists a sequence $(a_n,b_n,c_n,\lambda_n)$ converging to $(a,b,c,\lambda)$ with $a_n,b_n,c_n\in Z^{n,\lambda_n}$. It follows that 
\[X^{n,\lambda_n}_{a_n}=\underline X^{n,\lambda_n}_{a_n}=X^{n,\lambda_n}_{b_n}+n^{-1/3} = X^{n,\lambda_n}_{c_n}+2n^{-1/3}\,.\] The convergence of $X^{n,\lambda}$ to $X^\lambda$ then implies that $\underline X^\lambda_a=X^\lambda_a=X^\lambda_b=X^\lambda_c$, so that $a,b,c\in Z^\lambda$. The path properties of $X$ also imply that, for any $\lambda'>\lambda$, $(a,c)\cap Z^{\lambda'}=\varnothing$ so that $(a,b,c,\lambda)\in \Merge(X)$. This completes the proof.
\end{proof}

The next step concerns the convergence of the representation of the cyclic edges. Given $X$, let $\Xi$ be a Poisson point process with intensity $\I{x<y}  \I{(x,y)\cap Z^\lambda= \varnothing} dx dy d\lambda$ on $\R_+^2\times \R$. 

\begin{prop}[Convergence of cyclic edges]\label{pro:cyclic_edges}
We have the convergence in distribution
\[((X^{n,\lambda})_{\lambda \in \R}, \Xi_n) \xrightarrow[n\to\infty]{d}((X^\lambda)_{\lambda\in \R}, \Xi)\,.\]
Furthermore, for any $a\in \R_+$ and $\lambda\in \R$, we have jointly,
\[\#\Xi_n \cap \Big([0,a]^2 \times (-\infty, \lambda] \Big) 
\xrightarrow[n\to\infty]{d} 
\#\Xi \cap \Big([0,a]^2 \times (-\infty,\lambda]\Big)\,.\]
\end{prop}
\begin{proof}Consider first the larger point process $\Xi^\circ_n=\{(in^{-2/3}, jn^{-2/3}, \lambda^n_{ij}): 1\le i<j\le n\}$. For every set $A\subset \{(x,y,\lambda): a\le x<y\le b, \lambda_1 \le \lambda \le \lambda_2\}$, $|\Pi_n \cap A|$ is a binomial random variable with parameters asymptotic to $n^{4/3}(b-a)^2/2$ and $\pc{\lambda_{ij}^n\in[\lambda_1,\lambda_2]}= (\lambda_2-\lambda_1)n^{-4/3}$, and thus converges to a Poisson random variable $P_A$ with parameter $(\lambda_2-\lambda_1)(b-a)^2/2$. 

Let now $A$ and $A'$ be any two disjoint such sets; we show that the distributional limits $P_A$ and $P_{A'}$ are independent. If the space intervals $[a,b]$ and $[a',b']$ are disjoint, this is straightforward since the random variables $(Y_{ij})$ involved in the definitions of $\Pi_n$ on $A$ and $A'$ are themselves independent. Otherwise the space intervals do intersect, and the time intervals must then be disjoint. Fix any $\epsilon>0$. There is some constant $K$ such that $\pc{|\Pi_n \cap A| > K}<\epsilon$. Now, conditionaly on $\Pi_n\cap A$ and $|\Pi_n \cap A|\le K$, to estimate the distribution of $|\Pi_n\cap A'|$ we shall remove the pairs $(i,j)$ which correspond to a point in $\Pi_n \cap A$, and correct the probability of every other to account for the fact that they did not occur in $[\lambda_1,\lambda_2]$ (for those that indeed intersect). This removes only at most $K$ out of the $n^{4/3}(b'-a')^2/2$ pairs, and boots the probability of some of others by a factor $(1-n^{-4/3}(\lambda_2-\lambda_1))$. Overall, the limit remains Poisson random variable with the same distribution. Since $\epsilon>0$ was arbitrary, this proves that $\Xi^\circ_n$ converges to a Poisson point process $\Xi^\circ$ with unit rate on $\{(x,z,\lambda)\in \R_+^2\times \R:x<y\}$.

For the remainder of the proof, we consider now a probability space on which $(X^{n,\lambda})_{\lambda\in \R}$ converges almost surely to $(X^\lambda)_{\lambda\in \R}$. To complete the proof, it now suffices to show that, the set $A^n=\{(in^{-2/3}, jn^{-2/3}, \lambda): Z^{n,\lambda}\cap \{i+1,\dots, j\} = \varnothing\}$ used to filter the points of $\Xi^\circ_n$ converges to $A=\{(x,y,\lambda): Z^\lambda \cap (x,y) =\varnothing\}$ used to filter those of $\Xi^\circ$, for the Hausdorff distance. Indeed, since $Z^{n,\lambda}$ and $Z^\lambda$ are both decreasing in $\lambda$ this would imply the convergence of their Lebesgue measures. Let $(x,y,\lambda)\in A$, then $(x,y)\cap Z^\lambda = \varnothing$, so that for any $\epsilon>0$ small enough, $\inf\{X^\lambda_s - \underline X^\lambda_s: s\in (x+\epsilon, y-\epsilon)\} > 0$. It follows that $\inf\{X^{n,\lambda}_s - \underline X^{n,\lambda}_s: s\in (x+\epsilon, x-\epsilon)\}>0$ as well for all $n$ large enough, so that $(x+\epsilon,y-\epsilon, \lambda)\in A^n$. Similarly, if $(x,y,\lambda)\not\in A$, then $\inf\{X_s^\lambda-\underline X_s^\lambda: s\in (x+\epsilon,y-\epsilon)\}\le 0$ for every $\epsilon>0$ small enough, and therefore $\inf\{X^{\lambda-\epsilon}_s -\underline X^{\lambda-\epsilon}_s: s\in (x+\epsilon,y-\epsilon)\}<0$ for some $\epsilon>0$ small enough. The argument we used above implies that $(x+\epsilon,y-\epsilon,\lambda-\epsilon)\not\in A^n$ for all $n$ large enough. 

For the second claim, one only needs the additional tightness of the number $N_n$ of points of $\Xi_n$ in $[0,a]^2 \times (-\infty,\lambda]$. Since the discrete representation of $\Xi_n$ is delicate to handle, we shall change the point of view: By the exact distribution of $N_n$ in Proposition~\ref{prop:random_graph}, $N_n$ is dominated by the number of surplus edges in the connected components that have nodes with Prim ranks at most $an^{2/3}$ at time $p_n(\lambda)$. The latter is known to be tight by the results in \cite{BrMa2015a} (Corollary~20 and Section 7.2), which consider the alternative representation for the surplus edges using a Bernoulli pointset under the graph of the discrete reflected process $X^{n,\lambda}-\underline X^{n,\lambda}$ (just as in the results of Aldous in \cite{Aldous1997}). 
\end{proof}

We may also recast the results of \cite{BrMa2015a} about the convergence of surpluses of connected components in the Prim order in the present setting (see also, \cite{Aldous1997}). Recall the discrete defined in \eqref{eq:discrete_surplus}. There is a continuum analog to the surplus of a connected component, that can be defined in terms of $\Xi$. At time $\lambda\in \R$, the connected components correspond to the intervals of $\R_+\setminus Z^\lambda$, that are sorted in decreasing order as $(\gamma^\lambda_j)_{j\ge 1}$. We define
\begin{equation}\label{eq:continuum_surplus}
\surp^\lambda_j=\#\big\{(x,y,\lambda') \in \Xi: x,y\in \gamma^\lambda_j, \lambda'\le \lambda \big\}\,.
\end{equation}

\begin{cor}\label{cor:convergence_surplus}Jointly with the convergence in Proposition~\ref{pro:cyclic_edges}, for any fixed $\lambda\in \R$, and any $j\ge 1$, we have 
\[\surp_j^{n,\lambda} \xrightarrow[n\to\infty]{d} \surp_j^{\lambda}\,.\]
\end{cor}

% \begin{alert}ATTENTION the previous proof is not complete; a priori, the rate depends on $m$: Need to verify that most edges are indeed non-edges, since the total rate is $l^2/2$, where $l$ is the length of the interval $N\sim n^{2/3} l$; and this would not be this rate if $m$ happens to be of order $N^2$.

% This seem to also assume something about the sizes of the connected components. Question: what does this convergence really say ? Only stuff about cc of size at least $\epsilon n^{2/3}$ ?

% One way: in discrete, this has exactly the distribution as in \cite{AdBrGo2012a}, and thus this at time $\lambda$, we have an upper bound by a Bernoulli process under the curve. This should suffice to prove that $o(N^2)$ spots are occupied, and thus the PPP has indeed the correct intensity. 
% \end{alert}

% \begin{alert}Coupling of the uniforms. and definition of the discrete object for the proof.
	
% \end{alert}

% \subsubsection{Synchronisation of the family $(\bU)$}
% \label{sec:Sf}

We shall also want to couple the collection of uniform choices that are associated to each merge, and used to sample the edges in discrete and the point identifications in the limit.
In the limit, merge events are characterized by some element $(x,y,z,\lambda)$ where $0<x<y<z$ and $\lambda\in \R$. Let $\Delta:=\{(x,y,z): 0<x<y<z\}\subseteq \R^3_+$. We shall discretize the set of merge events. 
We decompose the collection of all potential merge triples $\Delta$ into countably many cells (up to a Lebesgue null set), so that each cell can contain at most one of the elements $(x,y,z)\in \Delta$ for which there exists some $\lambda$ and $(x,y,z,\lambda)\in \Merge(X)$. 

For $x>0$, let 
%$K(x)=\lceil \log_2(1/x)\rceil -1$; then 
$K(x)$ be the unique integer $k\in \Z$ such that $x\in [1/2^{k+1},1/2^k)$. For $k\in \Z$ and $(x,y,z)\in \Delta$ define
\[D(x,y,z; k)= \big(\floor{x2^k}/2^k,\floor{(y-x)2^k}/2^k,\floor{(z-y)2^k}/2^k\big). \]
Then, the integer $M(x,y,z)= 1+\max\{K(y-x), K(z-y)\}$ is used as the precision at which we shall encode the triple $(x,y,z)$.
For each $(x,y,z)\in \Delta$, define
\[\Code(x,y,z) =(M(x,y,z), D(x,y,z; M(x,y,z))\,,\]
and observe that $\Code$ takes its values in the countable set $\N\times \mathbb Q^3$.
A subset $S\subset \Delta$ is called nested if for any $(x,y,z),(x',y',z')\in S$ we have either (a) $(x,z)\cap (x',z')=\varnothing$, or (b) $(x',z')$ is contained in either $(x,y)$ or $(y,z)$, or (c) $(x,z)$ is contained in either $(x',y')$ or $(y',z')$. The set $\Merge(X)$ is such that its projection on the first three coordinates is nested. One easily verifies that, the map $\Code$ is injective on any nested subset $S\subset \Delta$. 
\begin{lem}Let $S\subseteq \Delta$ be nested. Then for any $(x,y,z),(x',y',z')\in S$ we have $\Code(x,y,z)\ne \Code(x',y',z')$.
\end{lem}
\begin{proof}Suppose that $\Code(x_1,y_1,z_1)=\Code(x_2,y_2,z_2)$. Then, there is $k\in \Z$ and $(a,b,c)\in \R_+^3$ such that for $i\in \{1,2\}$, we have
\begin{compactenum}[i)]
	\item $M(x_i,y_i,z_i)=k$, so that $\min\{|y_i-x_i|,|z_i-y_i|\}\in [1/2^k, 1/2^{k-1})$, and
	\item $x_i\in [a, a+1/2^k)$, $y_i-x_i\in [b, b+1/2^k)$ and $z-y\in [c, c+1/2^k)$.
\end{compactenum}
Now, since $S$ is nested, and there are two alternatives. Either $(x_1,z_1)\cap (x_2,z_2) = \varnothing$, and without loss of generality, $z_1\le x_2$, and because of i) we have $|x_1-x_2|\ge 1/2^{k-1}$ which contradicts ii). Or, without of generality, $(x_1,z_1)\subset [x_2,y_2]$ or $(x_1,z_1)\subset [y_2,z_2]$, and either way, because of i), we cannot have simultaneously all the inequalities in ii).
\end{proof}

The following lemma is straightforward:
\begin{lem}\label{lem:code_cv}
Suppose that $(a_n,b_n,c_n)_{n\ge 1}$ converges to $(a,b,c)\in \Delta$ such that neither $a$, $b-a$ nor $c-b$ are multiple of some $2^j$, with $j\in \Z$. Then, for all $n$ large enough $\Code(a_n,b_n,c_n)=\Code(a,b,c)$.
\end{lem}

From now on, we will see $\bU$ as indexed by this $\Code(\Delta)$ (which is countable) rather than by~$\N$. 

\subsection{A global coupling and ideal forests} % (fold)
\label{sub:global_coupling}

We are now ready to define the probability space on which we will work. By iterative applications of Skorohod's representation theorem, we can find a probability space in which we have the following almost sure convergences, as $n\to\infty$:
\begin{compactitem}
	\item $(X^{n,\lambda})_{\lambda \in \R}\to (X^{\lambda})_{\lambda\in \R}$;
	\item $\Xi_n \to \Xi$; 
	\item $\Merge_n((X^{n,\lambda})_{\lambda\in \R}) \to \Merge(X)$.
\end{compactitem}
\medskip

Our starting point is the following: for the discrete objects, we consider the Kruskal forest $K(n,p)$ and random graphs $G(n,p)$ that are constructed from $(X^{n,\lambda})_{\lambda\in \R}$, $\bU$, and $\Xi_n$, which also define the metric space $(M_n,d_n)$. The continuous objects are those built from $X$, $\bU$ and $\Xi$, in particular the metric $d$ on $\R_+$. In the following, we shall identify the label and the Prim ranks to ease the discussion. 

These objects are crucial for us, and we will show that their macroscopic structures are similar. Rather than trying to couple details at the scale $\delta n^{2/3}$ in the discrete, and $\delta$ in the continuous, we shall proceed as follows: both in the discrete and continuous setting, we can see the metric spaces at some given time $\lambda$ (that is $K(n,p_n(\lambda))$ and $G(n,p_n(\lambda))$ in discrete, and the metric spaces induced by the intervals of $\R_+\setminus Z^\lambda$ in continuous) as combining together the metric spaces that were already present far in the past, say at some time $\underline \lambda<\lambda$. We will never look any further in time, and replace the metrics in the connected components at time $\underline \lambda$ by some idealization. In general, the distribution will be incorrect, but we will ensure that $\underline \lambda$ can be chosen far enough for the distributions to be exact (or close enough) on an event of arbitrarily large probability. Observe that, this modification at time $\underline \lambda$ provides a coupling at all times $\lambda>\underline \lambda$ simultaneously. 

Fix any two points $s_1,s_2\in (0,\infty)$ and define $s_i^n=\lfloor s_i n^{2/3}\rfloor$, $i=1,2$. There is always some deterministic $\overline \lambda$ large enough such that $s_1^n$ and $s_2^n$ lie in the same connected component of $K(n,p_n(\overline \lambda))$ for $n$ large enough with probability close to one; the path between $s_1^n$ and $s_2^n$ we refer to is the one in this connected component (and at any larger time). For any $\lambda\in \R$, let $J^n_\lambda(s_1,s_2)$ denote the collection of ranks of the connected components of $G(n,p_n(\lambda))$ that contain some node on the path between $s_1^n$ and $s_2^n$; $j\in J^n_\lambda(s_1,s_2)$ means that the $j$th largest connected component of $K(n,p_n(\lambda))$ contains some node of the path between $s_1^n$ and $s_2^n$. Similarly, let $J_\lambda(s_1,s_2)$ be the collection of indices of the intervals $(\gamma^\lambda_j)_{j\ge 1}$ obtained as $\R_+\setminus Z^\lambda$ which contain part of $\llbracket s_1, s_2 \rrbracket$. By construction, any of the connected components at time $\lambda$ are traversersed by a single portion of the path, between two points that we will denote by $a_j^n$, $b^n_j$ and $a_j, b_j$ respectively. 
Then, we have the following exact decompositions:
\begin{align}\label{eq:dist_decomp_exact-disc}
d_n(s_1^n,s_2^n) = \sum_{j\in J_\lambda^n(s_1,s_2)} d_n(a_j^n, b_j^n) + \#J_\lambda^n(s_1,s_2)-1\,,
\end{align}
and 
\begin{align}\label{eq:dist_decomp_exact-cont}
d(s_1,s_2)=\sum_{j \in J_\lambda(s_1,s_2)} d(a_j,b_j)\,.
\end{align}

We will simply replace the distances in the components at time $\lambda$ by what they should be in an ideal situation; for now, we are only interested in the definition, the verifications will come later. 
With this goal in mind, let us suppose our probability space contains the following sequences of random variables. 
Let $(V_j)_{j\ge 1}$ be i.i.d.\ random variables uniform on $[0,1]$. For each $m\ge 1$, let $F_m$ denote the distribution function of the distance $D_m$ between two independent uniformly random points in a uniformly random labelled tree on $m$ nodes. Then, for each $m\ge 1$, $\bar D_j(m)=m^{-1/2} F_m^{-1}(V_j)$. This provides a sequence of random variables where each term $\bar D_j(m)$ is distributed like $m^{-1/2} D_m$, and that converges almost surely as $m\to\infty$ to a Rayleigh random variable $\bar R_j$ with density $xe^{-x^2/2}$ on $\R_+$ (see, e.g., \cite{Aldous1991b} for the convergence in distribution). 

The objective is to control the matrix of pairwise distances between multiple points, and our new approximation of the distance will depend on the entire set of points. Let $\bs=(s_1,s_2,\dots, s_k)\in \R_+^k$. We will only replace the distance in the connected components that only contain a single portion of paths between these points; in all the other components, which contain branch points of the collection of paths between the elements of $\bs$, we will keep the distance unchanged. Let 
\[J^n_\lambda(s_p,s_q; \bs) = J^n_\lambda(s_p,s_q) \setminus \bigcup_{i<j, i\ne p, j\ne q} J^n_\lambda(s_i,s_j)\,,\]
and similarly, define the continuum analog by
\[J_\lambda(s_p,s_q; \bs) = J_\lambda(s_p,s_q) \setminus \bigcup_{i<j, i\ne p, j\ne q} J_\lambda(s_i,s_j)\,.\]
Let $(n^{2/3}|\gamma_j^{n,\lambda}|)_{j\ge 1}$ denote the collection of sizes of the connected components at time $\lambda$, just as $(|\gamma^\lambda_j|)_{j\ge 1}$ denotes the Lebesgue measures in the continuous setting. Define the following approximations, for $1\le p<q\le k$,  
\begin{align}\label{eq:distance_discrete_comp}
	\tilde d_n(s_p^n, s_q^n) 
	= &~ n^{1/3} \sum_{j\in J^n_{\lambda}(s_p,s_q;\bs)} |\gamma^{n,\lambda}_j|^{1/2} \bar D_j(n^{2/3}|\gamma^{n,\lambda}_j|)  \\ 
	& + \sum_{j\in J^n_\lambda(s_p,s_1)\setminus J^n_\lambda(s_p,s_1; \bs)} d_n(a_j^n,b_j^n) + \# J^{n}_{ \lambda}(s_1,s_2) - 1\,,\notag
\end{align}
and 
\begin{equation}\label{eq:distance_continuous_comp}
	\tilde d(s_p,s_q) = \sum_{j\in J_{\lambda}(s_p,s_q;\bs)} |\gamma_j^{\lambda}|^{1/2} \bar R_j
	+ \sum_{j\in J_{\lambda}(s_p,s_q)\setminus J_\lambda(s_p,s_q;\bs)} d(a_j,b_j)\,.
\end{equation}

We first verify that these provide a suitable coupling of the pairwise distances between the $k$ points $s_1^n,s_2^n,\dots, s_k^n$ and $s_1,s_2,\dots, s_k$, respectively. 
\begin{prop}\label{pro:distribution_coupling}
Fix some compact interval $I\in (0,\infty)$. For any $\epsilon>0$, there exists $\lambda\in \R$ and an event $A$ of probability at least $1-\epsilon$, such that, for any $k$ i.i.d.\ uniform points $s_1,s_2, \dots, s_k \in I$, for all $n$ large enough, on the event $A$, we have
\[(\tilde d_n(s_p,s_q))_{1\le p<q\le k} \eqdist (d_n(s_p^n,s_q^n))_{1\le p<q\le k} 
\quad \text{and} \quad 
(\tilde d(s_p,s_q))_{1\le p<q\le k} \eqdist (d(s_p,s_q))_{1\le p<q\le k}\,.\]
\end{prop}
\begin{proof}There exists a $\lambda_1$ large enough that $I$ is contained in a single connected component at time $\lambda_1$: with Lemma~\ref{lem:position_Hlambda} in mind, let $\lambda_1$ be the smallest $\lambda$ for which $\sup I<2\lambda +1$ and $\exp(-c\lambda)<\epsilon/2$. This value being fixed, $I$ is contained in $[0,2\lambda_1+1]$ with probability at least $1-\epsilon/2$. 

Let $A=A_\lambda$ be the event that $\Xi$ does not contain any point in $[0,2\lambda_1+1]$ with time lower than $\lambda$. From the correspondence between the intensity of $\Xi$ and the area of $X^\lambda-\underline X^\lambda$, there exists $\lambda$ small enough that $A$ has probability at least $1-\epsilon$. The convergence of $\Xi_n$ to $\Xi$ implies that, on this event, for all $n$ large enough, $\Xi_n$ also has no point in $I$ with times before $\lambda$ (Proposition~\ref{pro:cyclic_edges}). 

We may choose $\lambda$ even smaller to ensure that, $|\gamma^\lambda_1|<\epsilon/(4k)$, so that, the probability that some point $s_i$, $1\le i\le k$, falls in an interval of $\R_+\setminus Z^\lambda$ that is not fully contained in $I$ is at most $\epsilon/(2|I|)$. When this occurs, conditionally on $s_i\in \gamma^\lambda_j$, the position $s_i$ is uniformly random in $\gamma^\lambda_j$. The same holds true for the discrete counterparts $s_i^n$ for all $n$ large enough. 

Now, on the event $A$, for all $n$ large enough, all the connected components of the random graph $G(n,p_n(\lambda))$ containing nodes with label at most $(2\lambda_1 +1)n^{2/3}$ are all trees, which are uniformly random. These are thus identical to the components in the Kruskal forest $K(n,p_n(\lambda))$. By Proposition~\ref{prop:random_forest} and Lemma~\ref{lem:dist_left-most-node} the points $a_j^n$ and $b_j^n$ are independent and uniformly random (their actual labels!) and independent of the component. Since the end points $s_i^n$ are themselves uniformly random in the connected component in which they lie, by the previous paragraph, this proves that, on $A$, the discrete approximation $\tilde d(s_p,s_q)$, $1\le p<q\le k$, has the same distribution as $d(s_p,s_q)$, $1\le p<q\le k$. 

The continuous analog follows from the calculations in Section~2.1 of \cite{AdBrGo2010} saying that, conditionally on having no point under the curve, an excursion under $\tilde \bn_\sigma$ is distributed according to $\bn_\sigma$, and is thus exactly a Brownian excursion; Brownian scaling and Remark~\ref{rem:one-point_function} saying that in $\CMT(\exc, \bU)$, the distance $d_\exc(0,V)$ between $0$ and a uniformly random point $V$ is Rayleigh distributed, which completes the proof.
\end{proof}

\subsection{Main proof of convergence} % (fold)
\label{sub:convergence}

Finally, we are ready to prove that, in the probability space defined in the previous section, we have convergence in probability of the pairwise distance. 
\begin{prop}\label{pro:convergence_distances}Fix $I$ a compact interval of $(0,\infty)$, and let $s_1,s_2,\dots, s_k$ be i.i.d.\ uniform points in $I$. For any $\epsilon,\delta>0$, there exists $\lambda\in \R$ such that, 
\[\limsup_{n\to\infty} \p{\sup_{1\le p<q\le k} \big|n^{-1/3}\tilde d_n(s_p^n,s_q^n)-\tilde d(s_p,s_q)\big|> \delta} \le \epsilon\,.\]	
\end{prop}
\begin{proof}
Fix any $\epsilon,\delta>0$.
Let us first deal with the portions of paths contained in connected components that are traversed by more than one path, and that we did not bother coupling. 
Consider the event $A$ in Proposition~\ref{pro:distribution_coupling} and the corresponding value $\lambda_1$ for $\lambda$ which ensures that $\pc{A^c}\le \epsilon/4$. By Lemma~\ref{lem:global_asympt_small_lambda}, there exists $\lambda_2$ such that, for all $\lambda\le \lambda_2$ and all $n$ large enough, the probability that the maximum diameter of a connected component of $G(n,p_n(\lambda))$ is larger than $n^{1/3}\delta/(3k)$ is at most $\epsilon/4$. 
Furthermore, on the event $A$, each one of the $k-1$ portions of continuum paths in the intervals $\gamma_j^\lambda$ which contain more than one portion has a length stochastically dominated by $|\gamma_j^\lambda|\bar R_j$ (or the diameter of the corresponding CRT). We can choose $\lambda_3$ small enough such that the probability that any of them is greater than $\delta/(3k)$ is at most $\epsilon/4$. Fix $\lambda=\min\{\lambda_1,\lambda_2,\lambda_3\}$. 
Finally, by Proposition~\ref{pro:local_asympt_small_lambda}, for this value of $\lambda$, there is some $\delta'>0$ small enough such that, with probability at least $1-\epsilon/4$ all the connected components $C^{n,\lambda}_j$, $j\in J^n_\lambda(s_p,s_q)$, $1\le p<q\le k$, contain at  least $\delta'n^{2/3}$ nodes. The probability that either of these bad events occur is at most $\epsilon$, and we now suppose we work on the event $A'$ that none occurs. 

On the event $A'$, we have from \eqref{eq:distance_discrete_comp} and \eqref{eq:distance_continuous_comp}, for any $1\le p<q\le k$, 
\begin{align*}
	\big|n^{-1/3}\tilde d_n(s_p^n,s_q^n)-\tilde d(s_p,s_q)\big|
	& \le  \Bigg| \sum_{j\in J^n_\lambda(s_p,s_q;\bs) } n^{-1/3} |\gamma^{n,\lambda}_j|^{1/2} \bar D_j(n^{2/3}|\gamma^{n,\lambda}_j|) - \sum_{j\in J_\lambda(s_p,s_q;\bs) }|\gamma_j^{\lambda}|^{1/2} \bar R_j\Bigg|\\
	& \quad + n^{-1/3} \# J^n_\lambda(s_p,s_q) + 2\delta/3\,.  
\end{align*}
Since the $\#J^n_\lambda(s_p,s_q)$ are all tight by the proof of Proposition~\ref{pro:local_asympt_small_lambda}, we only need to deal with the first term the right-hand side above. 

We claim that the fact that all discrete connected components $C^{n,\lambda}_j$, for $j\in J^n_\lambda(s_p,s_q)$ for some $1\le p<q\le k$ contain at least $\delta'n^{2/3}$ nodes, the convergence of the merge events implies that, for all $n$ large enough, we have $J^n_\lambda(s_p,s_q)=J_\lambda(s_p,s_q)$ for every $1\le p<q\le k$. The reason is the following: (1) for all $n$ large enough, for every $i$, if $s_i\in \gamma^\lambda_j$, then $s_i^n\in C^{n,\lambda}_j$, because $\{s_1,\dots, s_k\}$ and $Z^\lambda$ are almost surely disjoint. (2) The merges of large connected components do converge because $\Merge((X^{n,\lambda})_{\lambda\in \R})\to \Merge(X)$. (3) The points random points constructed in the discrete and continuuous model for matching merges use the same uniforms by Lemma~\ref{lem:code_cv}. It follows that, for $n$ large enough, these points themselves end up in matching pair of discrete and continuum components. (4) The number of such merges is finite (the $\#J^n_\lambda(s_p,s_q)$, $J_\lambda(s_p,s_q)$ are tight). As a consequence, for all $n$ large enough, we are lead to bounding
\[\p{\Bigg|
\sum_{j\in J_\lambda(s_p,s_q;\bs)} n^{-1/3} |\gamma^{n,\lambda}_j|^{1/2} \bar D_j(n^{2/3}|\gamma^{n,\lambda}_j|) - |\gamma_j^{\lambda}|^{1/2} \bar R_j \Bigg| > \delta/3}\,,\]
but we our coupling precisely ensures that every single term of the sum converges almost surely to zero. This completes the proof.
\end{proof}

% subsection convergence (end)

\subsection{Remaining proofs of convergence} % (fold)
\label{sub:remaining_proofs}

Finally, we rely on the results of the previous section to complete the proofs of the remaining results, namely that of Theorem~\ref{thm:limit_mst_surplus} about the MST of a connected graph with given surplus, and Theorem~\ref{thm:dynamics_X} about the dynamics for the limit random graph and Kruskal processes.

Before going further, let us discuss the types of convergence. Proposition~\ref{pro:convergence_distances} 
implies the convergence of the distribution of the matrix of pairwise distances between any finite number of points, and may thus be used to prove convergence in the Gromov--Prokhorov (GP) sense (Theorem~5 of \cite{GrPfWi2009a}): indeed, for any $\lambda$, restriction of the $d$ to any interval $\gamma^\lambda_j$, $j\ge 1$, is the limit of the metric of the discrete minimum on $C^{n,\lambda}_j$. The reason why this suffices to also prove convergence in the sense of Gromov--Hausdorff--Prokhorov (GHP) is that we actually already know that the sequences are tight for GHP (\cite{AdBrGoMi2013a,AdBrGo2012a}), and that the limit we construct has a mass measure which has full support because of Proposition~\ref{pro:length_measure} (see \cite{AtLoWi2016a}). In the following, we thus only discuss GP convergence.

\begin{proof}[Proof of Theorem~\ref{thm:dynamics_X}]\emph{i)} Since the coupling is global, the proof of the joint convergence of the Kruskal forest $({\frak F}^{n,\lambda_1}, \dots, {\frak F}^{n,\lambda_k})$ at times $\lambda_1<\lambda_2<\dots<\lambda_k$ is an immediate consequence of Proposition~\ref{pro:convergence_distances}, and the above discussion about the GHP versus GP convergence. The connected components at time $\lambda_i$ correspond to the intervals of $\R_+\setminus Z^{\lambda_i}$, equipped with the metric induced by $d$.

\noindent \emph{ii)} For the same reason, the proof of the joint convergence $({\frak G}^{n,\lambda_1}, \dots, {\frak G}^{n,\lambda_k})$ would be complete once we have an analog of Proposition~\ref{pro:convergence_distances} for the random graph at a fixed time. Proving this amounts to verifying that the joint convergence of the minimum spanning tree and of $\Xi_n$ is sufficient to guarantee the convergence of the end points of every single surplus edge. 

Once we have convergence of the end points of the edges, the techniques in \cite{AdBrGo2012a} imply the convergence of the graph. Proving that we indeed have convergence of the locations of the end points of edges is not immediate because the function $d(x,y)$ is not continuous in either $x$ or $y$. However, we can find a small $\underline \lambda\in \R$ such that the points appear between the correct connected components at time $\underline \lambda$ for all $n$ large enough (almost surely, since the points have a diffuse distribution). Since the diameter of these components at time $\underline \lambda$ may be made arbitrarily small by choice of $\underline \lambda$, we do have convergence of the locations of the end points. This completes the proof of the sequence of graphs, in the product topology for a fixed $\lambda$. The extension to a vector of $(\lambda_1,\dots, \lambda_k)$ is immediate using the same arguments as above. 
\end{proof}

\begin{proof}[Proof of Theorem~\ref{thm:limit_mst_surplus}]
Consider the probability space from above, and fix some interval $\gamma^\lambda_i$ of $\R_+\setminus Z^\lambda$. Recall the discrete and continuum surplus defined in \eqref{eq:discrete_surplus} and \eqref{eq:continuum_surplus}, respectively. Furthermore, $\surp^\lambda_i$ is a Poisson random variable with parameter the area of the process $X^\lambda-\underline X^\lambda$ on $\gamma^\lambda_i$. It thus follows from the calculations in Section~2.1 of \cite{AdBrGo2010} that, 
\begin{align*}
\E{f((X^\lambda_{t_0+t}-\underline X^\lambda_{t_0+t})_{0\le t\le \sigma})~\Big|~\gamma_i^\lambda=(t_0, t_0+\sigma),\surp_j^\lambda=s} 
%= \int f(\omega) \frac{(\int \omega(u)du)^s}{\int (\int \omega_1(u)du)^s \bn_{\sigma}(d\omega_1)} \bn_{\sigma}(d\omega)\,. 
&= \frac{\Ec{f(\exc_\sigma) \cdot (\int_0^\sigma \exc_\sigma(u)du)^s}} {\Ec{(\int_0^\sigma \exc_\sigma(u) du)^s}}
\end{align*}
where $\exc_\sigma$ is a Brownian excursion of duration $\sigma$. By definition, the right-hand side above is nothing else than $\Ec{f(\exc_\sigma^{(s)})}$. Furthermore, on the event that $\surp^\lambda_i=s$, by Corollary~\ref{cor:convergence_surplus}, we have $\surp^{n,\lambda}_i=s$ for all $n$ large enough. Therefore, up to a trivial relabelling, $C^{n,\lambda}_i$ is a uniformly random connected component with surplus $s$ and size $\gamma^{n,\lambda}_i$. Since each of the values for $\surp^{\lambda}_i$ has positive probability, Theorem~\ref{thm:limit_mst_surplus} follows from Proposition~\ref{pro:convergence_distances}, and the discussion about the strengthening to Gromov--Hausdorff--Prokhorov convergence. 
\end{proof}

Finally, we prove our main result about the entire minimum spanning tree. In \cite{AdBrGoMi2013a}, it is proved that the scaling limit of the minimum spanning tree can be constructed as the limit as $\lambda\to\infty$ of the scaling limit of the minimum spanning tree of the largest connected component of the random graph at $p_n(\lambda)$. Here, we use the limit as $\lambda \to \infty$ of the connected component containing the vertex with Prim order $\lfloor n^{2/3}\rfloor$. We now verify that this coincides with our definition, which uses a connected component $H_\lambda$containing the point $1$ and the measured metric space $(H_\lambda,d, \hat \mu_\lambda)$, $\hat \mu_\lambda$ is the (image of the) probability measure which is proportional to Lebesgue measure on $H_\lambda$. At this point, this should be essentially straightforward.

\begin{proof}[Proof of Theorem~\ref{thm:limit_mst_Kn}]
Let $\cE^\star_\lambda$ be the event that the largest connected component of $\R_+\setminus Z^\lambda$ contains the point $1$. Observe that, for all $\lambda\ge 2$, with $R_\lambda=\sup H_\lambda$,
\begin{align*}
\pc{\cE^\star_\lambda}
%& = \pc{E^\lambda, |\gamma^\lambda_2|\le 1} + \pc{E^\lambda, |\gamma^\lambda_2|\ge 1} \\ 
%& \ge \pc{E^\lambda, |\gamma^\lambda_2|\le 1}\\ 
& \ge \pc{R_\lambda >2 \lambda-1, |\gamma^\lambda_2|\le 2}\\
& \ge 1 - \pc{R_\lambda\le 2\lambda - 1} - \pc{|\gamma^\lambda_2|\ge 2}.
\end{align*}
Lemma~\ref{lem:position_Hlambda} implies that the first probability in the right-hand side above tends to zero as $\lambda\to\infty$. The same holds for the second one, see for instance, Proposition~5.3 of \cite{AdBrGoMi2019a} which says that $|\gamma^\lambda_2|$ is $O(\lambda^{-2} \log \lambda)$ in probability. This also easily follows from Lemma~\ref{lem:sizes_fragments-k} \emph{i)}: indeed, for any natural number $i\ge 1$, on the event that $R_\lambda> i^3$, we have (with the notation of Section~\ref{sub:statistics_small_components})
\[|\gamma^\lambda_2| \le 1 + \sup_{k\ge i}\sup\{m_{\lambda'}: \lambda'\in \Lambda_k\}\,,\]
which is at most $1+i^{-5}$ with probability at least $1-O(i^{-1/4})$. This implies $\pc{|\gamma^\lambda_2|\ge 2} = O(\lambda^{-1/12})$, and in turn that $\pc{\cE^\star_\lambda}\to 1$ as $\lambda \to \infty$. By Proposition~\ref{pro:convergence_distances}, $(H_\lambda, d, \hat \mu_\lambda)$ is the Gromov--Prokhorov limit (in distribution) of the minimum spanning tree of the connected component containing the vertex with Prim order $n^{2/3}$. However we know by the results of \cite{Ad2013a} that the sequence of rescaled minimum spanning trees converge for the Gromov--Hausdorff--Prokhorov topology, so that the convergence actually holds for GHP. Together with the fact that $\pc{\cE^\star_\lambda}\to 1$ as $\lambda \to \infty$, this proves that $(\sM, d, \mu)$ has the same distribution as $(\sM', d', \mu')$ constructed in \cite{Ad2013a}.	
\end{proof}

% {\red 
% \begin{prop}The measured metric space $\CMT(X, \bU)$ is indeed distributed like $(\sM',\delta',\mu', \rho')$ constructed in $\cite{AdBrGoMi2013a}$
% \end{prop}
% \begin{proof}
% \end{proof}
% }

% subsection remaining_proofs (end)

% subsection the_coupling_and_the_distribution_of_ (end)

% \section{The local limit}
% \label{sec:local_limit}
% \input{local_limit}

%% THE RIGHT-OBJECT: COUPLING
%\input{Representation}

{\small
\addcontentsline{toc}{section}{References}
\setlength{\bibsep}{0pt plus 0.4ex}
\bibliographystyle{plainnat}
\bibliography{MST_brownian}

\begin{thebibliography}{63}
\providecommand{\natexlab}[1]{#1}
\providecommand{\url}[1]{\texttt{#1}}
\expandafter\ifx\csname urlstyle\endcsname\relax
  \providecommand{\doi}[1]{doi: #1}\else
  \providecommand{\doi}{doi: \begingroup \urlstyle{rm}\Url}\fi

\bibitem[Addario-Berry et~al.(2010)Addario-Berry, Broutin, and
  Goldschmidt]{AdBrGo2010}
L.~Addario-Berry, N.~Broutin, and C.~Goldschmidt.
\newblock Critical random graphs: limiting constructions and distributional
  properties.
\newblock \emph{Electronic Journal of Probability}, 15:\penalty0 741--774,
  2010.

\bibitem[Addario-Berry et~al.(2012)Addario-Berry, Broutin, and
  Goldschmidt]{AdBrGo2012a}
L.~Addario-Berry, N.~Broutin, and C.~Goldschmidt.
\newblock {The continuum limit of critical random graphs}.
\newblock \emph{Probability Theory and Related Fields}, 152:\penalty0 367--406,
  2012.
\newblock \doi{10.1007/s00440-010-0325-4}.

\bibitem[Addario-Berry et~al.(2014)Addario-Berry, Broutin, and
  Holmgren]{AdBrHo2014a}
L.~Addario-Berry, N.~Broutin, and C.~Holmgren.
\newblock {Cutting down trees with a Markov chainsaw}.
\newblock \emph{The Annals of Applied Probability}, 24:\penalty0 2297--2339,
  2014.

\bibitem[Addario-Berry et~al.(2017)Addario-Berry, Broutin, Goldschmidt, and
  Miermont]{AdBrGoMi2013a}
L.~Addario-Berry, N.~Broutin, C.~Goldschmidt, and G.~Miermont.
\newblock {The scaling limit of the minimum spanning tree of the complete
  graph}.
\newblock \emph{The Annals of Probability}, 45:\penalty0 3075--3144., 2017.

\bibitem[Addario-Berry et~al.(2019)Addario-Berry, Dieuleveut, and
  Goldschmidt]{AdDiGo2019a}
L.~Addario-Berry, D.~Dieuleveut, and C.~Goldschmidt.
\newblock {Inverting the cut-tree transform}.
\newblock \emph{Annales de l'I. H. P. Probabilit{\'e}s et Statistiques},
  55:\penalty0 1349--1376, 2019.

\bibitem[Addario-Berry et~al.(2023)Addario-Berry, Broutin, Goldschmidt, and
  Miermont]{AdBrGoMi2019a}
L.~Addario-Berry, N.~Broutin, C.~Goldschmidt, and G.~Miermont.
\newblock {Continuum Erd\H{o}s--R\'enyi and Kruskal dynamics}.
\newblock In preparation, 2023.

\bibitem[Addario-Berry(2013)]{Ad2013a}
Louigi Addario-Berry.
\newblock The local weak limit of the minimum spanning tree of the complete
  graph.
\newblock \emph{arXiv preprint arXiv:1301.1667}, 2013.

\bibitem[Addario-Berry and Sen(2021)]{AdSe2021a}
Louigi Addario-Berry and Sanchayan Sen.
\newblock Geometry of the minimal spanning tree of a random 3-regular graph.
\newblock \emph{Probability Theory and Related Fields}, 180\penalty0
  (3):\penalty0 553--620, 2021.

\bibitem[Albenque and Goldschmidt(2015)]{AlGo2015a}
M.~Albenque and C.~Goldschmidt.
\newblock {The Brownian continuum random tree as the unique solution to a fixed
  point equation}.
\newblock \emph{Electronic Communications in Probability}, 20\penalty0
  (61):\penalty0 1--14, 2015.

\bibitem[Aldous(1991{\natexlab{a}})]{Aldous1991}
D.~Aldous.
\newblock The continuum random tree {II}: an overview.
\newblock In M.T. Barlow and N.H. Bingham, editors, \emph{Stochastic Analysis},
  pages 23--70. Cambridge University Press, 1991{\natexlab{a}}.

\bibitem[Aldous(1991{\natexlab{b}})]{Aldous1991b}
D.~Aldous.
\newblock The continuum random tree. {I}.
\newblock \emph{The Annals of Probability}, 19:\penalty0 1--28,
  1991{\natexlab{b}}.

\bibitem[Aldous(1993)]{Aldous1993a}
D.~Aldous.
\newblock The continuum random tree {III}.
\newblock \emph{The Annals of Probability}, 21:\penalty0 248--289, 1993.

\bibitem[Aldous(1994)]{Aldous1994a}
D.~Aldous.
\newblock Recursive self-similarity for random trees, random triangulations and
  {B}rownian excursion.
\newblock \emph{The Annals of Probability}, 22:\penalty0 527--545, 1994.

\bibitem[Aldous(1997)]{Aldous1997}
D.~Aldous.
\newblock Brownian excursions, critical random graphs and the multiplicative
  coalescent.
\newblock \emph{The Annals of Probability}, 25:\penalty0 812--854, 1997.

\bibitem[Aldous and Pitman(1998)]{AlPi1998a}
D.~Aldous and J.~Pitman.
\newblock The standart additive coalescent.
\newblock \emph{The Annals of Probability}, 26:\penalty0 1703--1726, 1998.

\bibitem[Aldous and Limic(1998)]{AlLi1998}
D.J. Aldous and V.~Limic.
\newblock {The entrance boundary of the multiplicative coalescent}.
\newblock \emph{Electronic Journal of Probability}, 3:\penalty0 1--59, 1998.

\bibitem[Armendariz(2001)]{Armendariz2001}
I.~Armendariz.
\newblock \emph{Brownian excursions and coalescing particle systems}.
\newblock Phd thesis, New York University, 2001.

\bibitem[Athreya et~al.(2016)Athreya, L{\"o}hr, and Winter]{AtLoWi2016a}
S.~Athreya, W.~L{\"o}hr, and A.~Winter.
\newblock {The gap between Gromov-vague and Gromov--Hausdorff-vague topology}.
\newblock \emph{Stochastic Processes and their Applications}, 126:\penalty0
  2527--2553, 2016.

\bibitem[Azuma(1967)]{Azuma1967}
K.~Azuma.
\newblock Weighted sums of certain dependent random variables.
\newblock \emph{Tohoku Mathematical Journal}, 37:\penalty0 357--367, 1967.

\bibitem[Bertoin(2000)]{Bertoin2000a}
J.~Bertoin.
\newblock {A fragmentation process connected to Brownian motion}.
\newblock \emph{Probability Theory and Related Fields}, 117:\penalty0 289--301,
  2000.

\bibitem[Bertoin and Miermont(2013)]{BeMi2013a}
J.~Bertoin and G.~Miermont.
\newblock {The cut-tree of large Galton-Watson trees and the Brownian CRT}.
\newblock \emph{The Annals of Applied Probability}, 23:\penalty0 1469--1493,
  2013.

\bibitem[Bhamidi et~al.(2014)Bhamidi, Budhiraja, and Wang]{BhBuWa2014a}
S.~Bhamidi, A.~Budhiraja, and X.~Wang.
\newblock {The augmented multiplicative coaslescent, bounded size rules and
  critical dynamics of random graphs}.
\newblock \emph{Probability Theory and Related Fields}, 160:\penalty0 733--796,
  2014.

\bibitem[Boucheron et~al.(2012)Boucheron, Lugosi, and Massart]{BoLuMa2012a}
S.~Boucheron, G.~Lugosi, and P.~Massart.
\newblock \emph{{Concentration Inequalities - A nonasymptotic theory of
  independence}}.
\newblock Clarendon Press, Oxford, 2012.

\bibitem[Broutin and Marckert(2016)]{BrMa2015a}
N.~Broutin and J.-F. Marckert.
\newblock {A new encoding of coalescent processes. Applications to the additive
  and multiplicative cases}.
\newblock \emph{Probability Theory and Related Fields}, 166:\penalty0 515--552,
  2016.

\bibitem[Broutin and Wang(2017)]{BrWa2017b}
N.~Broutin and M.~Wang.
\newblock {Reversing the cut tree of the Brownian continuum random tree}.
\newblock \emph{Electronic Journal of Probability}, 22\penalty0 (80):\penalty0
  1--23, 2017.

\bibitem[Chassaing and Louchard(2002)]{ChLo2002}
P.~Chassaing and G.~Louchard.
\newblock Phase transition for parking blocks, {B}rownian excursion and
  coalescence.
\newblock \emph{Random Structures \& Algorithms}, 21:\penalty0 76--119, 2002.

\bibitem[Chiswell(2001)]{Chiswell2001}
I.~Chiswell.
\newblock \emph{{Introduction to $\Lambda$-trees}}.
\newblock World Scientific Publishing Company, Singapore, 2001.

\bibitem[Corujo and Limic(2023{\natexlab{a}})]{CoLi2023a}
Josu{\'e} Corujo and Vlada Limic.
\newblock The standard augmented multiplicative coalescent revisited,
  2023{\natexlab{a}}.

\bibitem[Corujo and Limic(2023{\natexlab{b}})]{CoLi2023b}
Josu{\'e} Corujo and Vlada Limic.
\newblock A dynamical approach to spanning and surplus edges of random graphs,
  2023{\natexlab{b}}.

\bibitem[Curien and Haas(2017)]{CuHa2017a}
N.~Curien and B.~Haas.
\newblock {Random trees constructed by aggregation}.
\newblock \emph{Annales de l'Institut Fourier}, 67:\penalty0 1963--2001, 2017.

\bibitem[Durrett and Iglehart(1977)]{DuIg1977}
R.T. Durrett and D.L. Iglehart.
\newblock {Functionals of Brownian meander and Brownian excursion}.
\newblock \emph{The Annals of Probability}, 5:\penalty0 130--135, 1977.

\bibitem[Evans(2005)]{Evans2005}
S.N. Evans.
\newblock \emph{{Probability and real trees, \'Ecole d'\'Et\'e de
  Probabilit\'es de Saint-Flour XXXV-2005}}, volume 1920 of \emph{Lecture Notes
  in Mathematics}.
\newblock Springer, 2005.

\bibitem[Falconer(1986)]{Falconer1986}
K.~J. Falconer.
\newblock \emph{{The Geometry of Fractal Sets}}, volume~85 of \emph{Cambridge
  Tracts in Mathematics}.
\newblock Cambridge University Press, Cambridge, 1986.

\bibitem[Falconer(1990)]{Falconer1990a}
Kenneth Falconer.
\newblock \emph{{Fractal Geometry: Mathematical Foundations and Applications}}.
\newblock John Wiley \& Sons Ltd., Chichester, 1990.

\bibitem[Frilet(2021)]{Frilet2021}
N.~Frilet.
\newblock \emph{{Metric coalescence of homogeneous and inhomogeneous random
  graphs}}.
\newblock Phd thesis, Universit{\'e} Grenoble-Alpes, 2021.
\newblock https://hal.inria.fr/tel-03667362v1.

\bibitem[Graf et~al.(1988)Graf, Mauldin, and Williams]{GrMaWi1988a}
Siegfried Graf, R~Daniel Mauldin, and Stanley~C Williams.
\newblock \emph{{The exact Hausdorff dimension in random recursive
  constructions}}, volume 381.
\newblock American Mathematical Soc., 1988.

\bibitem[Greven et~al.(2009)Greven, Pfaffelhuber, and Winter]{GrPfWi2009a}
A.~Greven, P.~Pfaffelhuber, and A.~Winter.
\newblock Convergence in distribution of random metric measure spaces
  ($\lambda$-coalescent measure trees).
\newblock \emph{Probability Theory and Related Fields}, 145\penalty0
  (1):\penalty0 285--322, 2009.

\bibitem[Groeneboom(1983)]{Groeneboom1983a}
P.~Groeneboom.
\newblock {The concave majorant of Brownian motion}.
\newblock \emph{The Annals of Probability}, 11:\penalty0 1016--1027, 1983.

\bibitem[Hoeffding(1963)]{Hoeffding1963}
W.~Hoeffding.
\newblock Probability inequalities for sums of bounded random variables.
\newblock \emph{Journal of the American Statistical Association}, 58:\penalty0
  13--30, 1963.

\bibitem[Janson(2007)]{Janson2007}
S.~Janson.
\newblock Brownian excursion area, {W}right's constants in graph enumeration,
  and other {B}rownian areas.
\newblock \emph{Probability Surveys}, 4:\penalty0 80--145, 2007.

\bibitem[Karatzas and Shreve(1988)]{KaSh1988}
I.~Karatzas and S.E. Shreve.
\newblock \emph{{Brownian Motion and Stochastic Calculus}}.
\newblock Graduate Texts in Mathematics. Springer, New York, 1988.

\bibitem[Kennedy(1976)]{Kennedy1976}
D.P. Kennedy.
\newblock The distribution of the maximum {B}rownian excursion.
\newblock \emph{Journal of Applied Probability}, 13:\penalty0 371--376, 1976.

\bibitem[Knuth(1973)]{Knuth1973b}
D.~E. Knuth.
\newblock \emph{The Art of Computer Programming: Sorting and Searching},
  volume~3.
\newblock Addison-Wesley, Reading, MA, 1973.

\bibitem[Konheim and Weiss(1966)]{KoWe1966a}
Alan~G Konheim and Benjamin Weiss.
\newblock An occupancy discipline and applications.
\newblock \emph{SIAM Journal on Applied Mathematics}, 14\penalty0 (6):\penalty0
  1266--1274, 1966.

\bibitem[Kortchemski and Th{\'e}venin(2023)]{KoTh2023a}
Igor Kortchemski and Paul Th{\'e}venin.
\newblock {Coupling Bertoin's and Aldous-Pitman's representations of the
  additive coalescent}.
\newblock \emph{arXiv preprint arXiv:2301.01153}, 2023.

\bibitem[Kruskal(1956)]{Kruskal1956}
J.B. Kruskal.
\newblock On the shortest spanning subtree of a graph and the traveling
  salesman problem.
\newblock \emph{Proceedings of the American Mathematical Society}, 2:\penalty0
  48--50, 1956.

\bibitem[Le~Gall(1991)]{LeGall1991}
J.-F. Le~Gall.
\newblock Brownian excursions, trees and measure-valued branching processes.
\newblock \emph{The Annals of Probability}, 19:\penalty0 1399--1439, 1991.

\bibitem[Le~Gall(1993)]{Legall1993}
J.-F. Le~Gall.
\newblock {The uniform random tree in a Brownian excursion}.
\newblock \emph{Probability Theory and Related Fields}, 96:\penalty0 369--383,
  1993.

\bibitem[Le~Gall and Le~Jan(1998)]{LeLe1998a}
J.F. Le~Gall and Y.~Le~Jan.
\newblock {Branching processes in Levy processes: Laplace functionals of snakes
  and superprocesses}.
\newblock \emph{The Annals of Probability}, 26:\penalty0 1407--1432, 1998.

\bibitem[Marckert and Wang(2018)]{MaWa2018a}
Jean-Fran{\c c}ois Marckert and Minmin Wang.
\newblock A new combinatorial representation of the additive coalescent.
\newblock \emph{Random Structures \& Algorithms}, 54\penalty0 (2):\penalty0
  340--370, Apr 2018.
\newblock \doi{10.1002/rsa.20775}.

\bibitem[Martin and R{\'a}th(2017)]{MaRa2017a}
James~B. Martin and Bal{\'a}zs R{\'a}th.
\newblock Rigid representations of the multiplicative coalescent with linear
  deletion.
\newblock \emph{Electronic Journal of Probability}, 22\penalty0 (0), 2017.
\newblock \doi{10.1214/17-ejp100}.

\bibitem[Mattila(1999)]{Mattila1999a}
P.~Mattila.
\newblock \emph{{Geometry of sets and measures in Euclidean spaces: Fractals
  and rectifiability}}, volume~44 of \emph{Cambridge Studies in Advanced
  Mathematics}.
\newblock Cambridge Univ Press, 1999.

\bibitem[Mauldin and Williams(1986)]{MaWi1986a}
R.D. Mauldin and S.C. Williams.
\newblock {Random recursive constructions: Asymptotic Geometric and Topological
  Properties}.
\newblock \emph{Transactions of the American Mathematical Society},
  295:\penalty0 325--346, 1986.

\bibitem[Miermont(2003)]{Miermont2003}
G.~Miermont.
\newblock {Self-similar fragmentations derived from the stable tree I:
  Splitting at heights}.
\newblock \emph{Probability Theory and Related Fields}, 127:\penalty0 423--454,
  2003.

\bibitem[Miermont and Sen(2022)]{MiSe2022a}
Gr{\'e}gory Miermont and Sanchayan Sen.
\newblock On breadth-first constructions of scaling limits of random graphs and
  random unicellular maps.
\newblock \emph{Random Structures \& Algorithms}, 2022.

\bibitem[Nachmias and Peres(2008)]{NaPe2008}
A.~Nachmias and Y.~Peres.
\newblock Critical random graphs: diameter and mixing time.
\newblock \emph{The Annals of Probability}, 36:\penalty0 1267--1286, 2008.

\bibitem[Perkins(1981)]{Perkins1981a}
E.~Perkins.
\newblock {The exact Hausdorff measure of the level sets of Brownian motion}.
\newblock \emph{Zeitschrift f{\"u}r Wahrscheinlichkeitstheorie und verwandte
  Gebiete}, 58:\penalty0 373--388, 1981.

\bibitem[Pitman and Ross(2011)]{PiRo2011a}
Jim Pitman and Nathan Ross.
\newblock The greatest convex minorant of brownian motion, meander, and bridge.
\newblock \emph{Probability Theory and Related Fields}, 153\penalty0
  (3-4):\penalty0 771--807, Aug 2011.
\newblock \doi{10.1007/s00440-011-0385-0}.

\bibitem[Pitman(1983)]{Pitman1982a}
J.W. Pitman.
\newblock {Remarks on the convex minorant of Brownian motion}.
\newblock In E.~{\c C}inlar, K.L. Chung, and Getoor R.K., editors,
  \emph{Seminar on Stochastic Processes}, volume~5 of \emph{Progress in
  Probability and Statistics}, pages 219--227, Boston, 1983. Birkh{\"a}user.

\bibitem[Prim(1957)]{Prim1957}
R.~C. Prim.
\newblock Shortest connection networks and some generalizations.
\newblock \emph{Bell Syst. Tech. J.}, 36:\penalty0 1389--1401, 1957.

\bibitem[Rogers and Williams(2000)]{RoWi2000}
L.C.G. Rogers and D.~Williams.
\newblock \emph{Diffusions, Markov Processes and Martingales: It\^o Calculus},
  volume~2.
\newblock Cambridge University Press, Cambridge, UK, 2 edition, 2000.

\bibitem[Rossignol(2021)]{Rossignol2017a}
R.~Rossignol.
\newblock {Scaling limit of dynamical percolation on critical
  Erd\H{o}s--R\'enyi random graphs}.
\newblock \emph{The Annals of Probability}, 49:\penalty0 322--399, 2021.

\bibitem[Taylor and Wendel(1966)]{TaWe1966a}
S.J. Taylor and J.G. Wendel.
\newblock {The exact Hausdorff measure of the zero set of a stable process}.
\newblock \emph{Zeitschrift f{\"u}r Wahrscheinlichkeitstheorie und verwandte
  Gebiete}, 6:\penalty0 170--180, 1966.

\end{thebibliography}
}

\appendix

\section{Auxiliary technical results}
\label{sec:technical_results}
%!TEX root = MST_brownian.tex

\begin{lem}\label{lem:properties_Z}Let $\omega\in \cC(\R_+, \R)$. Then
\begin{compactenum}[i)]
  \item the process $(Z^\lambda(\omega))_{\lambda \in \R}$ non-increasing and right-continuous with left-limits;
	\item for every $\lambda \in \R$, $Z^\lambda(\omega)$ and $Z^{\lambda-}$ are both closed;
	\item the set $\{\lambda \in \R: Z^{\lambda-}(\omega)\setminus Z^\lambda(\omega) \ne \varnothing\}$ is countable.
\end{compactenum}
\end{lem}
\begin{proof}
\emph{i)} The monotony is a consequence of Lemma~\ref{lem:shear_composition}, this implies the existence of the left and right limits $\cap_{h>0} Z^{\lambda-h}$ and $\cup_{h>0} Z^{\lambda+h}$, respectively. The right-continuity follows by continuity of the maps $\lambda\mapsto \omega^\lambda$ and $\underline \omega^\lambda$: if $s\in Z^{\lambda-h}$ for all $h>0$, then $\omega(s)+(\lambda-h) s = \inf\{\omega(r)+(\lambda-h)r: 0\le r\le s\}$ for all $h>0$, and thus this also holds for $h=0$. \emph{ii)} The fact that $Z^\lambda$ is closed is an easy consequence of the continuity of $\omega$. The monotony shows that $Z^{\lambda-}$ is a decreasing limit of closed sets, and is thus closed. \emph{iii)} Since $Z^\lambda$ and $Z^{\lambda-}$ are both closed for every $\lambda\in\R$, if $\lambda$ is such that $Z^{\lambda-}\setminus Z^\lambda\ne \varnothing$, then there exists $\epsilon>0$ and $x=x_\lambda\in Z^{\lambda-}$ with $d(x, Z^\lambda)>\epsilon$.  It follows that 
\[\{\lambda \in \R: Z^{\lambda-}\setminus Z^{\lambda}\ne \varnothing\} = \bigcup_{n\ge 1} \{\lambda\in \R: \dH(Z^{\lambda-}, Z^\lambda)>1/n\}\,.\]
For each $n\ge 1$, there must exist for each $\lambda$ a ball of radius $1/n$, and the collection of these balls must be disjoint. 
For each $n\ge 1$, any collection of open balls of radius $1/n$ must be countable, and therefore any set in the right-hand side above is countable. The claim follows.
\end{proof}

\begin{lem}[Continuity properties of the metric $d$]\label{lem:continuity_distance}Let $d(\cdot,\cdot)$ be the pseudo-metric on $[0,1]$ defined from the pair $(\exc,\bU)$ used in the construction of $\CMT(\exc, \bU)$. Almost surely, 
\begin{compactenum}[i)]
    \item the map $d(0, \cdot)$ is continuous almost everywhere, but 
    \item for every $x\in \sL(\exc)\cap (0,1)$, the map $d(0,\cdot)$ is not left-continuous at $x$, and
    \item for every $x\in \sL(\exc)\cap (0,1)$, the map $d(0,\cdot)$ is neither left- nor right-continuous at $\ri(x)$.
\end{compactenum}
\end{lem}
\begin{proof}
\emph{i)} Let $x$ be uniformly random in $[0,1]$, then a.s.\ the vertices of the convex minorant of $\exc$ on $[0,x]$, $(t_i(x))_{i\ge 0}$ and the corresponding intercepts $(z_i(x))_{i\ge 0}$ are such that $t_i(x)< x < z_i(x)$. 
Furthermore, there exists a sequence of local minima $z_n>x$ with $z_n \downarrow x$ such that the vertices of the convex minorant of $\exc$ on $[0,z_n]$ are precisely $\cV_{z_n}=\cV_x \cup \{z_n\}$. 
Then, for any $i\ge 1$, $\sup\{d(x,t): t\in (t_i,z_i)\} \le D_i \cdot |t_i-z_i|^{1/2}$, where $(D_i)_{i\ge 1}$ are random variables distributed like the diameter of a continuum random tree of unit mass (which are not independent). It follows that 
\begin{align*}
\pc{\sup\{d(x,t): t\in (t_i,z_i)\} > |z_i-t_i|^{1/4}} 
&\le \pc{D_i> |z_i-t_i|^{-1/4}} \\
&\le \exp(-|z_i-t_i|^{-1/4}/2v)\,,
\end{align*}
for some constant $v>0$. It follows by the Borel--Cantelli Lemma that a.s.\ $\sup\{d(x,t): t\in (t_i,z_i)\}\le |z_i-t_i|^{1/4}$ for all but finitely many values of $i$, so that $|d(0,t)-d(0,x)|\le d(x,t) \to 0$ as $t\to x$. 

\emph{ii)} 
Let $x\in \sL\cap (0,1)$; then a.s. there are only finitely many vertices in $\cV_x$, and $x=t_i(x)$ for some $i\ge 1$. Let $z_n$ be a sequence of local minima with $z_n \in (t_{i-1},t_i)$ and $z_n\uparrow t_i$ as $n\to\infty$. Then, for all $n_0$ large enough, the vertices of the convex minorant of $\exc$ on $[0,z_n]$ are exactly $\{t_j, j<i\} \cup \{z_n\}$. For each $n\ge n_0$, the point $\ju(z_n)$ is uniform in $(t_{i-1}, z_n)$ and $\ju(x)$ is uniform in $(t_{i-1},t_i)$. With probability one, there exists a subsequence $(n_j)_{j\ge 1}$ such that $0<\ju(z_{n_j})-t_{i-1} < \tfrac 12 (\ju(x)-t_{i-1})$. In particular, since $\llb 0,x\rrb$ a.s.\ has an accumulation point at $\ju(x)$, we have $\sup d(0,z_{n_j}) = \sup d(0,\ju(z_{n_j})) < d(0,\ju(x)) = d(0,x)$. It follows that, for any $\epsilon>0$, $\inf\{d(0,s): s\in(x-\epsilon, x)\} < d(0,x)$.   

\emph{iii)} For $x\in \sL \cap (0,1)$, the point $\ri(x)$ is some intercept, and the proof that $d(0,\cdot)$ is not left-continuous at $\ri(x)$ is the same as in \emph{ii)}. For the lack of right-continuity at $\ri(x)$, this is also similar, but relies on the fact that one may find a sequence of local minima $z_n$ in $(\ri(x),1)$ with $z_n \downarrow \ri(x)$ such that $\cV_{z_n}=\cV_x \cup \{z_n\}$. The same argument as above can then be used by considering the random points $\ju(z_n)$, which are independent, and uniform in $[x,z_n]$.
\end{proof}

\begin{lem}[Surplus and area under the curve]\label{lem:area} Let $\exc$ be a Brownian excursion. 
Consider the subset $D$ of $[0,1]^2\times \R$ of points $(x,y,\lambda)$ such that $[x,y]\cap Z^\lambda(\exc)= \varnothing$. Then, the 3-dimensional volume of $D$ is equal to $\int_0^1 \exc(x) dx$.
\end{lem}
%Eventuellement dire que l'argument n'utilise pas la loi de ${\sf e}$, mais des propriétés

\begin{proof}Recall the recursive decomposition of  Section~\ref{sub:distribution_CMT}. Then, the set $D$ can be decomposed into countably many portions (with disjoint interior) $D_u$, $u\in \cU$, as follows: 
\[D_u:= \{(x,y,\lambda): a_u\le x<y< R_{-\lambda}(a_u), \lambda\le -\tau_u\}\,.\]
There is a corresponding decomposition of the set $\{(s,t): s\in [0,1], 0\le t\le \exc(s)\}$ also into portions with disjoint interior, $E_u=\{(s,\exc(a_u)-\lambda (s-a_u)): a_u\le s\le R_{-\lambda}(a_u), -\lambda\ge \tau_u\}$, for $u\in \cU$. We show that, for each $u\in \cU$, 
\[\int_{E_u} dsdt = \int_{D_u} dxdyd\lambda\,.\]
We treat the case $u=\varnothing$, the others are just the same, up to the more complicated notation. First observe that the left-hand side above with $u=\varnothing$ is precisely the area under the function $f$ given by, for $i\ge 1$, 
\[f(s)=\exc(a_i)+ \tau_i (s-a_i)\qquad a_i=R_{\tau_i}\le s<R_{\tau_i-}\,.\]
Now, since each point $z=(s, f(s))$ can be represented in polar coordinates as $z=\rho(\theta) e^{i\theta}$, or alternatively by the pair $(-\lambda, R_{-\lambda}(0))$, where $-\lambda=f(s)/s$ is the slope of the line from $0$ to $z$, we have 
\[\int_{E_\varnothing} dsdt = \int_0^1 f(s)ds = \frac 1 2 \int_0^{\pi/2} \rho(\theta)^2 d\theta = \int_{D_\varnothing} dxdyd\lambda\,.\]
The claim follows by summing the contributions for $u\in \cU$. 
\end{proof}

\end{document}